\documentclass[11pt]{amsart}
\usepackage{amsmath, amssymb, float, graphicx, microtype, amsthm, caption, subcaption,booktabs,enumerate,wrapfig,mathrsfs,bbm, fancyhdr,tikz-cd,mathtools}
\usepackage[top=1in,bottom=1in,left=1in,right=1in]{geometry}

\usepackage{hyperref} %%%% Adds index to generated pdf %%%%

\newtheorem{theorem}{Theorem}[section]

\newtheorem{claim}[theorem]{Claim}
\newtheorem{prop}[theorem]{Proposition}
\newtheorem{lemma}[theorem]{Lemma}
\newtheorem{cor}[theorem]{Corollary}
\theoremstyle{definition}

\newtheorem{definition}[theorem]{Definition}
\newtheorem{question}[theorem]{Question}

\theoremstyle{remark}
\newtheorem{remark}[theorem]{Remark}

\numberwithin{equation}{section} %%% Number equations as Section.Equation  e.g. first equation in section 1 is labeled as (1.1) %%%

\newcommand{\R}{{\mathbb R}}
\newcommand{\N}{{\mathbb N}}
\newcommand{\Z}{{\mathbb Z}}

\newcommand{\1}{{\mathbbm 1}}

\newcommand{\e}{{\varepsilon}}

\renewcommand{\sl}{\mathfrak{sl}}

\DeclareMathOperator{\im}{Im}
\DeclareMathOperator{\tr}{Trace}
\newcommand{\id}{{\text{id}}}

\DeclareMathOperator{\supp}{supp}

\DeclareMathOperator{\vol}{vol}

\DeclareMathOperator{\Aut}{Aut}

\DeclareMathOperator{\graph}{graph}

\DeclareMathOperator{\dist}{dist}

\DeclareMathOperator{\Ad}{Ad}

\DeclareMathOperator{\vecspan}{span}

\DeclareMathOperator{\LyapReg}{LyapReg}
\DeclareMathOperator{\Lip}{Lip}

\DeclareMathOperator{\diag}{diag}
\DeclareMathOperator{\card}{card}
\DeclareMathOperator{\diam}{diam}

\title{The Bernoulli property for $PSL_2\R$ Skew products over exponentially mixing base}
\author{Meg Doucette}

\date{}

\begin{document}
\maketitle

\begin{abstract}
	We study the Bernoulli property for skew products of the form $f(x,y)=(g(x),A(x)y)$ on $M\times PSL_2(\mathbb{R})/\Gamma$, where $g$ is a smooth, measure-preserving, exponentially mixing diffeomorphism, $\Gamma<PSL_2(\mathbb{R})$ is a cocompact lattice, and $A\colon M\to \operatorname{PSL}_2(\mathbb{R})$ is a $C^{1+\alpha}$ cocycle. We prove that if $A$ has a nonzero Lyapunov exponent, then $f$ is Bernoulli. The proof does not require that $f$ be exponentially mixing. This provides a new method for establishing the Bernoulli property, which applies to a large class of skew products. 
\end{abstract}

%\tableofcontents

\section{Introduction}

	Chaotic properties of dynamical systems have been the subject of extensive study. The strongest qualitative chaotic property a dynamical system can exhibit is the \emph{Bernoulli Property}. A dynamical system $f:(M,\mu)\to (M,\mu)$ is \emph{Bernoulli} if it is isomorphic (as a measure preserving system) to a \textit{Bernoulli Shift}. A Bernoulli Shift is the shift $\sigma:\left( \{1,...,n\}^{\Z}, \mathbf{p}^\Z \right) \to \left( \{1,...,n\}^{\Z}, \mathbf{p}^\Z \right)$ where $\mathbf{p}=(p_1,...,p_n)$ is a probability vector. 
	
	In this paper we will focus on the skew-products $f:(M\times N, \mu\times \nu) \to (M\times N, \mu\times \nu)$ of the form
		\begin{equation}\label{eq:skew-product-f_intro}
			f(x,y)=(g(x),A(x)y)
		\end{equation}
	where $g:(M,\mu)\to (M,\mu)$ is a $C^{1+\alpha_g}$ diffeomorphism of a closed manifold $M$ that preserves a smooth measure $\mu$, $N=PSL_2\R/\Gamma$, where $\Gamma<PSL_2\R$ is a cocompact lattice, $\nu$ is normalized Haar measure on $N$, and $A:M\to PSL_2\R$ is a $C^{1+\alpha}$ map with $\log\|A\|, \log\|A^{-1}\| \in L^1(\mu)$. 
	
	We say that $g$ is \emph{exponentially mixing} if there exists $r>0, \ C>0,$ and $\eta>0$ such that for any $\phi,\psi\in C^r(M)$, 
		\begin{equation*}
			\left| 
				\int_M \phi(x)\psi(g^nx)d\mu - \int_M \phi d\mu \int_M\psi d\mu
			 \right|
			 \leq 
			 Ce^{-\eta n} \|\phi\|_{C^r}\|\psi\|_{C^r}
		\end{equation*}
	
	Our main result is the following:
	\begin{theorem}\label{thm:Main_f_bernoulli}
		If $g$ is exponentially mixing and $A$ has a nonzero Lyapunov exponent, then the skew-product $f$ given in \eqref{eq:skew-product-f_intro} is Bernoulli. 
	\end{theorem}

\subsection{Historical Background}

	The main strategy for proving that a smooth system is Bernoulli originates with Ornstein and Weiss \cite{OrnsteinWeiss}. 	Their method revolves around constructing a coupling between nearby pieces of unstable manifolds that will mostly keep the pieces close together for a sufficiently long time. Using this method, several classical dynamical systems have been proved to be Bernoulli. This includes Anosov diffeomorphisms \cite{BowenEquilibrium, BowenBernoulli} and flows \cite{RatnerBernoulli}, isometric extensions of these systems \cite{DolgopyatCptGrpExtensions,RudolphIsometric}, and partially hyperbolic homogeneous systems \cite{Dani, KanigowskiBernoulliHomogeneous}. 
	
	An important extension of the Anosov property is \textit{nonuniform hyperbolicity}. If $f$ is a $C^{1+\alpha}$ diffeomorphism that preserves an absolutely continuous measure, then $f$ is measure theoretically isomorphic to the product of a Bernoulli shift and a finite rotation (\cite[Theorem 9.5.11]{BarreiraPesinNonUniform}, \cite{PesinLyapunovExponents}). Thus in this case, if we additionally assume that $f$ is weak-mixing, then $f$ is Bernoulli. 

	Another natural extension of the Anosov property is \textit{partial hyperbolicity}. Statistical properties of partially hyperbolic systems have been the subject of extensive studies. The strongest result in that direction is due to \cite{BW_KProp} which shows that volume-preserving partially hyperbolic diffeomorphisms which are center bunched and essentially accessible have the K-property. The {\em K-property} means that every measure preserving factor has positive entropy. It is a strong stochastic property, one step below Bernoullicity.

	The mechanism behind the K-property in these systems is called \textit{accessibility}. The \textit{accessibility class} of a point $x$ is the set of points which can be joined to $x$ by a piecewise smooth curve so that each piece is tangent to either $E^s$ or $E^u$. A system is \textit{accessible} if there is only one accessibility class. On the opposite end of the spectrum are direct products, in which the accessibility classes are compact submanifolds that form a horizontal foliation.  Thus accessibility provides a way for chaotic behavior coming from hyperbolicity to proliferate to the central direction. A system is {\em essentially accessible} if every set which  is a union of accessibility classes has either zero or full volume. According to a conjecture of Pugh and Shub (essential) accessibility holds for an open and dense set of partially hyperbolic systems. Thus we expect most volume preserving partially hyperbolic systems to be K.

	On the other hand, much less is known about circumstances under which diffeomorphisms with non-isometric center direction(s), are Bernoulli.\footnote{If the center direction(s) are isometric, then Ornstein and Weiss's method usually applies.} 
	There are a number of examples of partially hyperbolic diffeomorphisms that have the K-property, but are not Bernoulli \cite{FlexiblityCLT,Dim4KBernoulli,Katok80,RudolphBrownian}. There are also limited examples of partially hyperbolic diffeomorphisms that are Bernoulli \cite{ExpMixingBernoulli,SkewProductsBernoulli,RudolphIsometric}.  However, the features of partially hyperbolic diffeomorphisms responsible for the Bernoulli property are still poorly understood, even in relatively simple cases. 
	
	Consider the skew product $f:X\times Y\to X\times Y$ of the form
		\begin{equation}\label{eq:general_skew_product}
			f(x,y)=(g(x),\tau(x)y)
		\end{equation}
	where $g:X\to X$ is a diffeomorphism preserving a measure $\mu_X$ on $X$, and $\tau:X\to G$ is a smooth function, where $G$ is a Lie group that acts smoothly on the left on $Y$ and preserves a smooth measure $\mu_Y$ on $Y$. The examples of partially hyperbolic diffeomorphisms that are $K$ but not Bernoulli are mentioned in the previous paragraph are all of the form \eqref{eq:general_skew_product}. Previous studies of skew-products focused on the case where $G$ is either abelian (see \cite{LP12, KanICM}) or compact (see \cite{Rob88} and references therein). 
		
	The question of when the skew product given by \eqref{eq:general_skew_product} is Bernoulli is of great interest. For $f$ to be Bernoulli, $g$ must be Bernoulli \cite{OrnsteinFactorsBernoulli}. This leads us the following question:
	
	\begin{question}\label{question:skew_prod_over_bernoulli}
		When is a skew product over a map with the Bernoulli property Bernoulli?
	\end{question}
	
	By the earlier discussion, Question \ref{question:skew_prod_over_bernoulli} is primarily of interest in the case where the skew product has at least one zero Lyapunov exponent. To illustrate the subtlety of the Bernoulli property in this case, consider a skew product	
		$f: (X \times SL_d(\R)/\Gamma,\mu) \to (X \times SL_d(\R)/\Gamma,\mu)$ 
	of the form 
		\begin{equation*}
			f(x,y)= (g(x), A(x)y),
		\end{equation*}
 	where $\Gamma$ is a cocompact lattice in $SL_d(\R)$, $g:X\to X$ is a volume preserving Anosov diffeomorphism, and $A:X\to SL_d(\R)$ is a smooth map. Note that $f$ preserves the measure $\mu=\text{vol}_X\times \text{Haar}$. Also note that $f$ has at least one zero Lyapunov exponent. We now summarize some existing results.

	\begin{theorem}\label{thm:BackgroundSkew}
	\begin{enumerate}[(a)]
		\item (\cite{ExpMixingBernoulli})
			If $A(x)=\diag\left( e^{\tau_1(x)}, e^{\tau_2(x)}, \dots, e^{\tau_d(x)} \right)$ and  $\exists i: \int \tau_id\mu\neq 0$ then $f$ is Bernoulli;
		
		\item (\cite{FlexiblityCLT}, \cite{RudolphBrownian})
			If $A(x)=\diag\left( e^{\tau_1(x)}, e^{\tau_2(x)}, \dots, e^{\tau_d(x)} \right)$ and  $\forall i: \int \tau_id\mu=0$ then $f$ is not Bernoulli;
		
		\item (\cite{ExpMixingBernoulli}) 
			If $A$ is close to the identity and pinching and twisting in the sense of \cite{AviliaVianaSimplicity}, then $f$ is Bernoulli.
%			If $d=2$ and $A(x)$ is close to the identity then $f$ is Bernoulli unless either $A$ preserves a Riemannian metric on $SL_2(\R)$ or it is smoothly conjugated to
%			\begin{equation*}
%				\begin{pmatrix} e^{\alpha(x)} & \beta(x)\\ 0 & e^{-\alpha(x)}\end{pmatrix} \quad \text{with} \quad \int \alpha d\mu=0.
%			\end{equation*}
	\end{enumerate}
	\end{theorem}

	The systems in (a) and (c) are exponentially mixing (\cite{mixing_generalized_TTinverse}, \cite{gouezel_quantitative_2019}), and so parts (a) and (c) of Theorem \ref{thm:BackgroundSkew} are a consequence of the following more general result:
	\begin{theorem}[\cite{ExpMixingBernoulli}]\label{thm:expmixing}
		Exponentially mixing volume-preserving systems are Bernoulli.
	\end{theorem}
	
	The proof of Theorem \ref{thm:expmixing} uses the strategy developed by Ornstein and Weiss \cite{OrnsteinWeiss} of constructing a coupling of nearby unstable manifolds that stay close for a sufficiently long time. Constructing this coupling is problematic for systems (like exponentially mixing systems) that may have center directions with expansion because those directions pull nearby unstables apart. The key idea in \cite{ExpMixingBernoulli} is that we can use exponential mixing to force unstable pieces exponentially close, which counters the sub-exponential expansion in the center direction. 
	
	This illustrates a general theme we see in examples of Bernoulli systems with zero Lyapunov exponents: the speed of mixing needs to be sufficient to counter growth in the directions with zero Lyapunov exponents. In \cite{FlexiblityCLT}, Dolgopyat, Dong, Kanigowski, and Nandori give examples that are not Bernoulli despite mixing at an arbitrarily fast polynomial rate. On the other hand, if the center is isometric, weak mixing skew products are always Bernoulli \cite{RudolphIsometric}, even though their mixing rate may be arbitrarily slow \cite{DolgopyatCptGrpExtensions}. Faster mixing is unnecessary in this case because there is no growth in the center direction. 	

	Theorem \ref{thm:Main_f_bernoulli} extends the results of Theorem \ref{thm:BackgroundSkew}. Unlike in Theorem \ref{thm:BackgroundSkew}(a),(c), we allow the base map $g$ to be exponentially mixing, rather than using the stronger requirement that $g$ is Anosov. If $A$ is neither diagonal nor close to the identity and pinching and twisting in the sense of \cite{AviliaVianaSimplicity}, it is unknown whether the skew product $f$ that we consider in Theorem \ref{thm:Main_f_bernoulli} is exponentially mixing. This allows us to construct several new examples of Bernoulli skew products:
	
	\begin{cor}\label{cor:new_examples_bernoulli_skew}
		Suppose that $g$ is either an Anosov diffeomorphism or a partially hyperbolic translation on $G/\Delta$, where $G$ is a simple Lie group and $\Delta<G$ is a uniform lattice. Then the skew product \eqref{eq:general_skew_product} is Bernoulli except for a meagre codimension infinity subset of cocycles. 
	\end{cor}
	
	For an Anosov base, the Lyapunov exponents of a generic cocycle are nonzero by \cite[Theorem A]{viana_almost_2008}. For a partially hyperbolic translation on $G/\Delta$, the base is center bunched and accessible by \cite[Theorem B, Corollary b]{PS-JEMS}, and thus the Lyapunov exponents of a generic fiber bunched cocycle are nonzero \cite[Theorem A]{ASMV}. In both cases the exceptional set has infinite codimension. See also \cite{AviliaVianaSimplicity} for earlier results about nonzero Lyapunov exponents.

\subsection{Novelties in this work and future outlook}

	Regardless of the dynamics of the base map $g$, the skew-product $f$ we consider in \eqref{eq:skew-product-f_intro} will have a nontrivial center direction. Thus, Theorem \ref{thm:Main_f_bernoulli} provides new examples of smooth Bernoulli systems with nontrivial and non-isometric center directions. 
	
	The proof uses the framework of \cite{ExpMixingBernoulli}, i.e. by forcing unstable pieces exponentially close, we can counter sub-exponential expansion in the center direction to construct a coupling of nearby unstables that stays close for sufficiently long time. The key point of departure between the proof of Theorem \ref{thm:Main_f_bernoulli} and Theorem \ref{thm:expmixing} is that because $f$ is not assumed to be exponentially mixing, we need another mechanism to force unstable pieces exponentially close. Once we do that, constructing the coupling between unstable pieces is analogous to that given in \cite{ExpMixingBernoulli}. 
	
	To force unstable pieces exponentially close, we show that we can analyze the dynamics of $f$ on the base and in the fibers separately. This allows us to see that the rate of equidistribution of the parts of our unstables in the base and in the fiber are ``independent'' of each other. We show that the fiber part of our unstables equidistribute on exponentially small cubes (Proposition \ref{prop:exp_equidistribution_fiber}), regardless of the rate of equidistribution of unstables in the base. 
	
	There is still a large gap between the mixing rate which is sufficient to guarantee the Bernoulli property and the complexity of the dynamics in the center direction(s) which is sufficient to conclude that a system is not Bernoulli. This necessitates developing new methods for both proving and disproving the Bernoulli property. In the present work, we show that for skew products with a nontrivial, non-isometric center direction, exponential mixing of the full system is not necessary for Bernoullicity: exponential mixing of the base along with suitable fiberwise hyperbolicity suffices. This condition can be viewed as a ``non-uniform'' version of exponential mixing. In this paper, we analyzed $PSL_2\R$ fibers, but we expect many more applications of our method. We also note that the fiberwise equidistribution established in this paper does not rely on the base map being exponentially mixing, and thus has application to skew-products of the form \eqref{eq:skew-product-f_intro} with base dynamics that may not be exponentially mixing. These extensions will be addressed in future work. 
	
	There are many other interesting questions about ergodic properties of skew products with semisimple fibers, and our result is just one step in this direction. 
	
\subsection{Outline of Paper}

	In Section \ref{sec:basedynamics}, we give the definitions and results about the base dynamics that will be used throughout the paper. Section \ref{sec:geometry_fiber} contains preliminaries on the geometry of the fiber $N=PSL_2\R/\Gamma$, including local product coordinates and the definitions of the fiber cubes that will be used in the equidistribution argument in Section \ref{sec:equidistribution_f}. 
	
	We state and prove the results we will need about the fiber dynamics in Section \ref{sec:fiberdynamics}. A key result in this section is the Osceledets-Pesin reduction of $A$ (Theorem \ref{thm:Osceledets_Pesin_Reduction}), which allows us to conjugate $A$ to a diagonal cocycle $B(x)=P(gx)^{-1}A(x)P(x)$. We also obtain quantitative regularity estimates for both $B$ and for the conjugating map $P$. In Section \ref{sec:ergodicity_f}, we use this reduction to prove that the skew product $f$ is ergodic. 
	
	Sections \ref{sec:diagonal_skew_product} and \ref{sec:unstable_f} describe the unstable geometry of the skew product $f$. In Section \ref{sec:diagonal_skew_product}, we introduce the diagonal skew product 
		$
			f_{diag}(x,y)=(gx,B(x)y),
		$
	for which the unstable manifolds and their conditional measures can be described explicitly. In Section \ref{sec:unstable_f}, we transfer this description through the conjugacy by $P$ to obtain the unstable manifolds and conditional measures for the original skew product $f$. 
	
	In Section \ref{sec:equidistribution_f}, we state and prove the main ingredient in the proof of Theorem \ref{thm:Main_f_bernoulli}, that unstable manifolds for $f$ equidistribute on exponentially small product cubes. We first prove that the fiber components of unstables equidistribute on exponentially small cubes in $N$. Combining this with exponential equidistribution of unstable manifolds in the base yields equidistribution for $f$. 
	
	Finally, we prove Theorem \ref{thm:Main_f_bernoulli} in Section \ref{sec:proof_main_thm}. Following the coupling approach of \cite{OrnsteinWeiss} and \cite{ExpMixingBernoulli}, we use the equidistribution result from Section \ref{sec:equidistribution_f} to couple typical unstable pieces so that corresponding orbits remain close. This verifies the very weak Bernoulli property, and thus proves that $f$ is Bernoulli. 
			
\subsection*{Acknowledgments}
	The author would like to thank Adam Kanigowski and Dmitry Dolgopyat for many helpful discussions and comments on the draft of this paper. The author would also like to thank Jon Dewitt for a number of helpful discussions. The author also thanks Amie Wilkinson for her helpful comments on earlier drafts of this paper.

\section{The base dynamics}\label{sec:basedynamics}

	This section summarizes the facts about $g$ that we will need later. Since $g$ is exponentially mixing, it is ergodic and has a positive Lyapunov exponent \cite[Prop. 1.4]{ExpMixingBernoulli}. Let $\lambda_g>0$ be the smallest positive Lyapunov exponent of $g$. 
	
	This and Theorem \ref{thm:exp_equidistribution_g_expmixing} are the only places we use the fact that $g$ is exponentially mixing in this section. All the other results we state in this section hold if $g$ is a $C^{1+\alpha_g}$ diffeomorphism that preserves a smooth measure, $g$ is ergodic with respect to $\mu$, and $g$ has at least one nonzero Lyapunov exponent. Fix $\delta_g>0$ sufficiently small, and let $\LyapReg_g=\LyapReg_g(\delta_g)$ denote the corresponding full-measure set of Lyapunov regular points for $g$ \cite[Definition 2.1]{ExpMixingBernoulli}.
	
	For $x\in \LyapReg_g$, let $h_x$ be the Lyapunov chart from \cite[Lemma 2.4]{ExpMixingBernoulli}, and let $\mathfrak{r}_{\delta_g}(x)$ denote its size. For $\tau_g>0$, set 
		$$
			\hat{\mathcal{P}}_{\tau_g}=\hat{\mathcal{P}}_{\tau_g}^{\delta_g}:= \left\{ x\in M \colon \mathfrak{r}_{\delta_g}(x)\geq \tau_g \right\}.
		$$
	For $x\in \hat{\mathcal{P}}_{\tau_g}$, 
	
		\begin{equation}\label{eq:lipschitz_constant_lyapunov_charts_g1}
			\Lip(h_x), \Lip(h_x^{-1}) \leq \mathfrak{r}_{\delta_g}(x)^{-1}
		\end{equation}
		
	We always take $\delta_g,\tau_g\leq 1$. 
	
\subsection{Unstable manifolds and conditional measures}
	For $x\in \LyapReg_g$, let $W^u_g(x)$ be its local unstable manifold, and for $0<\xi\leq \mathfrak{r}_{\delta_g(x)}$, we let $W^u_{g,\xi}(x)$ denote the unstable plaque of size $\xi$ in Lyapunov coordinates:
	
	\begin{lemma}[{\cite[Lemma 2.8]{ExpMixingBernoulli}}]\label{lem:unstables_g}
		There are constants $C_{0,g}>0$ and $\alpha_{g,4}\in(0,1)$ such that, writing $\R^{\dim(M)}=\R^u\oplus \R^{cs}$ for every $x\in \LyapReg_g$, there is a $C^{1+\alpha_{g,4}}$ map $\eta_x^u:\R^u\to \R^{cs}$ with 
			\begin{equation*}
				\|\eta\|_{C^{1+\alpha_{g,4}}}\leq C_{0,g}, \ \eta^u_x(0)=0, \ D_0\eta^u_x=0
			\end{equation*}
		such that for $0<\xi\leq r_{\delta_g}(x)$
			\begin{equation*}
				h_x^{-1}(W^u_{g,\xi}(x))=\graph(\eta^u_x,B^u(\xi))
			\end{equation*}
	\end{lemma}
	
	By \cite[(2.13)]{ExpMixingBernoulli}, there exists a measurable function $r_{u,g}:\LyapReg_g\to (0,\infty)$  such that $r_{u,g}(x)\leq \mathfrak{r}_{\delta_g}(x)$ and for $x'\in W^u_{g}(x)$, $n\geq 1$
		\begin{equation}\label{eq:g_unstable_exp_contracting1}
			d(g^{-n}(x),g^{-n}x')
			\leq \frac{1}{r_{u,g}(x)}e^{(\sqrt{\delta_g}-\lambda_g)n}d(x,x')
		\end{equation}
	Fix $\delta_g>0$ sufficiently small for the earlier Pesin theory results, and assume in addition that 
		$$
			\delta_g<\frac{\lambda_g^2}{4}. 
		$$
	For $\tau_g>0$, we choose compact nested sets 
		$$
			\mathcal{P}_{\tau_g}=\mathcal{P}_{\tau_g}^{\delta_g}\subset \left\{ x\in \hat{\mathcal{P}}_{\tau_g} \colon r_{u,g}\geq \tau_g  \right\},
			\qquad \tau_g>0
		$$
	whose union has full measure. Then, for $x\in \mathcal{P}_{\tau_g}$, $x'\in W^u_{g}(x)$, and $n\geq 1$, \eqref{eq:g_unstable_exp_contracting1} gives that 
		\begin{equation}\label{eq:g_unstable_exp_contracting2}
			d(g^{-n}(x),g^{-n}x')
			\leq \tau_g^{-1} e^{ -\frac{\lambda_g}{2}n}d(x,x'),
		\end{equation}
	Given $\e>0$, we can choose $\tau_g>0$ sufficiently small that $\mu(\mathcal{P}_{\tau_g})\geq 1-\e^{10^{10}}$.
	
	For each $x\in \LyapReg_g$ and $x'\in W^u_g(x)$, we define
		$$
			\rho(x,x') := \prod_{j\geq 1}\frac{\det\left( D_{g^{-j}(x)}g \vert E_g^u\left(g^{-j}(x) \right)\right)}{\det\left( D_{g^{-j}(x')}g \vert E_g^u\left(g^{-j}(x') \right)\right)}, 
			\qquad 
			d\mu_x^u(x'):= \rho(x,x')d\bar{\mu}^u_x(x'),
		$$
	where $\bar{\mu}^u_x$ is Riemannian volume on $W^u_g(x)$. By \cite[Lemma 2.12]{ExpMixingBernoulli}, $\rho$ is H\"older in $x\in \mathcal{P}_{\tau_g}$ and in $x'\in W^u_g(x)$, $\rho(x,x)=1$,  and 
		\begin{equation}\label{eq:base_rho_switch_basepoint}
			\rho(x,x')\rho(x',x'')=\rho(x,x'').
		\end{equation}
	This means that we can make $\rho(x,x')$ arbitrarily close to 1 by taking $x'$ sufficiently close to $x$. Additionally, if $x_1,x_2$ lie on the same unstable leaf, then  $\mu^u_{x_1}(\mathcal{W})=\rho(x_1,x_2)\mu^u_{x_2}(\mathcal{W})$: in other words, changing the base point of an unstable leaf only changes $\mu^u_x$ by a constant factor. 
	
	We now give the local disintegration along $W^u_g$. For $x\in \hat{\mathcal{P}}_{\tau_g}$, take a transversal $T_x$ and let 
		\begin{equation*}
			Q_x:=h_x(B_\infty(\tau_g)), \qquad 
			\tilde T^{\tau_g}_x := \bigcup_{y\in \hat{\mathcal{P}}_{\tau_g}\cap Q_x} T_x\cap W^u_{g,\tau_g}(y), \qquad
			\tilde Q^{\tau_g}_x := \bigcup_{y\in \hat{\mathcal{P}}_{\tau_g}\cap Q_x}  W^u_{g,\tau_g}(y)
		\end{equation*}
	as in \cite[Section 2.4]{ExpMixingBernoulli}. For $z\in \tilde T^{\tau_g}_x$, let $W^u_{Q_x}(z)$ be the connected unstable plaque through $z$ that is obtained by chaining together connected components of $W^u_{g,\tau_g}(y)\cap Q_x$ for $y\in \hat{\mathcal{P}}_{\tau_g}\cap Q_x$. 
		\begin{lemma}[{\cite[Lemma 2.13]{ExpMixingBernoulli}}]\label{lem:disintegration_mu}
			For $x\in \hat{\mathcal{P}}_{\tau_g}$ there exists an absolutely continuous (with respect to Lebesgue) measure $\mu^{cs}_x$ such that if $h\in L^1(\tilde Q^{\tau_g}_x)$, 
				$$
					\int_{Q_x\cap \tilde Q^{\tau_g}_x} h(y)d\mu(y) =\int_{\tilde T^{\tau_g}_x} \int_{W^u_{Q_x}(z)} h(y)d\mu^u_z(y)d\mu^{cs}_x(z)
				$$
		\end{lemma}
		
	The following lemma shows that we can cover our Pesin sets by finitely many of these foliation boxes $Q_x\cap \tilde Q^{\tau_g}_x$:

	\begin{lemma}\label{lem:foliation_boxes_covering_P_tau}
		Fix $\tau_g>0$ and $r_0\in (0,\tau_g^2)$. There exists a finite set $\{x_j\}_{j=1}^m\subset \mathcal{P}_{\tau_g}$ such that 
			$$
				\mathcal{P}_{\tau_g}\subset \bigcup_{j=1}^m B(x_j,r_0)
				\qquad \text{ and } \qquad
				\mathcal{P}_{\tau_g}\cap B(x_j,r_0)\subset \tilde{Q}_{x_j}^{\tau_g}\cap Q_{x_j}
			$$
	\end{lemma}
	\begin{proof}
		Compactness of $\mathcal{P}_{\tau_g}$ gives the finite cover. If $y\in B(x_j,r_0)$, then by \eqref{eq:lipschitz_constant_lyapunov_charts_g1},
			$$
				\|h_{x_j}^{-1}(y)\|
				\leq \Lip(h_{x_j}^{-1})d(y,x_j)
				<\tau_g
			$$
		So, $y\in Q_{x_j}$. If in addition, $y\in \mathcal{P}_{\tau_g}\subset \hat{\mathcal{P}}_{\tau_g}$, then $y\in W^u_{g,\tau_g}(y)\subset \tilde{Q}_{x_j}^{\tau_g}$. 
	\end{proof}
	
\subsection{Exponential equidistribution of unstables for $g$}

	\begin{theorem}[{\cite[Proposition 7.3]{ExpMixingBernoulli}}]\label{thm:exp_equidistribution_g_expmixing}
		Assume that $g$ is exponentially mixing. There exists $\eta_g>0$ such that, for all $\e>0$ there exists $\xi_{0,g}>0$ such that for all $\xi\in (0,\xi_{0,g})$, there exists $\bar{n}=n(\e,\xi)\in \N$ such that for all $n\geq \bar{n}$, there is a set $\mathcal{R}_n\subset M$ with $\mu(\mathcal{R}_n)\geq 1-300\e^{10^{10}/16}$ and a set $K_n\subset M$ with $\mu(K_n)\geq 1-\e^{10}$ satisfying the following:
		
		Let $\{B_i\}_{i\in I_1}$ be a family of disjoint sets in $M$ such that either
			\begin{itemize}
			\item cubes of unstable radius $e^{\e^2n-\eta_g\e n}$ and center stable radius $e^{-\eta_g \e n}$ and $\mu\left( \bigcup_{i\in I_1} B_i \right)\geq 1-100\e^{10^{10}/16}$, or 
			\item Balls of radius between $\left[ e^{-\eta_g \e n}, 1\right)$ and $\mu\left( \bigcup_{i\in I_1} B_i \right)>1-\e^{10^5}$
			\end{itemize}
		Then, for all $x\in K_n$ and for every unstable box $\mathcal{W}$ of radius $\xi$ containing $x$, there exists a subset $J_g(x)\subset I_1$ such that $\mu\left( \bigcup_{i\in J_g(x)} B_i \right)>1-\e^{10}$ and for every $i\in J_g(x)$, 
			\begin{equation*}
				\mu^u_{x}\left( \mathcal{W} \cap g^{-\e n}(B_i)\cap \mathcal{R}_n \right)\in (1-\e^{10},1+\e^{10})\mu^u_x\left( \mathcal{W} \right) \mu(B_i)
			\end{equation*}
	\end{theorem}
	
	In Theorem \ref{thm:exp_equidistribution_g_expmixing}, when we write $g^{-\e n}$, we mean $g^{-\lfloor\e n\rfloor}$. For the sake of brevity, we will use this convention throughout the paper.

\section{Geometry of the Fiber N}\label{sec:geometry_fiber}

	Throughout, $N=PSL_2\R/\Gamma$, where $\Gamma<PSL_2\R$ is a cocompact lattice. This section gives the notation and preliminaries about $N$ that we will use later. It also defines cubes in the fiber. 
	
\subsection{$SL_2\R$ and its Lie algebra $\sl_2\R$}\label{sec:lie_theory_sl_2R}

	Recall that $SL_2(\R)=\{B\in M_{2\times 2}(\R): \det(B)=1\}$ has Lie algebra $T_eSL_2(\R)=\sl_2\R=\{ b\in M_{2\times 2}(\R): \tr(b)=0 \}$. We use the basis 
		$$
			e_1=\begin{pmatrix} 0&1\\0&0\end{pmatrix}, 
			\quad 
			e_2=\begin{pmatrix} 1&0\\0&-1\end{pmatrix}, 
			\quad \text{ and } \quad 
			e_3=\begin{pmatrix} 0&0\\1&0\end{pmatrix}
		$$
	for $\sl_2\R$. 
	
	We will consider two different norms on $\sl_2\R$. For the construction of cubes and the right-invariant metric on $SL_2\R$, we will use the inner product $\left\langle \cdot, \cdot \right\rangle_{coord}$ for which $e_1,e_2,e_3$ are orthonormal. When studying the adjoint cocycle in Section \ref{sec:fiberdynamics}, we will use the Frobenius inner product $\langle X, Y\rangle_F=\tr(X^TY)$.
	Since 
		$
			\|X\|_{coord}\leq \|X\|_F\leq \sqrt{2}\|X\|_{coord},
		$
	all qualitative statements involving H\"older regularity, Lyapunov Exponents, and perturbation estimates remain true under either norm, although constants can change. 
	
	In Section \ref{sec:fiberdynamics}, we will use the adjoint representation
		$$
			\Ad \colon SL_2\R \to \Aut(\sl_2\R), \qquad \Ad(g)X=gXg^{-1}
		$$
		
\subsection{Quotients, Haar measure, and distance} 
	We take all quotients on the right. We define the quotient maps,
		$$
			\pi_{PSL_2\R}:SL_2\R \to SL_2\R /\{\pm \id\}=PSL_2\R,
			\qquad 
			\pi_N:PSL_2\R\to PSL_2\R /\Gamma = N,
		$$
	and we write $\pi=\pi_N\circ \pi_{PSL_2\R}: SL_2\R \to N$. All of these maps are covering maps, so locally the spaces $SL_2\R, \ PSL_2\R$, and $N$ all look the same.  
	
	Haar measure on $SL_2\R$ descends to Haar measure on $PSL_2\R$ and to a finite measure on $N$. We use $\nu$ to denote Haar measure on $SL_2\R$ and the measures it descends to on $PSL_2\R$ and $N$, normalized so that that $\nu(N)=1$. Since $SL_2\R$ is unimodular, Haar measure on $SL_2\R$ and $PSL_2\R$ is bi-invariant, and the probability measure induced on $N$ is invariant under left multiplication by $PSL_2\R$. 
	
	Let $d_{SL_2\R}$ be right-invariant Riemannian distance on $SL_2\R$ induced by $\langle \cdot, \cdot \rangle_{coord}$. It descends to a right-invariant metric $d_{PSL_2\R}$ on $PSL_2\R$ and induces the metric 
		$$
			d_N(x,y)=\inf_{\gamma\in \Gamma} d_{PSL_2\R}(x,y\gamma)
		$$
	on $N$. We will simply write $d$ when the ambient space is clear. 
	
	For $B\in SL_2\R$, the operator norm of $B$ is 
		$$
			\|B\|_{op}:= \sup_{\|v\|_{\R^2}=1} \|Bv\|_{\R^2}.
		$$
	Unless otherwise specified, the norm of an element of $SL_2\R$ or $PSL_2\R$ will mean this operator norm. Since $\|B\|_{op}=\|-B\|_{op}$, this norm is well-defined on $PSL_2\R$. 
	
	The operator norm gives us another way to measure the distance between two matrices in $SL_2\R$ (or $PSL_2\R$). This distance function is locally (but not globally) equivalent to the right-invariant Riemannian distance function. 
	
	\begin{lemma}\label{lem:SL2R_difference_equivalent_to_Riem_dist}
		There exists a neighborhood $U\subset SL_2\R$ of the identity and a constant $K>0$ such that if $M\in U$, then 
			$$
				\frac{1}{K}\|M-e\|\leq d_{SL_2\R}(M,e) \leq K\|M-e\|
			$$
	\end{lemma}
	\begin{proof}
		Immediate from the fact that in a sufficiently small neighborhood of $e$, the Riemannian distance function is equivalent to the Euclidean metric. 
	\end{proof}

\subsection{Cubes in $SL_2\R$ and $N$} \label{sec:FiberCubes}
	Let
		\begin{equation*}
			U^-:= \left\{ h_s:=\begin{pmatrix} 1 & s \\ 0 & 1\end{pmatrix} \in SL_2\R \colon s\in \R \right\}\subset SL_2\R.
		\end{equation*}
	We define the unstable cube of radius $\delta>0$ to be 
		$
			C^-(\delta):= \{h_s:s\in (-\delta,\delta)\}.
		$
	Let $m_{U^-}$ be the pushforward of Lebesgue measure under $s\in \R \mapsto h_s\in U^-$, so that $m_{U^-}(C^-(\delta))=2\delta$. Note that $m_{U^-}$ is Haar measure on $U^-$. 
	
	Let 
		\begin{equation*}
			U^{cs}=\left\{ \begin{pmatrix} a & 0 \\ b & a^{-1} \end{pmatrix}\in SL_2\R \colon a,b\in \R, \ a\neq 0 \right\}
		\end{equation*}
	On $U^{cs}$, we use right Haar measure given by $dm^R_{U^{cs}}=a^{-2}dadb$. We define $\theta_{cs}(te_2+se_3)=\exp(te_2)\exp(se_3)$, where $\exp:\sl_2\R\to SL_2\R$ is the Lie exponential. We then define the center-stable cube of radius $\delta>0$ by 
		$$
			C^{cs}(\delta)
			:= \theta_{cs}\left( \left\{ te_2+se_3 \colon |t|, |s|<\delta\right\} \right)
			=\left\{ \begin{pmatrix} e^t & 0 \\ se^{-t} & e^{-t} \end{pmatrix} \colon |t|,|s|<\delta \right\}
		$$
	Then, $m^R_{U^{cs}}(C^{cs}(\delta))=\delta e^{2\delta}-\delta e^{-2\delta}$. 
	
	We now define a map
		\begin{equation*}
			\theta: \R e_1 \oplus \left( \R e_2 \oplus \R e_3 \right) \to SL_2\R, 
			\qquad \theta(v^u,v^{cs})=\exp(v^u)\theta_{cs}(v^{cs})
		\end{equation*}
	and the cube of radius $\delta>0$ by
		$$
			C(\delta)=\theta\left( C_{\sl_2\R}(\delta) \right) = C^-(\delta) C^{cs}(\delta)
		$$
	where $C_{\sl_2\R}(\delta)=\left\{ te_1+se_2+re_3 \colon |t|, |s|,|r| <\delta\right\}$. Since $D_0\theta=\id$, for $\delta_0$ sufficiently small, $\theta$ is a diffeomorphism on a neighborhood of $\overline{C_{\sl_2\R}(\delta_0)}$. From now on, all cubes in $SL_2\R$ will have radius $\delta\in \left(0,\delta_0\right]$. A cube centered at $y\in SL_2\R$ is the right translation $C(\delta)y$. 
	
	Now, we compute the measure of $C(\delta)$. Let $\phi:U^-\times U^{cs}\to SL_2\R$ be $\phi(u^-,u^{cs})\mapsto u^-u^{cs}$. Then by  \cite[Lemma 11.31]{EinsiedlerWard}, there is a constant $C_1>0$ such that 
		$$
			\nu|_{U^-U^{cs}}=C_1\phi_*\left( m_{U^-}\times m_{U^{cs}}^R \right)
		$$
	Since $C(\delta)=C^{-}(\delta)C^{cs}(\delta)$, we get that 
		\begin{equation}\label{eq:measure_delta_cube}
			\nu(C(\delta))=2C_1\delta^2\left( e^{2\delta}-e^{-2\delta} \right).
		\end{equation}
	In particular, since $e^{2\delta}-e^{-2\delta}\geq 4\delta$, we have the lower bound
		\begin{equation}\label{eq:lower_bound_measure_cubes_SL2R}
			\nu(C(\delta))\geq 8C_1\delta^3
		\end{equation}
		
	We now give an upper bound:
	
	\begin{prop}\label{prop:upper_bound_difference_measure_cubes_SL2R}
		There exists a constant $K_4>0$ such that for $r_1,r_2<\delta_0$, 
		\begin{equation*}
			\left| \nu(C(r_2))-\nu(C(r_1)) \right| \leq K_4 |r_2^3-r_1^3|.
		\end{equation*}
	\end{prop}
	\begin{proof}
			Let $\psi(r)=\nu(C(r))=2C_1r^2(e^{2r}-e^{-2r})$. Since $\psi'(r)=4C_1r(e^{2r}-e^{-2r})+4C_1r^2(e^{2r}+e^{-2r})$ and $e^{2r}-e^{-2r}=O(r)$ and $e^{2r}+e^{-2r}=O(1)$, we can find a constant $C>0$ such that for $r\in(0,\delta_0)$, $|\psi'(r)|\leq Cr^2$. Integrating between $r_1$ and $r_2$, then proves the claim. 
	\end{proof}

	We define cubes in $N$ by projecting down cubes in $SL_2\R$. Because $N$ is compact and $\pi:SL_2\R\to N$ is a covering map, there exists $r_0>0$ such that $\pi$ is injective on every ball $B_{SL_2\R}(y,r_0)$, and $\pi(B_{SL_2\R}(y,r_0))=B_N(\pi(y),r_0)$. 
	
	Choose $\delta_{inj}\in \left(0,\delta_0\right]$ so that $C(\delta_{inj})\subset B_{SL_2\R}(e,r_0)$. Then for $\delta<\delta_{inj}$ and $y\in SL_2\R$, we can define the cube of radius $\delta$ centered at $\pi(y)\in N$ by 
		$$
			C(\delta)\pi(y)=\pi(C(\delta)y).
		$$
	Whenever we speak of a cube in $N$ of radius $\delta$, we implicitly assume that $\delta\in (0,\delta_{inj})$.
	
	Finally, we observe that we can cover all but an arbitrarily small piece of $N$ by sufficiently small disjoint cubes:
	
	\begin{lemma}\label{lem:covering_N_cubes}
		For all $\e>0$, $\exists \delta_\e>0$ such that if $\delta\in (0,\delta_\e)$, then there exist $y_1,...,y_m\in N$ for which $C(\delta)y_1,...,C(\delta)y_m$ are pairwise disjoint and 
			$\displaystyle
				\nu\left(\bigcup_{j=1}^m C(\delta)y_j\right)\geq 1-\e
			$
	\end{lemma}
	\begin{proof}
		The proof is analogous to the proof of \cite[Lemma 3.9]{KanigowskiBernoulliHomogeneous}.
	\end{proof}

\section{The cocycle $A$} \label{sec:fiberdynamics}

	The purpose of this section is to reduce the cocycle $A$ to a diagonal cocycle and to establish the regularity of the resulting diagonal cocycle and of the conjugacy. Throughout this section, we only require that $g\colon M\to M$ is a $C^{1+\alpha_g}$ diffeomorphism with a nonzero Lyapunov exponent. We assume that $A\colon M\to PSL_2\R$ is $C^\alpha$, that $\log \|A\|, \log\|A^{-1}\|\in L^1(\mu)$, and that $A$ has a nonzero Lyapunov exponent. In particular, the exponentially mixing assumption on $g$ is not used in this section, and this section does not require $A$ to be $C^{1+\alpha}$. 
	
	For the convenience of the reader, we summarize the main results from this section that we will need later. Before doing this we recall a definition: A measurable function $P:M\to PSL_2\R$ is \textit{tempered with respect to} $g:M\to M$ if for almost every $x\in M$.
		$$
			\lim_{n\to \pm\infty} \frac{1}{n} \log\|P(g^nx)\| = \lim_{n\to \pm \infty} \frac{1}{n}\log\|P(g^nx)^{-1}\|=0.
		$$
	We now summarize the main results of the section:
	
	\begin{theorem}\label{thm:results_about_A}
		There exist measurable maps $P \colon M\to PSL_2\R$ and $\tau:M\to R$ defined on a full-measure set such that 
			\begin{equation*}
				A(x)=P(gx)B(x)P(x)^{-1}, \qquad 
				B(x):= \begin{pmatrix} e^{\tau(x)} & 0 \\ 0 & e^{-\tau(x)} \end{pmatrix}.
			\end{equation*}
		The diagonal cocycle $B(x)$ has the same Lyapunov exponents as $A$, and $P$ is tempered. Moreover, there exist $\alpha_b, \beta\in (0,1)$ such that on Pesin sets, 
			\begin{enumerate}
			\item 
				$\tau$ is $\alpha_b$-H\"older with a H\"older constant that grows subexponentially, more precisely,
					$$
						\left| \tau(g^kx)-\tau(g^ky)\right| \leq C_\tau e^{7\e|k|}d(g^kx,g^ky)^{\alpha_b}.
					$$
			\item 
				If $x,y\in M$ lie sufficiently close on the same local unstable manifold, then 
					$$
						d_{PSL_2\R}\left( P(g^{-k}x),P(g^{-k}y) \right)\leq C_{P}e^{\left(9\e-\frac{\lambda_g \beta}{2}\right)k}d(x,y)^\beta,
						\qquad k\geq 0.
					$$
				In particular, for $\e$ sufficiently small, $P(g^{-k}x)$ and $P(g^{-k}y)$ converge exponentially fast to one another. 
			\end{enumerate}
	\end{theorem}
	
	Throughout this section, we only require that $g\colon M\to M$ is a $C^{1+\alpha_g}$ diffeomorphism with a nonzero Lyapunov exponent. We assume that $A\colon M\to PSL_2\R$ is $C^\alpha$, that $\log \|A\|, \log\|A^{-1}\|\in L^1(\mu)$, and that $A$ has a nonzero Lyapunov exponent. In particular, the exponentially mixing assumption on $g$ is not used in this section, and this section does not require $A$ to be $C^{1+\alpha}$. 
	
	For $n\geq 1$, set 
		$$
			A_n(x)= A(g^{n-1}(x))\cdots A(x), 
			\qquad 
			A_{-n}(x)=(A_n(g^{-n}x))^{-1}
		$$
	Many of the arguments we will use in this section only make sense for cocycles over $g$ that can be represented by matrices in $GL_n(R)$. As a result, we introduce two cocycles that are extremely closely related to $A$. Choose a measurable lift $\tilde{A} \colon M\to SL_2\R$ of $A$. Although the cocycle $\tilde{A}$ need not be continuous, the adjoint cocycle 
		$$
			\mathcal{A}(x):=\Ad(A(x))=\Ad(\tilde A(x))
		$$
	is $\alpha$-H\"older. Moreover, $\mathcal{A}_n=\Ad(A_n)$ and $\|\mathcal{A}_n\|_{op, F}=\|\tilde A_n\|_{op}^2=\|A_n\|_{op}^2$, where $\|\cdot \|_{op,F}$ is the operator norm on $\Aut(\sl_2\R)$ induced by the Frobenius norm. The integrability assumptions on $A^{\pm 1}$ therefore give analogous integrability assumptions for $\tilde A^{\pm 1}$ and $\mathcal{A}^{\pm 1}$. 
 	
 \subsection{Lyapunov Exponents and Osceledets Splitting}\label{sec:Lyap_exps_osceledets_A}
 	
 \subsubsection{Lyapunov Exponents}\label{sec:Lyap_exps_A}
 	By the Furstenburg-Kesten Theorem \cite[Theorem 3.12]{VianaLyapExps}, the Lyapunov exponents for $\tilde A(x)$ are $\pm \lambda_2$ for some constant $\lambda_2\geq 0$ (ergodicity of $g$ and the integrability assumptions on $\tilde{A}^{\pm1}$ are used here). Since $\|\mathcal{A}_n\|=\|\tilde{A}_n\|^2$ and we can view $\mathcal{A}(x)\in SL_3\R$, the Lyapunov exponents for $\mathcal{A}$ are $2\lambda_2, 0, -2\lambda_2$. Thus, the hypothesis in Theorem \ref{thm:Main_f_bernoulli} that $A$ has a nonzero Lyapunov exponent is equivalent to $\lambda_2>0$, which we assume from now on. 

\subsubsection{Osceledets Splittings}\label{sec:osceledets_A}

	By Osceledets Theorem \cite[Theorem 4.2]{VianaLyapExps}, for $\mu$-a.e. $x\in M$, there is a $\tilde{A}(x)$-invariant splitting 
		$$
			\R^2=E^u_2(x)\oplus E^s_2(x)
		$$ 
	such that the subspaces $E^u_2(x)$ and $E^s_2(x)$ correspond to the Lyapunov exponents $\lambda_2$ and $-\lambda_2$ respectively. Moreover, the angles between $E^u_2(g^nx)$ and $E^s_2(g^nx)$ decay subexponentially in $n$, i.e.
		$$
			\lim_{n\to \pm\infty} \frac{1}{n}\log \sin \theta_{g^nx} =0, 
			\qquad \theta_{g^nx}:=\angle\left( E^u_2(g^nx),E^s_2(g^nx)\right)\in \left(0,\frac{\pi}{2}\right]
		$$
	We call the points $x\in M$ for which the conclusion of Osceledets Theorem holds \textit{regular}. Note that since multiplication by $-\id$ changes orientations, but does not affect expansion or contraction along orbits, the splitting $\R^2=E^u_2(x)\oplus E^s_2(x)$ is independent of the choice of lift $\tilde{A}$.
	
	Osceledets Theorem also gives, for $\mu$-a.e. $x\in M$, an $\mathcal{A}$-invariant splitting, 
		$$
			\sl_2\R =E^u_3(x)\oplus E^c_3(x) \oplus E^s_3(x)
		$$
	with the subspaces $E^u_3(x)$, $E^c_3(x)$, and $E^s_3(x)$ corresponding to the Lyapunov exponents $2\lambda_2, 0, -2\lambda_2$ respectively. We can express the Osceledets splitting for $\mathcal{A}$ explicitly in terms of the Osceledets splitting for $\tilde{A}$. Fix unit vectors $v_u(x)\in E^u_2(x)$ and $v_s(x)\in E^s_2(x)$, and let $u_u(x)^T$ and $u_s(x)^T$ be the corresponding dual basis. 
 
  	\begin{prop}\label{prop:osceledets_subspaces_Ad}
 	 	The Osceledets subspaces for $\mathcal{A}$ are 
 	 		\begin{equation*}
 	 			E^u_3(x) = \vecspan \{ v_u(x)u_s(x)^T \}, \quad
 	 			E^c_3(x) = \vecspan \{  v_u(x)u_u(x)^T-v_s(x)u_s(x)^T \}, \quad
 	 			E^s_3(x) = \vecspan \{ v_s(x)u_u(x)^T \}.
 	 		\end{equation*}
 	 	Thus, the set of regular points for $\tilde{A}$ is the same as the set of regular points for $\mathcal{A}$.
  	\end{prop}
  	\begin{proof}
		For $\sigma\in \{u,s\}$, write 
			$
				\tilde{A}_n(x)v_{\sigma}(x)=a_n^\sigma(x)v_\sigma(g^nx).
			$
		Then, 
			$$
				\mathcal{A}_n(x) v_i(x)u_j(x)^T = \frac{a_n^i(x)}{a_n^j(x)}  v_i(g^nx) u_j(g^nx)^T
			$$
		Thus, 
	 		$$
	 			\lim_{n\to \pm\infty} \frac{1}{n}\log \left\| \mathcal{A}_n(x) v_i(x)u_j(x)^T\right\|
	 			= \lambda_i-\lambda_j 
	 		$$
	 	Since $v_u(x)u_s(x)^T, v_s(x)u_u(x)^T, v_u(x)u_u(x)^T-v_s(x)u_s(x)^T$ are all nonzero elements of $\sl_2\R$ and each $E^\sigma_3(x)$ is one dimensional, we're done. 
  	\end{proof}
 	
 	Proposition \ref{prop:osceledets_subspaces_Ad} implies that the set of regular points for $\mathcal{A}$ is the same as the set of regular points for $\tilde A$. It also gives the following characterization of $E^u_3$ and $E^s_3$ in terms of $E^u_2$ and $E^s_2$:
 	
 	\begin{cor}\label{cor:stable_unstable_osceledets_subspaces_Ad}
		\begin{align*}
			E^u_3(x)&=\left\{ X\in \sl_2\R \colon X(\R^2)\subset E^u_2(x), \ X|_{E^u_2(x)}=0 \right\}, \quad
			E^s_3(x) = \left\{ X\in \sl_2\R \colon X(\R^2)\subset E^s_2(x), \ X|_{E^s_2(x)}=0 \right\} 
		\end{align*}
 	\end{cor}
 	
 \subsection{Pesin Sets}\label{sec:pesin_sets_A}
 	In this section, we will define the Pesin set $\Lambda_\e^\ell$ and give the properties of it that we will use later. We begin by defining Pesin sets for the cocycle $\tilde{A}$. 
	
 	\begin{definition}[{\cite[Definition 2.2.6]{BarreiraPesinNonUniform}}]\label{def:PesinSets}
 		For $\e>0$ and $\ell\geq 1$, the \emph{Pesin set} $\Lambda_{\e,\tilde A}^\ell$ is the set of regular points in $M$ such that $\forall k\in \Z$ and $m\geq 0$ the following hold:
 		\begin{enumerate}
 		\item 
 			If $v\in E^s_2(g^kx)$, then 
 				\begin{align*}
 					\left\| \tilde A_{m}(g^kx)v\right\| \leq \ell e^{-\lambda_2 m+\e m+\e |k|} \|v\|
 					\quad \text{and} \quad 
 					\left\| \tilde A_{-m}(g^kx)v\right\| \geq \ell^{-1} e^{\lambda_2 m-\e m-\e |k-m|} \|v\|
 				\end{align*}
 		\item 
 			If $v\in E^u_2(g^kx)$, then 
 				\begin{align*}
 					\left\| \tilde A_{-m}(g^kx)v\right\| \leq \ell e^{-\lambda_2 m+\e m+\e |k|} \|v\|
 					\quad \text{and} \quad 
 					\left\| \tilde A_{m}(g^kx)v\right\| \geq \ell^{-1} e^{\lambda_2 m-\e m-\e |k+m|} \|v\|
 				\end{align*}
 		\item 
 			$\displaystyle
 				\angle \left( E^u_2(g^kx),E^s_2(g^kx) \right)\geq \ell^{-1}e^{-\e |k|}
 			$
 		\end{enumerate}
 	\end{definition}
 	
 	We also need Pesin sets for the adjoint cocycle $\mathcal{A}$. For $\e>0$ and $\ell\geq 1$, let $\Lambda_{\e,\mathcal{A}}^\ell$ be the Pesin set for $\mathcal{A}$ corresponding to the Osceledets splitting $\sl_2\R =E^u_3(x)\oplus E^c_3(x) \oplus E^s_3(x)$. The precise definition of $\Lambda_{\e,\mathcal{A}}^\ell$ is only used in the proof of Proposition \ref{prop:E^sE^uHolderPesin} and is included in Appendix \ref{appendix:pesin_theory}. 
 	
 	We now define 
 		$$
 			\Lambda_\e^\ell:= \Lambda_{\e,\tilde{A}}^\ell \cap \Lambda_{\e, \mathcal{A}}^\ell
 		$$
 
 	We will use the following properties of $\Lambda_\e^\ell$ throughout this paper:
 
  	\begin{prop}\label{prop:PesinSetsFacts}
  		Fix $\e>0$, 
  		\begin{enumerate}
  		\item For all $\ell\geq 1$, $\Lambda_{\e}^\ell\subset \Lambda_{\e}^{\ell+1}$.
  		\item For all $\ell\geq 1$ and $k\in \Z$, $g^k\left( \Lambda_{\e}^\ell \right)\subset \Lambda_{\e}^{\ell'}$ where $\ell'=\ell e^{|k|\e}$.
  		\item $\displaystyle \mu\left( \bigcup_{\ell\geq 1} \Lambda_{\e}^\ell \right)=1$. In fact, $\bigcup_{\ell\geq 1} \Lambda_{\e }^\ell$ is the set of regular points for $A$.
  		\end{enumerate}
  	\end{prop}
  	\begin{proof}
  		Each property for $\Lambda_\e^\ell$ follows from the analogous properties holding for both $\Lambda_{\e,\tilde{A}}^\ell$ and $\Lambda_{\e,\mathcal{A}}^\ell$. The properties for $\Lambda_{\e,\tilde{A}}^\ell$ and $\Lambda_{\e,\mathcal{A}}^\ell$ follow by the same arguments given in \cite[Proposition 2.2.7]{BarreiraPesinNonUniform}.
  	\end{proof}

\subsection{Regularity of $E^s_2$ and $E^u_2$}\label{sec:regularityE^s_E^u}
	
	For subspaces $E,F\subset \R^n$ let $\dist(E,F)$ be the gap metric \cite[(1.3)]{AraujoBufetovFilip}. If $E$ and $F$ are one-dimensional, then
		\begin{equation}\label{eq:angle_distance_subspaces}
			\dist(E,F)=\sin \left(\angle (E,F)\right)
		\end{equation}
	
	The stable and unstable distributions $E^u_2$ and $E^s_2$ for $\tilde A$ are H\"older on Pesin sets with a H\"older constant that grows subexponentially along $g$-orbits. More precisely, 
	
	\begin{prop}\label{prop:E^sE^uHolderPesin}
		There exists $\beta\in (0,1)$ such that for all $\e\in (0,\frac{\lambda_2}{4})$, $\ell\geq 1$, there exists a constant $C_{\e,\ell}>0$ such that $\forall k\in \Z$ and for $\sigma\in \{s,u\}$, if $x,y\in \Lambda_\e^\ell$ with $d(g^kx,g^ky)\leq 1$, then 
			\begin{equation*}
				\dist\left( E^\sigma_2(g^kx), E^\sigma_2(g^ky) \right) \leq C_{\e,\ell}e^{6\e|k|} d(g^kx,g^ky)^\beta
			\end{equation*}
	\end{prop}
	\begin{proof}
		See Appendix \ref{appendix:pesin_regularity_osceledets}
	\end{proof}

	\begin{cor}\label{cor:theta_x_holder}
		Under the same hypotheses as Proposition \ref{prop:E^sE^uHolderPesin}, 
			\begin{equation*}
				|\theta_{g^kx}-\theta_{g^ky}|\leq 4C_{\e,\ell}e^{6\e|k|}d(g^kx,g^ky)^{\beta}
			\end{equation*}
	\end{cor}
	\begin{proof}
		Immediate from the triangle inequality and Proposition \ref{prop:E^sE^uHolderPesin}.
	\end{proof}

\subsection{Osceledets-Pesin Reduction} \label{subsub:Osceledets_Pesin_Reduction}
	The following result, known as the Osceldets-Pesin Reduction, allows us to ``diagonalize" $\tilde A(x)$. The result in its standard form can be found in \cite[Theorem 3.5.5]{BarreiraPesinNonUniform}. We will use a slightly modified form: 
	
	\begin{theorem}\label{thm:Osceledets_Pesin_Reduction}
		Let $x\in M$ be a regular point. Let $\theta_x\in \left( 0,\frac{\pi}{2}\right]$ be the angle between $E^u_2(x)$ and $E^s_2(x)$. There exists a matrix $\tilde P(x)\in SL_2\R$ such that 
		\begin{enumerate}[(1)]
		\item 
			$\tilde P(x)$ sends the orthogonal decomposition of $\R^2$ to the decomposition $\R^2=E^u_2(x)\oplus E^s_2(x)$. 
		\item 
			The cocycle $\tilde B(x):=\tilde P(gx)^{-1}\tilde A(x)\tilde P(x)\in SL_2\R$ has diagonal form $\tilde B(x)=\begin{pmatrix} b(x) & 0 \\ 0 & b(x)^{-1}\end{pmatrix}$, where $b(x)\neq 0$. 
		\item 
			$\log|b(x)|\in L^1(\mu)$ and $\log^+\|\tilde B\|_{2,op}\in L^1(\mu)$.
		\item 
			The Lyapunov exponents for $\tilde B(x)$ are $\pm \lambda_2$.
		\item 
			$\tilde P$ is tempered $\mu$-a.e. i.e. for $\mu$-a.e. $x\in M$, 
				\begin{equation*}
					\lim_{n\to \pm\infty} \frac{1}{n}\log\|\tilde P(g^nx)\|_{2,op}=\lim_{n\to \pm\infty} \frac{1}{n}\log\|\tilde P(g^nx)^{-1}\|_{2,op}=0
				\end{equation*} \label{thm:Osceledets_Pesin_Reduction_tempering_part}
		\end{enumerate}
	\end{theorem}
	
	\begin{proof}
		The proof is standard, but we include an outline of it in Appendix \ref{appendix:proof_osceledets_pesin_reduction} because we will use several of the intermediate steps and bounds later. 
	\end{proof}
	
	Let $P(x)\in PSL_2\R$ be the projection of $\tilde{P}(x)$. Projecting $\tilde{B}(x)$ to $PSL_2\R$ and choosing the representative with positive diagonal entries gives
		\begin{equation}\label{eq:osceledets_pesin_diagonalize_A}
			B(x)=P(gx)^{-1}A(x)P(x)
			=\begin{pmatrix}
				e^{\tau(x)} & 0 \\ 0 & e^{-\tau(x)}
			\end{pmatrix}, 
			\qquad \tau(x):=\log|b(x)|
		\end{equation}
		
\subsection{Regularity of the Osceledets-Pesin Reduction}
	In this section we give regularity estimates for the diagonal cocycle $B$ and for the conjugacy $P$ that will be used in sections \ref{sec:diagonal_skew_product} and \ref{sec:unstable_f}. 
	
	For $\e>0$ and $\ell\geq 1$, we define the Pesin set $\bar{\Lambda}_\e^\ell$ analogously to how we defined the Pesin set $\Lambda_{\e,\tilde{A}}^\ell$ for $\tilde A$ in Definition \ref{def:PesinSets}.
	
	We first establish the H\"older regularity of the diagonal cocycle $B$.
	
	\begin{prop}\label{prop:tau_holder_on_pesin} %in D14 was {cor:tau_holder_on_pesin}
		There exists $\alpha_b\in (0,1)$ such that for all $\e\in (0,\frac{\lambda_2}{4})$ and $\ell\geq 1$, there exists a constant $C_\tau=C_\tau(\e,\ell)>0$ such that for all $k\in \Z$, if $x,y\in \Lambda_\e^\ell \cap \bar\Lambda_\e^\ell$ with $d(g^kx,g^ky)\leq 1$, then 
			$$
				\left| \tau(g^kx)-\tau(g^ky)\right| \leq C_\tau e^{7\e|k|}d(g^kx,g^ky)^{\alpha_b}
			$$
	\end{prop}
	\begin{proof}
		See Appendix \ref{appendix:proof_tau_holder}. 
	\end{proof}
	
	Now, we turn to the regularity of the conjugacy $P$. The following estimate follows from the Holder regularity of the stable and unstable distributions $E^s_2$ and $E^u_2$ on Pesin sets:

	\begin{prop}\label{prop:P_holder_on_pesin-norm}
		For all $\e\in (0,\frac{\lambda_2}{4})$ and $\ell\geq 1$, there exists a constant $C_{P,1}=C_{P,1}(\e,\ell)>0$ such that $\forall k\in \Z$, if $x,y\in \Lambda_\e^\ell$ with $d(g^kx,g^ky)\leq 1$ then, 
			$$
				\min\left\{ \left\|\tilde P(g^kx) - \tilde P(g^ky) \right\|, \left\|\tilde P(g^kx) + \tilde P(g^ky) \right\| \right\} \leq C_{P,1} e^{8\e|k|}d(g^kx,g^ky)^\beta
			$$
	\end{prop}
	\begin{proof}
		See Appendix \ref{appendix:proofs_holder_P(x)}.
	\end{proof}

	Proposition \ref{prop:P_holder_on_pesin-norm} doesn't quite show that $P$ is H\"older on Pesin sets, or at least not in a form that is independent of $x$ and $k$. This is because for arbitrary matrices $M_1,M_2\in SL_2\R$, our right-invariant Riemannian distance function $d_{SL_2\R}(M_1,M_2)$ is not equivalent to $\|M_1-M_2\|$. However, we can use that these distances are equivalent in a neighborhood of the identity (Lemma \ref{lem:SL2R_difference_equivalent_to_Riem_dist}) to show that Proposition \ref{prop:P_holder_on_pesin-norm} implies that $P$ is H\"older when restricted to local unstables and Pesin sets. 
	
	\begin{prop}\label{prop:P_holder_on_pesin_unstables}
		Take $\e\in (0,\min\{\frac{\lambda_2}{4},\frac{\lambda_g\beta}{18}\})$ and $\ell\geq 1$. Take $\tau_g\in (0,1)$. There exists a constant $C_{P,2}=C_{P,2}(\e,\ell,\tau_g)>0$ and $r_p=r_p(\e,\ell,\tau_g)\in (0,1)$ such that if $x,y\in \Lambda_\e^\ell \cap \mathcal{P}_{\tau_g}^{\delta_g}$ with $y\in W^u_g(x)\cap B(x,r_p)$, then $\forall k\geq 0$, 
			$$
				d_{PSL_2\R}\left( P(g^{-k}x),P(g^{-k}y) \right)\leq C_{P,2}e^{\left(9\e-\frac{\lambda_g \beta}{2}\right)k}d(x,y)^\beta
			$$
	\end{prop}
	\begin{proof}
		See Appendix \ref{appendix:proofs_holder_P(x)}.
	\end{proof}

\section{Ergodicity of $f$}\label{sec:ergodicity_f}

	\begin{prop}\label{prop:f_ergodic}
		Assume $g:(M,\mu)\to (M,\mu)$ is invertible and ergodic and $A:M\to PSL_2\R$ is measurable with a nonzero Lyapunov exponent and $\log^+\|A^{\pm 1}\|\in L^1(\mu)$. Then, the skew product $f:M\times N\to M\times N$,  $f(x,y)=(g(x),A(x)y)$ is ergodic.
	\end{prop}
	\begin{proof}
		By Theorem \ref{thm:Osceledets_Pesin_Reduction}, we have that 
			\begin{equation*}
				A(x)=P(gx)B(x)P(x)^{-1}, 
				\qquad B(x)=\diag(e^{\tau(x)},e^{-\tau(x)}), 
				\qquad \int_M \tau d\mu=\lambda_2>0
			\end{equation*}
		Note that the map $\Phi:M\times N\to M\times N$, $\Phi(x,y)=(x,P(x)y)$ is a measure-theoretic isomorphism and that $f\circ \Phi=\Phi\circ f_{diag}$, where $f_{diag}(x,y)=(gx,B(x)y)$. It therefore suffices to prove that $f_{diag}$ is ergodic. 
		
		Let $m=\mu\times \nu$, and let $F(x,y)=\phi(x)u(y)$ and $G(x,y)=\psi(x)v(y)$, where $\phi,\psi\in C^0(M)$ and $u,v\in C^0(N)$.  We use $g_t$ to denote the geodesic flow. Since for $\mu$-a.e. $x\in M$, $\tau_n(x)\to \infty$ and the geodesic flow $g_t$ is mixing on $N$ by \cite[Theorem 1.1]{eskin_mixing_1993}, we get that as $n\to \infty$, 
			\begin{equation}\label{eq:f_ergodic_pf1}
				c_k(x):=\int_N u(y)v(B_k(x)y)d\nu(y) \to \left( \int_N u d\nu \right)\left( \int_N v d\nu \right).
			\end{equation}
		Observe that 
			$
				\langle F, G\circ f^k_{diag}\rangle 
				= \int_M \phi(x)\psi(g^kx)c_k(x)d\mu(x).
			$
		Since $|c_k(x)|\leq \|u\|_{L^2}\|v\|_{L^2}$, \eqref{eq:f_ergodic_pf1}, dominated convergence, and the Mean Ergodic Theorem for $g$ give that 
			\begin{equation}\label{eq:f_ergodic_pf2}
				\frac{1}{n}\sum_{k=0}^{n-1} \langle F, G\circ f^k_{diag}\rangle  \to \left(\int F dm\right)\left( \int G dm \right).
			\end{equation}
		Since linear combinations of product functions $\phi(x)u(y)$ are dense in $L^2(M\times N)$, and the Koopman operator for $f_{diag}$ is an $L^2(M\times N)$ isometry, we get that \eqref{eq:f_ergodic_pf2} holds for all $F,G\in L^2(M\times N)$.
	\end{proof}

\section{The ``Diagonal'' Skew Product $f_{diag}$}\label{sec:diagonal_skew_product}
	
	Recall from Theorem \ref{thm:Osceledets_Pesin_Reduction} that for $\mu$-a.e. $x\in M$, 
		\begin{equation*}
			B(x):=P(gx)^{-1}A(x)P(x)=\begin{pmatrix} e^{\tau(x)} & 0 \\ 0 & e^{-\tau(x)} \end{pmatrix}.
		\end{equation*}
	After defining $B$ arbitrarily on the null set where the reduction is not defined, we let 
		\begin{equation*}
			f_{diag}\colon M\times N\to M\times N, \qquad f_{diag}(x,y)=(gx,B(x)y)
		\end{equation*}
	The advantage of $f_{diag}$ over $f$ is that its fiber dynamics can be computed explicitly. In this section, we explicitly describe the unstable manifolds for $f_{diag}$ and then identify the corresponding conditional measures. Before doing this, we introduce some notation.
	
	For $n\in \Z$, we can write 
		\begin{equation*}
			f^n_{diag}(x,y)=(g^nx,B_n(x)y), \qquad
			B_n(x)= \begin{pmatrix} e^{\tau_n(x)} & 0 \\ 0 & e^{-\tau_n(x)} \end{pmatrix}.
		\end{equation*}
	where, for $n>0$, 
		\begin{equation*}
			\tau_n(x)= \sum_{j=0}^{n-1}\tau(g^jx), \qquad 
			\tau_0(x)=0, \qquad
			\tau_{-n}(x)=-\tau_{n}(g^{-n}x)
		\end{equation*}
	
	By Theorem \ref{thm:Osceledets_Pesin_Reduction}, the Lyapunov exponents for $B(x)$ acting on $\R^2$ are $\pm \lambda_2$ and $\lambda_2=\int \tau d\mu$. By an analogous argument to that given in Section \ref{sec:Lyap_exps_A}, we see that the Lyapunov exponents for $\Ad(B(x))$ acting on $\sl_2\R$ are $2\lambda_2, 0, -2\lambda_2$. By the Birkhoff Ergodic Theorem, 
		\begin{equation*}
			\frac{1}{n}\tau_n\to \int_{M}\tau d\mu=\lambda_2
		\end{equation*}
	for $\mu$-a.e. $x\in M$ as $n\to \pm \infty$. We also recall that $\tau$ is H\"older on $\Lambda_\e^\ell\cap  \bar{\Lambda}_\e^\ell$ by Proposition \ref{prop:tau_holder_on_pesin}.
	
	\begin{remark}
		The arguments in this section only require that $g:(M,\mu)\to (M,\mu)$ is a $C^{1+\alpha_g}$ ergodic diffeomorphism that preserves the smooth measure $\mu$ and that has at least one nonzero Lyapunov exponent and that $A:M\to PSL_2\R$ is $C^{\alpha}$ and has a nonzero Lyapunov exponent. 
	\end{remark}
	
\subsection{Unstable manifolds for $f_{diag}$} \label{sec:unstable_manifolds_f_diag}
	Fix 
		$$
			0<\e <\min\left\{ \frac{\lambda_2}{4}, \frac{\alpha_b\lambda_g}{14} \right\}.
		$$
	Then
		$$
			\kappa:= \frac{\alpha_b\lambda_g}{2}-7\e>0
		$$
	We will work on the full measure set
		$$
			\mathcal{R}_{\delta_g,\e}:= \bigcup_{\ell\geq 1} \bigcup_{\tau_g\in(0,1)} G_{\ell,\tau_g}, 
			\qquad \text{where} \qquad
			G_{\ell,\tau_g}:=\Lambda_\e^\ell \cap \bar\Lambda_\e^\ell \cap \mathcal{P}_{\tau_g}^{\delta_g}.
		$$

	Before we can explicitly characterize the unstable manifolds for $f_{diag}$, we define the correction term that accounts for the difference in the fiber dynamics along two points on the same base unstable leaf. 
	
	\begin{lemma}\label{lem:tau_infty_convergence}
		Take $x,x'\in \mathcal{R}_{\delta_g,\e}$ with $x'\in W^u_{g}(x)$. Then the series 
			\begin{equation}\label{eq:tau_infty_convergence}
				\sum_{j=1}^\infty \left( \tau\left(g^{-j}x'\right)-\tau\left(g^{-j}x\right) \right)
			\end{equation}
		converges absolutely.
	\end{lemma}
	\begin{proof}
		Since $x,x'\in \mathcal{R}_{\delta_g,\e}$, the exists $\ell\geq 1$ and $\tau_g\in (0,1)$ such that $x,x'\in G_{\ell,\tau_g}$. Since $\lambda_g>0$, we can find $j_0\in \Z_+$ such that for $j\geq j_0$, we have that $\tau_g^{-1}e^{-\frac{\lambda_g}{2}j}d(x,x')\leq 1$. Thus, \eqref{eq:g_unstable_exp_contracting2} implies that $d\left(g^{-j}x',g^{-j}x\right)\leq 1$ for all $j\geq j_0$.  We can therefore combine Proposition \ref{prop:tau_holder_on_pesin} and \eqref{eq:g_unstable_exp_contracting2} to get that for $j\geq j_0$,
			\begin{equation}\label{eq:tau_infty_terms_bound}
				\left| \tau\left(g^{-j}x'\right)-\tau\left(g^{-j}x\right) \right| 
				\leq C_\tau e^{7\e j} d\left(g^{-j}x',g^{-j}x\right)^{\alpha_b} 
				\leq C_\tau \tau_g^{-\alpha_b}e^{-\kappa j}d(x,x')^{\alpha_b}
			\end{equation}
		Since the tail of \eqref{eq:tau_infty_convergence} is bounded by a geometric series and the first $j_0-1$ terms are finite, we get that the series \eqref{eq:tau_infty_convergence} converges absolutely. 
	\end{proof}
	
	As a consequence of Lemma \ref{lem:tau_infty_convergence}, if $x,x'\in \mathcal{R}_{\delta_g,\e}$ with $x'\in W^u_{g}(x)$, we can define
		\begin{equation*}
			\tau_\infty(x',x):=\sum_{j=1}^\infty \left( \tau\left(g^{-j}x'\right)-\tau\left(g^{-j}x\right) \right)
		\end{equation*}
	Whenever $x,x',$ and $x''$ are all on the same unstable leaf, we get that 
		\begin{equation}\label{eq:tau_infty_change_basepoints}
			\tau_\infty(x'',x)=\tau_\infty(x'',x')+\tau_\infty(x',x).
		\end{equation}
	Moreover, the $n$th partial sum satisfies 
		\begin{equation}\label{eq:tau_infty_partial_sums}
			\sum_{j=1}^n \left( \tau\left(g^{-j}x'\right)-\tau\left(g^{-j}x\right) \right) 
			= \tau_{-n}(x)-\tau_{-n}(x')
		\end{equation}
		
	The estimates used to prove Lemma \ref{lem:tau_infty_convergence} also give the following leafwise H\"older control on $\tau_\infty$:
	
	\begin{lemma}\label{lem:bound_tau_infty}
		Fix $\ell\geq 1$ and $\tau_g\in (0,1)$. There is a constant $C_{\tau_\infty}=C_{\tau_\infty}(\tau_g,\delta_g,\e,\ell)>0$ such that if $x,x'\in G_{\ell,\tau_g}$ with $x'\in W^u_g(x)$ and $d(x,x')\leq \tau_g e^{\frac{\lambda_g}{2}}$, then 
			\begin{equation}\label{eq:bound_tau_infty}
				|\tau_\infty(x',x)|\leq C_{\tau_\infty} d(x',x)^{\alpha_b}
			\end{equation}	
		Further, there exists $\xi_0=\xi_0(\tau_g,\delta_g)>0$ such that $d(x,x')\leq \tau_g e^{\frac{\lambda_g}{2}}$ holds whenever $x'\in W^u_{g,\xi}(x)$ for $0<\xi\leq \xi_0$. 
	\end{lemma}
	\begin{proof}
		Since $d(x,x')\leq \tau_g e^{\frac{\lambda_g}{2}}$, we can take $j_0=1$ in the proof of Lemma \ref{lem:tau_infty_convergence}. Thus, \eqref{eq:tau_infty_terms_bound} holds for all $j\geq 1$, which gives that 
			\begin{equation*}
				\left| \tau_\infty(x',x) \right|
				\leq C_\tau \tau_g^{-\alpha_b}d(x,x')^{\alpha_b} \sum_{j=1}^{\infty} e^{-\kappa j}.
			\end{equation*}
		So, \eqref{eq:bound_tau_infty} holds with 
			$
				C_{\tau_\infty}= C_\tau \tau_g^{-\alpha_b} \frac{e^{-\kappa}}{1-e^{-\kappa}}.
			$
		The final assertion follows from the uniform parameterization of unstables for $g$ in Lemma \ref{lem:unstables_g}.
	\end{proof}
	
	The following consequence of Lemma \ref{lem:bound_tau_infty} will also be useful. 
	
	\begin{lemma}\label{lem:bound_e^tau_infty}
		For every $\delta>0$, there exists $\xi_3\in \left(0,\xi_0\right]$ (where $\xi_0$ is as in Lemma \ref{lem:bound_tau_infty}) such that if $x_0,x,x'\in \Lambda_\e^\ell \cap \bar\Lambda_\e^\ell \cap \mathcal{P}_{\tau_g}$ and $x,x'\in W^u_{g,\xi_3}(x_0)$, then $e^{\pm\tau_\infty(x',x)}\in (1-\delta,1+\delta)$.
	\end{lemma}
	\begin{proof}
		Combining \eqref{eq:tau_infty_change_basepoints} with Lemma \ref{lem:bound_tau_infty}, we get that 
			\begin{equation*}
				|\tau_\infty(x',x)|\leq C_{\tau_\infty}\left( d(x',x_0)^{\alpha_b}+d(x,x_0)^{\alpha_b} \right)
			\end{equation*}
		Choose $\xi_3\leq \xi_0$ so that the right hand side is less than $\log(1+\delta)$. (We can do by Lemma \ref{lem:unstables_g}.) Then $e^{|\tau_\infty(x',x)|}<1+\delta$ and $e^{-|\tau_\infty(x',x)}>(1+\delta)^{-1}\geq 1-\delta$.
	\end{proof}
	
	We combine the diagonal correction $\tau_\infty(x',x)$ with the unstable horocycle flow to parameterize the fiber component of the unstable manifolds for $f_{diag}$. For $s\in \R$, we define 
		\begin{equation}\label{eq:def_r_s}
			r_s=r_s(x,x')
			:= \begin{pmatrix}
					e^{\tau_\infty(x',x)} & s \\ 0 & e^{-\tau_\infty(x',x)}
				\end{pmatrix}.			
		\end{equation} 
	
	We can write $r_s$ as a product of a horocycle flow and a geodesic flow. To be precise, for $t\in \R$, we let $g_t=\diag(e^t,e^{-t})$ be the geodesic flow and let $h_t=\begin{pmatrix} 1 & t \\ 0 & 1\end{pmatrix}$ be the horocycle flow. For any $s,t\in \R$, 
		\begin{equation}\label{eq:r_s_horocycle_geodesic}
			r_s(x,x')
			=h_{se^{\tau_\infty(x',x)}}g_{\tau_\infty(x',x)}
			=g_{\tau_\infty(x',x)}h_{se^{-\tau_\infty(x',x)}}
		\end{equation}
		
	For $(x,y)\in \mathcal{R}_{\delta_g,\e}\times N$, the local unstable manifold for $f_{diag}$ centered at $(x,y)$ can be characterized as follows:
		\begin{equation}\label{eq:unstable_fdiag_characterization}
			W^u_{diag}(x,y)=\left\{ (x',y')\in M\times N \colon \limsup_{n\to \infty} \frac{1}{n}\log d_{M\times N}\left( f_{diag}^{-n}(x,y), f_{diag}^{-n}(x',y')\right) <0 \right\}
		\end{equation}
	(See \cite[p.233]{BarreiraPesinNonUniform} or \cite[Lemma 2.8]{ExpMixingBernoulli}.) We can explicitly describe the local unstable manifolds for $W^u_{diag}(x,y)$ using the function $r_s$:
		
	\begin{prop}\label{prop:unstables_f_diag}
		For $(x,y)\in \mathcal{R}_{\delta_g,\e}\times N$,
			$$
				W^u_{diag}(x,y)\cap \left( \mathcal{R}_{\delta_g,\e}\times N \right)=\left\{
					(x',y')\in \mathcal{R}_{\delta_g,\e}\times N \colon x'\in W^u_g(x), \ y'= r_s(x,x')y, \ s\in \R
				\right\}
			$$
	\end{prop}
		
	\begin{proof}
		Fix $(x,y)\in \mathcal{R}_{\delta_g,\e}\times N$. Since the $G_{\ell,\tau_g}$ are nested, it suffices to prove that if $\tau_g\in (0,1)$ and $\ell\geq 1$ are such that $x\in G_{\ell,\tau_g}$, then 
			$$
				W^u_{diag}(x,y)\cap \left( G_{\ell,\tau_g}\times N \right)=\left\{
					(x',y')\in G_{\ell,\tau_g}\times N \colon x'\in W^u_g(x), \ y'= r_s(x,x')y, \ s\in \R
				\right\}.
			$$
		Fix $\tau_g\in (0,1)$ and $\ell\geq 1$ such that $x\in G_{\ell,\tau_g}$. We first prove the reverse inclusion. Suppose $x'\in W^u_g(x)\cap G_{\ell,\tau_g}$ and that $y'=r_s(x,x')y$. To prove that $(x',y')\in W^u_{diag}(x,y)$, by \eqref{eq:unstable_fdiag_characterization} it suffices to prove that $d_M(g^{-n}x,g^{-n}x')$ and $d_N(B_{-n}(x)y,B_{-n}(x')y')$ go to zero exponentially fast. This is automatic for the base coordinates since $x'\in W^u_g(x)$. To see it for the fiber coordinates, observe that by \eqref{eq:tau_infty_partial_sums}, 
			\begin{equation*}
				q_n:= \tau_\infty(x',x)+\tau_{-n}(x')-\tau_{-n}(x)=\sum_{j=n+1}^\infty \left( \tau(g^{-j}x')-\tau(g^{-j}x) \right)
			\end{equation*}
		After replacing $y'$ and $y$ with appropriate lifts, we can observe that
			\begin{equation}\label{eq:unstables_fdiag_proof1}
				B_{-n}(x')y'y^{-1}B_{-n}(x)^{-1}
				= \begin{pmatrix} 	
					e^{q_n} & se^{\tau_{-n}(x')+\tau_{-n}(x)} \\ 0 & e^{-q_n}
					\end{pmatrix}
			\end{equation}
		The proof of Lemma \eqref{lem:tau_infty_convergence} gives that for sufficiently large $n$, $|q_n|\leq Ce^{-\kappa n}$. Since $x,x'\in \bar{\Lambda}_\e^\ell$, we get that $e^{\tau_{-n}(x)}, e^{\tau_{-n}(x')}\leq \ell e^{-(\lambda_2-\e)n}$. We've shown that the entries of the matrix in \eqref{eq:unstables_fdiag_proof1} go to the corresponding entries of the identity matrix exponentially fast. Applying Lemma \ref{lem:SL2R_difference_equivalent_to_Riem_dist} along with the quotient map $\pi_n:PSL_2\R\to N$ then gives that that $d_N(B_{-n}(x)y,B_{-n}(x')y')$ goes to zero exponentially fast. 
		
		We now prove the forward inclusion. Suppose $(x',y_1)\in W^u_{diag}(x,y)\cap (G_{\ell,\tau_g}\times N)$. Projecting down to $M$, we get that $x'\in W^u_g(x)$. So, now we just need to show that there exists $s\in \R$ such that $y_1=r_s(x,x')y$. Let $y_0:=r_0(x,x')y$. By what we just proved, $(x',y_0)\in W^u_{diag}(x,y)$. Combining \eqref{eq:unstable_fdiag_characterization} with the triangle inequality gives that 
			$$
				\limsup_{n\to \infty} \frac{1}{n} \log d_N(B_{-n}(x')y_1,B_{-n}(x')y_0)<0. 
			$$
		Thus, there exists $c>0$ such that for all $n$ sufficiently large, 
			\begin{equation*}
				d_N(B_{-n}(x')y_1,B_{-n}(x')y_0)\leq e^{-cn}.
			\end{equation*}
		We can write $B_{-n}(x')=g_{\tau_{-n}(x')}=g_{-s_n}$, where $s_n=-\tau_{-n}(x')$. If $t>0$, we get that for sufficiently large $n$, 
			\begin{equation}\label{eq:unstables_fdiag_proof2}
				d(g_{-t}y_1,g_{-t}y_0) 
				\leq \left\|\Ad(g_{s_n-t})\right\| d_N(g_{-s_n}y_1,g_{-s_n}y_0)
				\leq e^{2|s_n-t|-cn}
				\leq e^{(2\left|\frac{s_n}{t}-1\right|-c\frac{n}{t})t}
			\end{equation}
		For $t>0$, let $n_t=\lfloor \frac{t}{\lambda_2}\rfloor$. Note that $\frac{t}{n_t}\to \lambda_2$ as $t\to \infty$. Also note that since $x'\in \bar{\Lambda}_\e^\ell$ is a Lyapunov regular point for $B$, we have that 
			$$
				\frac{\tau_{-n}(x')}{-n}=\frac{s_n}{n}\to \lambda_2
				\quad \Longrightarrow \quad 
				\frac{s_{n_t}}{n_t} \cdot \frac{n_t}{t} \to 1
			$$
		Thus \eqref{eq:unstables_fdiag_proof2} implies that $d(g_{-t}y_1,g_{-t}y_0) \to 0$ as $t\to \infty$. We have shown that $y_0$ and $y_1$ are on the same strong unstable leaf of the geodesic flow on $N$. The strong unstable leaves of the geodesic flow on $N$ are of the form $\{h_by_2 \colon b\in \R\}$ for $y_2\in N$. Thus, $y_1=h_by_0$ for some $b\in \R$. Combining this with the fact that $y_0:=r_0(x,x')y$ and \eqref{eq:r_s_horocycle_geodesic}, we get
			$$
				y_1=h_br_0(x,x')y=r_{be^{-\tau_\infty(x',x)}}(x,x')y.
			$$ 
	\end{proof}

	\begin{definition}\label{def:unstable_segments_f_diag}
		For $(x,y)\in \mathcal{R}_{\delta_g,\e}\times N$, the local unstable manifold for $f_{diag}$ centered at $(x,y)$ of radius $\xi>0$ is
		\begin{equation*}
			W_{diag,\xi}^u(x,y)=\left\{(x',y')\in \mathcal{R}_{\delta_g,\e}\times N \colon x'\in W^u_{g,\xi}(x), \ y'=r_s(x,x')y, \ s\in (-\xi,\xi)\right\}
		\end{equation*}
	\end{definition}
	
\subsection{Conditional Measures on Unstables for $f_{diag}$}\label{sec:unstable_conditionals_f_diag}

	Fix $\e, \ell, \tau_g$ as in Section \ref{sec:unstable_manifolds_f_diag}. In this section, we disintegrate $\mu\times \nu$ in local coordinates adapted to the unstable plaques for $f_{diag}$. 
	
	We will use the following reformulation of Lemma \ref{lem:disintegration_mu}:
	
	\begin{lemma}\label{lem:disintegration_mu_reformulation}
		Fix $x_0\in \hat{\mathcal{P}}_{\tau_g}$ and $r>0$. There exists a standard Borel space $\mathcal{T}$, a measure $\mu^{\mathcal{T}}$, a family $\{\mathcal{W}_t\}_{t\in \mathcal{T}}$ of nonempty, pairwise disjoint Borel sets, and a measurable function $t\mapsto x_t\in \mathcal{W}_t$ such that 
			\begin{enumerate}
			\item 
				for each $t\in \mathcal{T}$, $\mathcal{W}_t\subset W^u_{Q_{x_0}}(z)\cap G_{\ell,\tau_g}$ for some $z\in \tilde{T}_{x_0}^{\tau_g}$,
			\item 
				for each $t\in \mathcal{T}$, $\diam(\mathcal{W}_t)<r$, and 
			\item 
				for each $h\in L^1(M)$, 
				\begin{equation}\label{eq:disintegration_mu_reformulation}
					\int_{Q_{x_0}\cap \tilde{Q}_{x_0}^{\tau_g} \cap G_{\ell,\tau_g}} h(x')d\mu(x')
					= \int_{\mathcal{T}} \int_{\mathcal{W}_t} h(x')d\mu^u_{x_t}(x')d\mu^{\mathcal{T}}(t)
				\end{equation}				
			\end{enumerate}
	\end{lemma}
	\begin{proof}
		Let $U_1,...,U_m$ be a finite partition of $\tilde{Q}_{x_0}^{\tau_g}\cap Q_{x_0}$ into Borel sets of diameter less than $r$. We define 
			\begin{equation*}
				\mathcal{T}:= \bigcup_{i=1}^m \mathcal{T}_i\times \{i\}, \qquad 
				\mathcal{T}_i:= \left\{ z\in \tilde{T}_{x_0}^{\tau_g} \colon \mu^u_z(\mathcal{W}_{z,i})>0 \right\}, \qquad 
				\mathcal{W}_{z,i}:= W^u_{Q_{x_0}}(z)\cap G_{\ell,\tau_g}\cap U_i
			\end{equation*}
		For $t=(z,i)\in \mathcal{T}$, we set $z_t:=z$ and $\mathcal{W}_t:=\mathcal{W}_{z,i}$. Properties (1) and (2) are automatic. 
		
		For $E\subset \mathcal{T}$, we can write $E=\bigcup_{i=1}^m (E_i\times \{i\})$, where $E_i\subset \mathcal{T}_i$. For $E\subset \mathcal{T}$ measurable, we define  $\mu_0^{\mathcal{T}}(E)=\sum_{i=1}^m \mu^{\mathcal{T}_i}(E_i)$, where $\mu^{\mathcal{T}_i}(E_i)=\mu^{cs}_{x_0}(E_i)$. We can find a measurable function $t\mapsto x_t\in \mathcal{W_t}$ by the Jankov-von Neumann Uniformization Theorem. 
		
		Now, we define a measure $\mu^{\mathcal{T}}$ on $\mathcal{T}$ by letting $d\mu^{\mathcal{T}}(t)=\rho(z_t,x_t)d\mu_0^{\mathcal{T}}(t)$. Take $h\in L^1(M)$. Applying Lemma \ref{lem:disintegration_mu} to $h\chi_{G_{\ell,\tau_g}}$ and using \eqref{eq:base_rho_switch_basepoint}, we get that Property (3) holds. 
	\end{proof}
	
	Now, fix $x_0\in \hat{\mathcal{P}}_{\tau_g}$ and $y_0\in N$. Fix 
		\begin{equation}\label{eq:disintegration_diag_choice_r}
			0<r<\min\left\{ \tau_ge^{\frac{\lambda_g}{2}},  \left(\frac{\delta_{inj}}{2C_{\tau_\infty}}\right)^{1/\alpha_b} \right\},
		\end{equation}
	and apply Lemma \ref{lem:disintegration_mu_reformulation}. We define
		\begin{equation*}
			\Psi(t,u^{cs},x',h_s):=\left(x',r_s(x_t,x')u^{cs}y_0\right), \qquad 
			t\in \mathcal{T}, \quad
			u^{cs}\in U^{cs}, \quad 
			x'\in \mathcal{W}_t, \quad 
			h_s\in U^-.
		\end{equation*}
	By Proposition \ref{prop:unstables_f_diag}, for fixed $t\in \mathcal{T}$ and $u^{cs}\in U^{cs}$,
		$$
			\Psi(t,u^{cs},\mathcal{W}_{t},U^-)\subset \mathcal{W}^u_{diag}(x_t,u^{cs}y_0).
		$$
	By restricting $u^{cs}$ and $h_s$, we can make $\Psi$ injective:
	
	\begin{prop}\label{lem:f_diag_foliation_charts_injective}
		There exist $\delta_1,\delta_2>0$ such that $\Psi$ is injective on 
		\begin{equation*}
			\left\{ (t,u^{cs},x',h_s) \colon t\in \mathcal{T}, \ u^{cs}\in C^{cs}(\delta_1), \ x'\in \mathcal{W}_t, \ h_s\in C^-(\delta_2) \right\}
		\end{equation*}
	\end{prop}
	\begin{proof}
		Take 
			$$
				\delta_1=\frac{\delta_{inj}}{2} 
				\qquad \text{and} \qquad 
				\delta_2=\delta_{inj} e^{-C_{\tau_\infty}r^{\alpha_b}}.
			$$
		Suppose $\Psi(t_1, u^{cs}_1,x'_1,h_{s_1})=\Psi(t_2, u^{cs}_2,x'_2,h_{s_2})$. Equality of the base coordinates and the fact that the $\mathcal{W}_t$ are disjoint gives that $x_1'=x_2'$ and $t_1=t_2$. From now on, we let $x':=x_1'=x_2'$ and $t:=t_1=t_2$. Setting the fiber coordinates equal and applying \eqref{eq:r_s_horocycle_geodesic} gives that 
			\begin{equation*}
				h_{s_1e^{\tau_\infty(x',x_t)}}g_{\tau_\infty(x',x_t)}u^{cs}_1 y_0 
				=r_{s_1}(x_t,x')u^{cs}_1 y_0
				=r_{s_2}(x_t,x')u^{cs}_2 y_0
				=h_{s_2e^{\tau_\infty(x',x_t)}}g_{\tau_\infty(x',x_t)}u^{cs}_2 y_0 
			\end{equation*}
		If we can show that $h_{s_ie^{\tau_\infty(x',x_t)}}\in C^-(\delta_{inj})$ and $g_{\tau_\infty(x',x_t)}u^{cs}_i\in C^{cs}(\delta_{inj})$, then we'll be done by injectivity of the local product chart $(u^-,u^{cs})\mapsto u^-u^{cs}y_0$ on $C^-(\delta_{inj})\times C^{cs}(\delta_{inj})$ from Section \ref{sec:FiberCubes}. Combining Lemma \ref{lem:bound_tau_infty} with \eqref{eq:disintegration_diag_choice_r}, we get that 
			\begin{equation*}
				|\tau_\infty(x',x_t)|\leq C_{\tau_\infty}r^{\alpha_b} <\frac{\delta_{inj}}{2}.
			\end{equation*}
		Thus, if $h_{s_i}\in C^-(\delta_2)$, then $h_{s_ie^{\tau_\infty(x',x_t)}}\in C^-(\delta_{inj})$, and if $u^{cs}_i\in C^{cs}(\delta_1)$, then $g_{\tau_\infty(x',x_t)}u^{cs}_i\in C^{cs}(\delta_{inj})$.
	\end{proof}
	
	Now, let $\delta_1,\delta_2$ be as in Lemma \ref{lem:f_diag_foliation_charts_injective}. Let 
		$$
			\mathcal{S}=\mathcal{S}_{(x_0,y_0)}= \bigcup_{t\in \mathcal{T}} \Psi \left( t,C^{cs}(\delta_1), \mathcal{W}_t, C^-(\delta_2) \right)
				\subset \left( \tilde{Q}^{\tau_g}_{x_0} \cap Q_{x_0}\cap G_{\ell,\tau_g}\right)\times \left( C^-(\delta_{inj})C^{cs}(\delta_{inj})y_0 \right)
		$$
	We now disintegrate $\mu\times \nu$ on $\mathcal{S}$ 
	
	\begin{lemma}\label{lem:disintegrate_f_diag_unstables}
		If $h\in L^1(M\times N)$, then 
			\begin{equation*}
				\int_{\mathcal{S}} h(x',y)d(\mu\times \nu)(x',y)
				= C_1 \int_{\mathcal{T}} \int_{C^{cs}(\delta_1)} \int_{\mathcal{W}_t} \int_{-\delta_2}^{\delta_2} h(x',r_s(x_t,x')u^{cs}y_0)e^{-\tau_\infty(x',x_t)}dsd\mu^u_{x_t}(x')dm^R_{U^{cs}}(u^{cs})d\mu^{\mathcal{T}}(t)
			\end{equation*}
	\end{lemma} 
	\begin{proof}
		We want to change variables using $(x',y)=(x',r_s(x_t,x')u^{cs}y_0)=\Psi(t,u^{cs},x',h_s)$. By Lemma \ref{lem:disintegration_mu_reformulation}(3), we get that $d\mu(x')=d\mu^u_{x_t}(x')d\mu^{\mathcal{T}}(t)$. Now, we compute $d\nu(y)$ in these coordinates. Recall from Section \ref{sec:FiberCubes} that $d\nu(h_su^{cs})=C_1dm_{U^-}(h_s) dm^R_{U^{cs}}(u^{cs})$. So, 
			\begin{align*}
				d\nu(r_s(x_t,x')u^{cs}y_0)
				&=d\nu(r_s(x_t,x')u^{cs})\\
				&=d\nu\left( h_{se^{\tau_\infty(x',x_t)}}g_{\tau_\infty(x',x_t)}u^{cs} \right)\\
				&=C_1 dm_{U^-}(h_{se^{\tau_\infty(x',x_t)}}) dm^R_{U^{cs}}(  g_{\tau_\infty(x',x_t)}u^{cs})
			\end{align*}
		So, using the definitions of $m_{U^-}$ and $m^{R}_{U^{cs}}$ from Section \ref{sec:FiberCubes}, we get that 
			$
				dm_{U^-}(h_{se^{\tau_\infty(x',x_t)}})  = e^{\tau_{\infty}(x',x_t)}ds
			$
		and 
			$
				dm^R_{U^{cs}}(  g_{\tau_\infty(x',x_t)}u^{cs}) = e^{-2\tau_{\infty}(x',x_t)} dm^R_{U^{cs}}(u^{cs}).
			$
		So, 
			$$
				d\nu(r_s(x_t,x')u^{cs}y_0) 
				= C_1 e^{-\tau_\infty(x',x_t)} ds dm^R_{U^{cs}}(u^{cs}).
			$$
		We now change variables using $\Psi$ and apply Fubini's theorem to get the conclusion of the Lemma. 
	\end{proof}	
	
	\begin{definition}\label{def:f_diag_unstable_conditional_measure}
		Take $\xi_1=\min\{ \xi_0, \delta_2 \}$ where $\xi_0$ is from Lemma \ref{lem:bound_tau_infty}. Take $\xi\in(0,\xi_1)$, for all $(x,y)\in G_{\ell,\tau_g}\times N$, we can define a measure $m^{u,diag}_{(x,y)}$ on $W^u_{diag,\xi}\cap G_{\ell,\tau_g}$ by for $E\subset W^u_{diag,\xi}\cap G_{\ell,\tau_g}$,
			$$
				m^{u,diag}_{(x,y)}(E)=\int_{W^u_{g,\xi}(x)\cap G_{\ell,\tau_g}} \int_{-\xi}^\xi \chi_{E}\left( x',r_s(x,x')y \right)e^{-\tau_\infty(x',x)} ds d\mu^u_x(x')
			$$
	\end{definition}
	
	The following is immediate from Lemma \ref{lem:disintegrate_f_diag_unstables} and Definition \ref{def:f_diag_unstable_conditional_measure}
	
	\begin{prop}\label{prop:decompose_unstable_diagonal_conditionals}
		Fix $\xi\in(0,\xi_1)$ where $\xi_1$ is from Definition \ref{def:f_diag_unstable_conditional_measure}. For $(x,y)\in G_{\ell,\tau_g}\times N$ and $x'\in W^u_{g,\xi}(x)$, we define a measure $\nu^u_{x,x',y}$ on $\{r_s(x,x')y \colon s\in (-\xi,\xi)\}$ as follows. Take $E\subset \{r_s(x,x')y \colon s\in (-\xi,\xi)\}$, and let $I_E=\{s\in (-\xi,\xi)\colon r_s(x,x')y\in E\}$. We define 
			$$
				\nu^u_{x,x',y}(E)=\int_{I_E} e^{-\tau_{\infty}(x',x)}ds.
			$$
		Then, 
			$$dm^{u,\diag}_{(x,y)}(x',y')=d\nu^u_{x,x',y}(y')d\mu^u_x(x').$$
	\end{prop}
	
	\begin{prop}\label{prop:radon-nikodym_unstable_conditional_wrt_riem_vol}
		Take $\delta>0$.  Take $(x,y)\in G_{\ell,\tau_g}\times N$. If $\xi\in (0,\xi_1)$ (where $\xi_1$ is from Definition \ref{def:f_diag_unstable_conditional_measure}) and $x'\in W^u_{g,\xi}(x)\cap G_{\ell,\tau_g}$, then for all $s\in (-\xi,\xi)$, 
			$$
				\frac{d\nu^u_{x,x',y}}{ds}(r_s(x,x')y)=e^{-\tau_\infty(x',x)}
			$$
		Moreover, for all $\delta>0$, there exists $\xi_2=\xi_2(\delta,\e,\ell,\delta_g,\tau_g)\in (0,\xi_1)$ independent of $x,y$ such that if $\xi\in (0,\xi_2)$,
			$$
				\frac{d\nu^u_{x,x',y}}{ds}(r_s(x,x')y)\in (1-\delta,1+\delta)
			$$			
	\end{prop}
	\begin{remark}\label{remark:radon-nikdodym_unstable_conditional_wrt_riem_vol}
		In Proposition \ref{prop:radon-nikodym_unstable_conditional_wrt_riem_vol}, we can take $\xi_2=\min(\xi_1,\xi_3)$ where $\xi_3$ is from Lemma \ref{lem:bound_e^tau_infty}.
	\end{remark}

\section{Unstable Manifolds for $f$}\label{sec:unstable_f}
	In this section we define the unstable manifolds for $f$ using the unstable manifolds for $f_{diag}$. Throughout this section fix
		\begin{equation}\label{eq:unstable_f_e_restricitons}
			0<\e<\min\left\{\frac{\lambda_2}{4}, \frac{\alpha_b\lambda_g}{14}, \frac{\lambda_g\beta}{18}\right\},
		\end{equation}
	and $\ell\geq 1$, $\delta_g$ as in Section \ref{sec:basedynamics}, and $\tau_g\in (0,1)$. The arguments in this section only require that $g$ is a $C^{1+\alpha_g}$ ergodic diffeomorphism that preserves the smooth measure $\mu$ and that has at least one nonzero Lyapunov exponent and that $A:M\to PSL_2\R$ is $C^{\alpha}$ and has a nonzero Lyapunov exponent. 

	After defining $P$ on a set of measure 0, we can define a measure theoretic isomorphism 
		$$
			\Phi:M\times N\to M\times N, \qquad \Phi(x,y)=(x,P(x)y).
		$$
	Observe that $\Phi^{-1}(x,y)=(x,P(x)^{-1}y)$, and 
		\begin{equation*}
			f\circ \Phi= \Phi\circ f_{diag}.
		\end{equation*}
		
	As in \eqref{eq:unstable_fdiag_characterization}, for $(x,y)\in \mathcal{R}_{\delta_g,\e}\times N$, the local unstable manifold for $f$ centered at $(x,y)$ can be characterized by 
		\begin{equation}\label{eq:unstable_f_characterization}
			W^u_{f}(x,y)=\left\{ (x',y')\in M\times N \colon \limsup_{n\to \infty} \frac{1}{n}\log d_{M\times N}\left( f^{-n}(x,y), f^{-n}(x',y')\right) <0 \right\}.
		\end{equation}
	
	The map $\Phi$ takes unstable manifolds $W^u_{diag}$ for $f_{diag}$ to unstable manifolds $W^u_f$ for $f$:
	
	\begin{lemma}\label{lem:Phi_unstables_f_diag_to_unstables_f}
		Let $(x,y), (x',y')\in G_{\ell,\tau_g}\times N$ with $d(x,x')<r_p$, where $r_p$ is from Proposition \ref{prop:P_holder_on_pesin_unstables}. Then, 
			\begin{equation}\label{eq:Phi_unstables_f_diag_to_unstables_f}
				(x',y')\in W^u_{diag}(x,y)
				\qquad \Longleftrightarrow \qquad 
				\Phi(x',y')\in W^u_f(\Phi(x,y)).
			\end{equation}
	\end{lemma}
	\begin{proof}
		From \eqref{eq:unstable_fdiag_characterization} and  \eqref{eq:unstable_f_characterization}, we have that 
			\begin{equation*}
				(x',y')\in W^u_{diag}(x,y) 
				\quad \Longleftrightarrow \quad 
				x'\in W^u_g(x)  \text{ and } \limsup_{n\to \infty} \frac{1}{n} \log d_N(B_{-n}(x)y,B_{-n}(x')y')<0
			\end{equation*}
		and
			\begin{equation*}
				\Phi(x',y')\in W^u_{f}(\Phi(x,y))
				\quad \Longleftrightarrow \quad 
				x'\in W^u_g(x)  \text{ and } \limsup_{n\to \infty} \frac{1}{n} \log d_N(A_{-n}(x)P(x)y,A_{-n}(x')P(x')y')<0.
			\end{equation*}
		Let 
			$$
				P_n=P(g^{-n}x), \quad 
				P_n'=P(g^{-n}x'), \quad
				z_n=B_{-n}(x)y, \quad 
				z_n'=B_{-n}(x')y'.
			$$
		Applying \eqref{eq:osceledets_pesin_diagonalize_A} and using the triangle inequality and the right-invariance of the Riemannian metric, we get that 
			\begin{equation}\label{eq:lem:Phi_unstables_f_diag_to_unstables_f_bound1}
				d_N(P_nz_n, P_n'z_n')
				\leq \|\Ad(P_n)\|_{coord} d_N(z_n,z_n')+d_{PSL_2\R}(P_n,P_n')
				\leq C\|P_n\|^2 d_N(z_n,z_n')+d_{PSL_2\R}(P_n,P_n'),
			\end{equation}
		where the constant $C\geq 1$ is uniform. (In fact, $C$ depends only on the geometry of $PSL_2\R$.) Applying Proposition \ref{prop:P_holder_on_pesin_unstables} and using \eqref{eq:unstable_f_e_restricitons}, we get that 
			\begin{equation}\label{eq:lem:Phi_unstables_f_diag_to_unstables_f_d(P_n,P_n')}
				\lim_{n\to \infty} \frac{1}{n} \log d_{PSL_2\R}(P_n,P_n') <0.
			\end{equation}
		By Theorem \ref{thm:Osceledets_Pesin_Reduction}, $P$ is tempered, i.e. 
			\begin{equation}\label{eq:lem:Phi_unstables_f_diag_to_unstables_f_Ptempered}
				\lim_{n\to \infty} \frac{1}{n}\log \|P_n\|=\lim_{n\to \infty} \frac{1}{n}\log \|P_n^{-1}\|=0.
			\end{equation}
		Combining \eqref{eq:lem:Phi_unstables_f_diag_to_unstables_f_bound1}, \eqref{eq:lem:Phi_unstables_f_diag_to_unstables_f_d(P_n,P_n')}, and \eqref{eq:lem:Phi_unstables_f_diag_to_unstables_f_Ptempered}, we get the forward direction of \eqref{eq:Phi_unstables_f_diag_to_unstables_f}. 
		
		The same argument we gave to get \eqref{eq:lem:Phi_unstables_f_diag_to_unstables_f_bound1} gives 
			\begin{equation*}
				d_N(z_n, z_n')
				=d_N(P_n^{-1}P_nz_n, (P_n')^{-1}P_n'z_n')
				\leq C\|P_n^{-1}\|^2 d_N(P_nz_n,P_n'z_n')+d_{PSL_2\R}(P_n^{-1},(P_n')^{-1}).
			\end{equation*}			
		Since $d_{PSL_2\R}$ is right-invariant, we get that 
			\begin{equation*}
				d_{PSL_2\R}(P_n^{-1},(P_n')^{-1})
				= d(P_n^{-1}P_n',e)
				= d(P_n^{-1}P_n'P_{n}^{-1},P_n^{-1})
				\leq \|\Ad(P_n^{-1})\|_{coord} d(P_n',P_n)
				\leq C\|P_n^{-1}\|^2 d(P_n,P_n'). 
			\end{equation*}
		Thus, 
			\begin{equation}\label{eq:lem:Phi_unstables_f_diag_to_unstables_f_bound2}
				d_N(z_n, z_n') \leq C\|P_n^{-1}\|^2 \left( d_N(P_nz_n,P_n'z_n')+d_{PSL_2\R}(P_n,P_n') \right).
			\end{equation}
		Combining \eqref{eq:lem:Phi_unstables_f_diag_to_unstables_f_d(P_n,P_n')}, \eqref{eq:lem:Phi_unstables_f_diag_to_unstables_f_Ptempered}, and \eqref{eq:lem:Phi_unstables_f_diag_to_unstables_f_bound2}, we get the reverse direction of \eqref{eq:Phi_unstables_f_diag_to_unstables_f}. 
	\end{proof}

	\begin{prop}\label{prop:unstables_f}
		If $(x,y)\in G_{\ell,\tau_g}\times N$, then 
		\begin{align*}
			&W^u_f(x,y) \cap \left( (G_{\ell,\tau_g}\cap B(x,r_p))\times N \right)\\
			&\hspace{0.5in} = \left\{ (x',y')\in (G_{\ell,\tau_g}\cap B(x,r_p))\times N \colon x'\in W^u_g(x), y'=P(x')r_s(x,x')P(x)^{-1}y, s\in \R \right\}
		\end{align*}
	\end{prop}
	\begin{proof}
		Immediate from combining Proposition \ref{prop:unstables_f_diag} and Lemma \ref{lem:Phi_unstables_f_diag_to_unstables_f}.
	\end{proof}
	
	By Lemma \ref{lem:unstables_g}, we can find $\xi_f>0$ depending only on $G_{\ell,\tau_g}$ such that for all $x\in G_{\ell,\tau_g}$ and $\xi\in (0,\xi_f)$, 
		$
			W^u_{g,\xi}(x)\subset B(x,r_p).
		$
	For $x\in G_{\ell,\tau_g}$ and $\xi\in (0,\xi_f)$, define 
		\begin{equation}\label{eq:definition_unstable_f_radius_xi_using_f_diag}
			W^u_{f,\xi}(x,y):= \Phi(W^u_{diag,\xi}(x,P(x)^{-1}y)\cap (G_{\ell,\tau_g}\times N)).
		\end{equation}
	Equivalently, 
		\begin{equation}\label{eq:definition_unstable_f_radius_xi}
			W^u_{f,\xi}(x,y)= \left\{ \left( x',P(x')r_s(x,x')P(x)^{-1}y \right) \colon x'\in W^u_{g,\xi}(x)\cap G_{\ell,\tau_g}, \ |s|<\xi  \right\}.
		\end{equation}
	
\subsection{Conditional measures on unstable manifolds for $f$}\label{sec:unstable_conditional_measure_f}

	Fix $0<\xi<\min\{\xi_1,\xi_f\}$, where $\xi_1$ is from Definition \ref{def:f_diag_unstable_conditional_measure}. For $(x,y)\in G_{\ell,\tau_g}\times N$, we define a measure
		\begin{equation}\label{eq:conditionals_f_define}
			m^u_{(x,y)}:=\Phi_* m^{u,diag}_{(x,P(x)^{-1}y)}.
		\end{equation}
	Then combining Lemma \ref{lem:Phi_unstables_f_diag_to_unstables_f} with the fact that $\Phi$ and $\Phi^{-1}$ preserve $m=\mu\times \nu$ and the results of Section \ref{sec:unstable_conditionals_f_diag}, we get that
	
	\begin{prop}\label{prop:conditional_unstable_f}
		After normalization on each plaque, the measures $m^u_{(x,y)}$ defined in \eqref{eq:conditionals_f_define} are the conditional probabilities along the local unstable plaques of $f$. 
	\end{prop}

\section{Exponential Equidistribution of Unstables for $f$} \label{sec:equidistribution_f}

	In this section, we show that typical unstable plaques for $f$ equidistribute on exponentially small product cubes. This is the main ingredient in the proof of Theorem \ref{thm:Main_f_bernoulli}.  As in previous sections, we assume that $g$ is a $C^{1+\alpha_g}$ ergodic diffeomorphism that preserves the smooth measure $\mu$ and that has at least one nonzero Lyapunov exponent and that $A:M\to PSL_2\R$ is $C^{\alpha}$ and has a nonzero Lyapunov exponent. We will be explicit when we use any additional assumptions, in particular exponential mixing of $g$. 
	
	Throughout this section, fix 
		\begin{equation}\label{eq:equidistribution_e_restricitons}
			0<\e<\min\left\{\frac{\lambda_2}{4}, \frac{\alpha_b\lambda_g}{14}, \frac{\lambda_g\beta}{18}, \frac{\alpha_1\lambda_2}{16}\right\},
		\end{equation}
	where $\alpha_1$ is the polynomial rate of the equidistribution of the horocycle flow, see \eqref{eq:thm:flaminio-forni} from Theorem \ref{thm:FlaminioForni}. Fix $\delta_g$ as in Section \ref{sec:basedynamics}.

	\begin{theorem}\label{thm:exp_equidistribution_unstables_f}
		Suppose that $g:M\to M$ is exponentially mixing. Take $\e_1,\e_4,\e_5>0$ and fix 
			\begin{equation*}
				0<\e_{12}<\min\left\{ \left( \frac{\e_1}{2} \right)^{1/10}, \left( \frac{\e_4}{4} \right)^{1/10},  \left( \frac{\e_5}{4} \right)^{1/10} \right\}.
			\end{equation*}
		Fix $\ell$ sufficiently large and $\tau_g$ sufficiently small. There exists $\eta_0>0$ such that for all $\eta\in \left(0,\eta_0\right)$, $\exists \xi_0>0$ such that $\forall \xi\in (0,\xi_0)$, there exists $N_0\in \N$ such that for all $n>N_0$, there are sets $Q,Q_n\subset M$ with 
			\begin{equation*}
				\mu(Q)>1-\e_1, \qquad 
				\mu(Q_n)>1-\frac{\e_1\e_4\e_5}{64}-300\e_{12}^{10^{10}/16}
			\end{equation*}
		such that the following holds: Let 
			$$
				m_n:=  \left\lceil \frac{n}{\e_{12}}\right\rceil,
			$$
		and let $\{B_i\}_{i\in I_1}$ be a family of pairwise disjoint sets in $M$ as in Theorem \ref{thm:exp_equidistribution_g_expmixing} (applied with $\e=\e_{12}$ and at time $m_n$). Let $D\subset N$ be a cube of radius $e^{-\eta n}$.
		
		For all $(x,y)\in Q\times N$, there is a subset $J=J(x)\subset I_1$ with $\mu\left( \cup_{i\in J} B_i \right)>1-\e_4$ such that for all $i\in J$
			\begin{equation}\label{eq:thm_exp_equidistribution_unstables_f}
				m^u_{(x,y)}\left( f^{-n}(B_i\times D) \cap W^u_{f,\xi}(x,y) \cap (Q_n\times N) \right)
				\in (1-\e_5,1+\e_5)\mu(B_i)\nu(D)m^u_{(x,y)}(W^u_{f,\xi}(x,y))
			\end{equation}
		Moreover, the subset $J=J(x)$ is independent of the choice of cube $D\subset N$ of radius $e^{-\eta n}$. 
	\end{theorem}
	
	The proof separates the base and fiber coordinates. Proposition \ref{prop:exp_equidistribution_fiber} gives equidistribution in the fiber, while Theorem \ref{thm:exp_equidistribution_g_expmixing} gives equidistribution in the base. In Section \ref{sec:proof_equidistribution}, we combine these two results to prove Theorem \ref{thm:exp_equidistribution_unstables_f}.

\subsection{Fiberwise Equidistribution}\label{sec:exp_equidistribution_fiber}
	In this section, we prove equidistribution in the fiber. This section does not assume that $g$ is exponentially mixing. Recall from Theorem \ref{thm:Osceledets_Pesin_Reduction} that for $n>0$,
		$$
			\tau_n(x)=\sum_{j=0}^{n-1}\tau(g^jx) 
			\qquad \text{ and } \qquad 
			\int_M \tau d\mu=\lambda_2.
		$$
	By the Birkhoff Ergodic Theorem, $\frac{\tau_n(x)}{n} \to \lambda_2$ for almost every $x\in M$. We can therefore apply Egorov's Theorem to get that for all $\delta_1>0$, we can find a measurable set $Q_1=Q_1(\delta_1)\subset M$ and $N_1=N_1(\delta_1)\in \N$ such that 
		\begin{equation}\label{eq:equidistribution_Q1}
			\mu(Q_1)>1-\delta_1
			\qquad \text{ and } \qquad 
			\tau_n(x)\geq \frac{\lambda_2n}{4} 
			\quad \forall x\in Q_1, \ n\geq N_1.
		\end{equation} 
		
	Throughout this section, we will fix $\ell\geq 1$ and $\tau_g\in (0,1)$.
		
	\begin{prop}\label{prop:exp_equidistribution_fiber}
		There exists $\eta_0=\eta_0(\e)>0$ such that if $\eta\in (0,\eta_0)$, then for all $0<\e_7<1$, there exists $\xi_0>0$ such that for all $\xi \in (0,\xi_0)$, there exists $N_0\in \N$ such that $\forall n >N_0$ the following holds:
		
		Let $D$ be a cube in $N$ of radius $e^{-\eta  n}$. Take 
			$$
				(x_0,y_0)\in G_{\ell,\tau_g}\times N, 
				\qquad
				(x,y)\in W^u_{f,\xi}(x_0,y_0)\cap \left( G_{\ell,\tau_g}\times N \right),
				\qquad 
				x'\in W^u_{g,\xi}(x_0)\cap G_{\ell,\tau_g}\cap Q_1.
			$$
		Set
			$$
				L(x')=L(x_0,x',\xi,y_0):= \left\{ r_s(x_0,x')P(x_0)^{-1}y_0 \colon |s|<\xi \right\}.
			$$
		Then, 
			\begin{equation}\label{eq:exp_equidistribution_fiber}
				\mathcal{I}_3:= \int_{L(x')} \chi_{D}\left( A_n(x')P(x')y' \right) d\nu^u_{x,x',P(x)^{-1}y}(y')
					\in \left( 1-\e_7, 1+\e_7 \right) \nu_{x,x',P(x)^{-1}y}^u(L(x')) \nu(D)
			\end{equation}
	\end{prop}

	The proof of Proposition \ref{prop:exp_equidistribution_fiber} relies on the fact that the horocycle flow equidistributes on $N=PSL_2\R/\Gamma$ at a polynomial rate:
	
	\begin{theorem}[{\cite[Theorem 1.5]{Flaminio-Forni}}]\label{thm:FlaminioForni}
		There exists a constant $\alpha_1>0$ and a constant $C_4>0$ such that for any $\phi \in C^4(N)$, $T>0$ and $y\in N$, 
			\begin{equation}\label{eq:thm:flaminio-forni}
				\left| \frac{1}{T} \int_0^T \phi(h_s y)ds - \int_N \phi d\nu \right| \leq C_4 T^{-\alpha_1} \| \phi \|_{C^4}
			\end{equation}
	\end{theorem}

	\begin{proof}[Proof of Proposition \ref{prop:exp_equidistribution_fiber}]
		Fix 
			\begin{equation*}
				0<\e_7<1, \qquad 
				0<\eta<\eta_0:=\frac{\alpha_1\lambda_2}{44}-\frac{4\e}{11}, \qquad 
				0<\xi<\xi_0:=\min\left\{\frac{\xi_2}{4},\xi_f\right\},
			\end{equation*}
		where $\xi_2$ is from Proposition \ref{prop:radon-nikodym_unstable_conditional_wrt_riem_vol} applied with $\delta=1$ and $\xi_f$ is defined right before \eqref{eq:definition_unstable_f_radius_xi_using_f_diag}. Fix 
			\begin{equation}\label{eq:exp_equidistribution_fiber_N_0}
				n>N_0
				:=\max\left\{  N_1, 
										\frac{1}{\eta}\ln\left( \frac{2}{\delta_{inj}} \right),
										\frac{1}{\eta} \ln \left( \frac{2K_4}{C_1\e_7} \right),
										\frac{4}{\alpha_1 \lambda_2}\ln \left( \frac{C_4C_7C_8 \ell^4(2\xi)^{-\alpha_1}}{4C_1\e_7} \right)
									\right\},
			\end{equation}
		where $C_1$ is from \eqref{eq:measure_delta_cube}, $K_4$ is from Proposition \ref{prop:upper_bound_difference_measure_cubes_SL2R}, $C_4$ and $\alpha_1$ are from Theorem \ref{thm:FlaminioForni}, $C_7$ is from \eqref{eq:equidistribution_fiber_smooth_approx_bounds}, and $C_8:=e^{\alpha_1C_{\tau_\infty}\xi^{\alpha_b}}$. Before we bound $\mathcal{I}_3$, we pause to give the consequences of \eqref{eq:exp_equidistribution_fiber_N_0} that we will use in the rest of the proof: 
		
			\begin{enumerate}[(i)]
			\item \label{eq:equidistribution_fiber_N_0_consequence_1}
				$\displaystyle \tau_n(x')>\frac{\lambda_2 n}{4}$ by \eqref{eq:equidistribution_Q1}.
			\item \label{eq:equidistribution_fiber_N_0_consequence_2}
				$\displaystyle 2e^{-\eta n}<\delta_{inj}$. This allows us to apply Corollary \ref{cor:smoothing_lemma_N} with $\e=\delta=e^{-\eta n}$. 
			\item \label{eq:equidistribution_fiber_N_0_consequence_3}
				$\displaystyle \frac{K_4}{C_1}e^{-\eta n}<\frac{\e_7}{2}$
			\item \label{eq:equidistribution_fiber_N_0_consequence_4}
				$\displaystyle \frac{C_4C_7 C_8 \ell^4}{8C_1} (2\xi)^{-\alpha_1} e^{-\frac{\alpha_1\lambda_2}{4}n}<\frac{\e_7}{2}$. 
			\end{enumerate}
			
		We now need to prove that \eqref{eq:exp_equidistribution_fiber} holds. We'll do this by first considering the simpler case where $(x,y)=(x_0,y_0)$. 
		
	\textit{Step 1:} in this step, we will assume $(x,y)=(x_0,y_0)$.
		
		By Proposition \ref{prop:decompose_unstable_diagonal_conditionals}, we get that 
			\begin{equation}\label{eq:fiberslices_measure_L(x')}
				\nu^u_{x_0,x',P(x_0)^{-1}y_0}(L(x'))=2\xi e^{-\tau_\infty(x',x_0)},
			\end{equation}
		and 
			\begin{equation*}
				\mathcal{I}_3
				=e^{-\tau_\infty(x',x_0)}\int_{-\xi}^\xi \chi_D(A_n(x')P(x')r_s(x_0,x')P(x_0)^{-1}y_0)ds.
			\end{equation*}
		Thus, showing that \eqref{eq:exp_equidistribution_fiber} holds is equivalent to showing 
			\begin{equation}\label{eq:exp_equidistribution_fiber2}
				\frac{1}{2\xi}\int_{-\xi}^\xi \chi_D(A_n(x')P(x')r_s(x_0,x')P(x_0)^{-1}y_0)ds
				\in (1-\e_7, 1+\e_7)\nu(D).
			\end{equation}
			
		We now change variables in the integral in \eqref{eq:exp_equidistribution_fiber2}. Combining \eqref{eq:r_s_horocycle_geodesic} with 
			\begin{equation*}
				A_n(x')P(x')=P(g^nx')B_n(x')=P(g^nx')g_{\tau_n(x')}, 
				\qquad 
				g_th_s=h_{se^{2t}}g_t,
			\end{equation*}
		we get that 
			\begin{equation*}
				A_n(x')P(x')r_s(x_0,x')P(x_0)^{-1}y_0
				= P(g^nx') h_{se^{\tau_\infty(x',x_0)+2\tau_n(x')}}g_{\tau_\infty(x',x_0)+\tau_n(x')} P(x_0)^{-1} y_0.
			\end{equation*}
		Setting 
			\begin{equation}\label{eq:equidistribution_fiber_choice_T_ybar}
				T:= 2\xi e^{\tau_\infty(x',x_0)+2\tau_n(x')}, 
				\qquad 
				\bar{y}:= h_{-T/2}g_{\tau_\infty(x',x_0)+\tau_n(x')} P(x_0)^{-1} y_0,
			\end{equation}
		and changing variables using $t=(s+\xi)e^{\tau_\infty(x',x_0)+2\tau_n(x')}$, we get that 
			\begin{equation}\label{eq:exp_equidistribution_fiber_change_variables}
				\frac{1}{2\xi}\int_{-\xi}^\xi \chi_D(A_n(x')P(x')r_s(x_0,x')P(x_0)^{-1}y_0)ds
				= \frac{1}{T}\int_0^T \chi_{P(g^nx')^{-1}D}(h_t\bar{y})dt.
			\end{equation}
			
		We now estimate the horocycle average on the right hand side of \eqref{eq:exp_equidistribution_fiber_change_variables} using Theorem \ref{thm:FlaminioForni}. By \eqref{eq:equidistribution_fiber_N_0_consequence_2}, we can apply Corollary \ref{cor:smoothing_lemma_N} to $D$ to get smooth functions $\hat\phi^-\leq \chi_D\leq \hat\phi^+$. If we pre-compose $\hat\phi^\pm$ with left multiplication by $P(g^nx')$, we get smooth functions 
			\begin{equation}\label{eq:equidistribution_fiber_smooth_approx_above_below}
				\phi^- \leq \chi_{P(g^nx')^{-1} D} \leq \phi^+
			\end{equation}
		such that
			\begin{equation}\label{eq:equidistribution_fiber_smooth_approx_bounds}
				\int_N \left| \phi^{\pm}-\chi_{P(g^nx')^{-1} D} \right|d\nu 
				\leq \frac{K_4}{C_1}e^{-\eta n}\nu(D),
				\qquad 
				\| \phi^\pm \|_{C^4}\leq C_7 \ell^4e^{(8\eta+4\e)n},
			\end{equation}
		for some constant $C_7$ independent of $n$. We briefly explain how the $C^4$ bound is obtained from Corollary \ref{cor:smoothing_lemma_N}(4). Recall that derivative of left multiplication by $P(g^nx')$ (evaluated in a right-invariant frame) is given by $\Ad(P(g^nx'))$. Then the same argument we used to obtain \eqref{eq:bound_norm_P(g^-kx)} along with the chain rule gives that
			$$
				\|\phi^\pm\|_{C^4}
				\leq C\ell^4 e^{4n\e} \|\hat\phi^\pm \|_{C^4}
				\leq C_7 \ell^4 e^{(8\eta+4\e)n}.
			$$
		We now apply Theorem \ref{thm:FlaminioForni} to $\phi^\pm$ to get 
			\begin{equation*}
				E(\phi^\pm)- C_4 T^{-\alpha_1}\left\|\phi^\pm\right\|_{C^4}
				\leq \frac{1}{T}\int_{0}^{T} \phi^\pm \left(h_t \bar y\right) d t
				\leq E(\phi^\pm)+ C_4 T^{-\alpha_1}\left\|\phi^\pm\right\|_{C^4}.				
			\end{equation*}
		Since $\nu$ is Haar measure (and thus is invariant under left multiplication by $P(g^nx')^{-1}$), we get 
			$$
				E(\phi^\pm)
				%=\nu(P(g^nx')^{-1} D)+\int_N \left( \phi^\sigma - \chi_{P(g^nx')^{-1} D} \right)d\nu
				=\nu(D)+ \int_N \left( \phi^\pm - \chi_{P(g^nx')^{-1} D} \right)d\nu.
			$$
		Combining this with \eqref{eq:equidistribution_fiber_smooth_approx_above_below} and \eqref{eq:equidistribution_fiber_smooth_approx_bounds}, we get that 
			\begin{equation*}
				\left|\frac{1}{T}\int_0^T \chi_{P(g^nx')^{-1}D}(h_t\bar{y})dt -\nu(D)\right| 
				\leq 
				\frac{K_4}{C_1}e^{-\eta n}\nu(D) + C_7 \ell^4T^{-\alpha_1}e^{(8\eta+4\e)n}.
			\end{equation*}
		Combining this with \eqref{eq:exp_equidistribution_fiber_change_variables}, we get that showing that \eqref{eq:exp_equidistribution_fiber2} holds is equivalent to showing that 
			\begin{equation*}
			 	\frac{K_4}{C_1}e^{-\eta n}\nu(D) + C_4C_7\ell^4 T^{-\alpha_1}  e^{(8\eta+4\e)n} <\e_7\nu(D)
			\end{equation*}
		To do this, by \eqref{eq:equidistribution_fiber_N_0_consequence_3}, it suffices to show that 
			\begin{equation}\label{eq:exp_equidistribution_fiber3}
				C_4C_7\ell^4 T^{-\alpha_1}  e^{(8\eta+4\e)n} <\frac{\e_7}{2}\nu(D).
			\end{equation}		
		By our choice of $\xi_0$, Lemma \ref{lem:bound_tau_infty} and Definition \ref{def:f_diag_unstable_conditional_measure}, we have that 
			\begin{equation*}
				e^{-\alpha_1 \tau_\infty(x',x_0)}
				\leq e^{\alpha_1|\tau_\infty(x',x_0)|}
				\leq e^{\alpha_1C_{\tau_\infty}\xi^{\alpha_b}}=:C_8.
			\end{equation*}
		Plugging this and \eqref{eq:equidistribution_fiber_N_0_consequence_1} into  \eqref{eq:equidistribution_fiber_choice_T_ybar}, we get that
			\begin{equation*}
				T^{-\alpha_1}
				= (2\xi )^{-\alpha_1} e^{-\alpha_1\tau_\infty(x',x_0)}e^{-2\alpha_1\tau_n(x')}
				\leq C_8 (2\xi )^{-\alpha_1} e^{-2\alpha_1\tau_n(x')}
				\leq C_8 (2\xi )^{-\alpha_1} e^{-\frac{\alpha_1\lambda_2}{2}n}.
			\end{equation*}			
		Thus, 
			\begin{equation*}
				C_4C_7\ell^4 T^{-\alpha_1}  e^{(8\eta+4\e)n}
				\leq C_4C_7 C_8 \ell^4 (2\xi)^{-\alpha_1} e^{-\left( \frac{\alpha_1\lambda_2}{2}-8\eta-4\e \right)n}.
			\end{equation*}
		Recall from \eqref{eq:lower_bound_measure_cubes_SL2R} that 
			$
				8C_1e^{-3\eta n}\leq \nu(D).
			$
		We therefore get that 
			\begin{equation*}
				C_4C_7\ell^4 T^{-\alpha_1}  e^{(8\eta+4\e)n}
				\leq \frac{C_4C_7 C_8 \ell^4}{8C_1} (2\xi)^{-\alpha_1} e^{-\left( \frac{\alpha_1\lambda_2}{2}-11\eta-4\e \right)n}\nu(D).
			\end{equation*}			
		Since $\eta<\eta_0$, and $\frac{\alpha_1\lambda_2}{2}-11\eta_0-4\e=\frac{\alpha_1\lambda_2}{4}$, we get that 
			\begin{equation*}
				C_4C_7\ell^4 T^{-\alpha_1}  e^{(8\eta+4\e)n}
				\leq \frac{C_4C_7 C_8 \ell^4}{8C_1} (2\xi)^{-\alpha_1} e^{-\frac{\alpha_1\lambda_2}{4}n}.
			\end{equation*}			
		Applying \eqref{eq:equidistribution_fiber_N_0_consequence_4}, we get that \eqref{eq:exp_equidistribution_fiber3} holds. 
	
	\textit{Step 2: Changing the reference point for the fiber slice}. 
		Now, let $(x,y)\in W^u_{f,\xi}(x_0,y_0)\cap (G_{\ell,\tau_g}\times N)$ be arbitrary. By Step 1, we have that 
			\begin{equation}\label{eq:exp_equidistribution_fiber_centered}
				\mathcal{I}_3:= \int_{L(x')} \chi_{D}\left( A_n(x')P(x')y' \right) d\nu^u_{x_0,x',P(x_0)^{-1}y_0}(y')
					\in \left( 1-\e_7, 1+\e_7 \right) \nu_{x_0,x',P(x_0)^{-1}y_0}^u(L(x')) \nu(D).				
			\end{equation}
		To show that \eqref{eq:exp_equidistribution_fiber} holds, we want to replace $\nu^u_{x_0,x',P(x_0)^{-1}y_0}$ by $\nu^u_{x,x',P(x)^{-1}y}$ in \eqref{eq:exp_equidistribution_fiber_centered}. This is just changing our parameterization of $L(x')$ from one with $(x_0,y_0)$ as the basepoint to one with $(x,y)$ as the basepoint. 
		
		By \eqref{eq:definition_unstable_f_radius_xi}, we have that for some $s_0\in (-\xi,\xi)$, 
			\begin{equation}\label{eq:equidistribution_fiber_y_in_terms_of_y_0}
				P(x)^{-1}y=r_{s_0}(x_0,x)P(x_0)^{-1}y_0.
			\end{equation}
		Let 
			\begin{equation*}
				a:=\tau_\infty(x,x_0), \qquad 
				b:= \tau_\infty(x',x).
			\end{equation*}
		The definition of $r_s$ gives that 
			\begin{equation*}
				r_t(x,x')r_{s_0}(x_0,x)
				=r_s(x_0,x'), 
				\qquad 
				s=te^{-a}+s_0e^b.
			\end{equation*}
		Combining this with \eqref{eq:equidistribution_fiber_y_in_terms_of_y_0}, we get
			\begin{equation*}
				L(x')
				:= \left\{ r_s(x_0,x')P(x_0)^{-1}y_0 \colon |s|<\xi \right\}
				= \left\{ r_t(x,x')P(x)^{-1}y \colon t\in (-\xi e^{a}-s_0e^{a+b}, \xi e^a-s_0e^{a+b}) \right\}.
			\end{equation*}
		We now show that $(-\xi e^{a}-s_0e^b, \xi e^a-s_0e^b)\subset (-\xi_1,\xi_1)$. Since $\xi<\xi_0\leq \frac{\xi_2}{4}\leq \xi_3$, we can apply Lemma \ref{lem:bound_e^tau_infty} to get that
			\begin{equation*}
				\left|\pm\xi e^{a}-s_0e^b\right| \leq 2 (\xi+|s_0|)\leq 4\xi < \xi_1.
			\end{equation*}
		The last inequality comes from the fact that $\xi<\xi_0\leq \frac{\xi_2}{4}\leq\frac{\xi_1}{4}$. We can therefore apply Proposition \ref{prop:decompose_unstable_diagonal_conditionals} to both parameterizations of $L(x')$. This along with \eqref{eq:tau_infty_change_basepoints} gives that on $L(x')$, 
			\begin{align*}
				d\nu^u_{x_0,x',P(x_0)^{-1}y_0}(r_s(x_0,x')P(x_0)^{-1}y_0)
				&= e^{-\tau_\infty(x',x_0)}ds\\
				&= e^{-\tau_\infty(x',x_0)} e^{-\tau_\infty(x,x_0)}dt\\
				&=e^{-2\tau_\infty(x,x_0)}e^{-\tau_\infty(x',x)}dt \\
				&= e^{-2\tau_\infty(x,x_0)} d\nu^u_{x,x',P(x)^{-1}y}(r_t(x,x')P(x)^{-1}y)
			\end{align*}
		Thus, multiplying both sides of \eqref{eq:exp_equidistribution_fiber_centered} by $e^{2\tau_\infty(x,x_0)}$ gives \eqref{eq:exp_equidistribution_fiber}.
	\end{proof}

\subsection{Proof of Theorem \ref{thm:exp_equidistribution_unstables_f}}		\label{sec:proof_equidistribution}
	
	Assume without loss of generality that $0<\e_1,\e_4,\e_5<1$. Fix
		\begin{equation*}
			0<\eta <\eta_0:=\frac{\alpha_1\lambda_2}{44}-\frac{4\e}{11}, 
			\qquad 
			\delta_4:= \frac{\e_1 \e_4 \e_5}{8}>0.
		\end{equation*}
	Choose $Q_1$ and $N_1$ as in the beginning of Section \ref{sec:exp_equidistribution_fiber} with $\delta_1:=\frac{\delta_4^4}{4}$. Choose $\ell\geq 1$ sufficiently large that 
		\begin{equation*}
			\mu(\Lambda^\ell_\e)>1-\frac{\delta_4^4}{4}, \qquad 
			\mu(\bar{\Lambda}^\ell_\e)>1-\frac{\delta_4^4}{4},
		\end{equation*}
	and choose $\tau_g>0$ sufficiently small that 
		\begin{equation*}
			\mu(\mathcal{P}_{\tau_g})>1-\frac{\delta_4^4}{4}.
		\end{equation*}
	Combining this with \eqref{eq:equidistribution_Q1}, we get that
		\begin{equation}\label{eq:equidistribution_measure_hatQ_1}
			\mu(\widehat{Q}_1^c)<\delta_4^4, \qquad
			\widehat{Q}_1:=Q_1\cap G_{\ell,\tau_g}.
		\end{equation}
	Now, choose $\xi_0>0$ sufficiently small that for $0<\xi<\xi_0$,
		\begin{itemize}
		\item 
			Proposition \ref{prop:exp_equidistribution_fiber} holds with $\e_7=\frac{\e_5}{4}$,
		\item 
			Theorem \ref{thm:exp_equidistribution_g_expmixing} holds with parameter $\e_{12}$,
		\item 
			\cite[Lemma 4.4]{ExpMixingBernoulli} holds with $\hat{\e}=\delta_4$, and
		\item 
			the second part of Proposition \ref{prop:radon-nikodym_unstable_conditional_wrt_riem_vol} holds with $\delta=\frac{\e_5}{20}$. This along with \eqref{eq:fiberslices_measure_L(x')} gives that 
				\begin{equation}\label{eq:fiberslice_measure_L(x')_bound}
					\nu^u(x,x',P(x)^{-1}y)(L(x'))\in \left( 1-\frac{\e_5}{20}, 1+\frac{\e_5}{20} \right) 2\xi.
				\end{equation}
		\end{itemize}
	Fix $0<\xi<\xi_0$. Take $N_0=\max\{ \e_{12}\bar{n}_g, N_0' \}$, where
		\begin{itemize}
		\item 
			$\bar{n}_g$ comes from Theorem \ref{thm:exp_equidistribution_g_expmixing} with parameter $\e_{12}$. Note that for $n>N_0$, $m_n:=\lceil \frac{n}{\e_{12}}\rceil >\bar{n}_g$. 
		\item 
			$N_0'$ comes from applying Proposition \ref{prop:exp_equidistribution_fiber} with $\e_7=\frac{\e_5}{4}$. Also recall from the proof of Proposition \ref{prop:exp_equidistribution_fiber} that $N_0'\geq N_1$. 
		\end{itemize}
	Fix $n>N_0$. Let $\mathcal{R}_{m_n}$ and $K_{m_n}$ be the sets given by Theorem \ref{thm:exp_equidistribution_g_expmixing}. Recall from Theorem \ref{thm:exp_equidistribution_g_expmixing} that 
		\begin{equation}\label{eq:equidistribution_g_measure_K,R}
			\mu(\mathcal{R}_{m_n}) \geq 1-300 \e_{12}^{10^{10}/16}, \qquad 
			\mu(K_{m_n})\geq 1-\e_{12}^{10}.
		\end{equation}
	We set 
		\begin{equation*}
			Q_n:= \widehat{Q}_1 \cap \mathcal{R}_{m_n}.
		\end{equation*}
	By \eqref{eq:equidistribution_measure_hatQ_1} and \eqref{eq:equidistribution_g_measure_K,R}, we get that 
		\begin{equation*}
			\mu(Q_n)>1-\frac{\e_1\e_4\e_5}{64}-300\e_{12}^{10^{10}/16}.
		\end{equation*}
	Now, we define the set 
		\begin{equation*}
			Q_2 := \left\{ x\in M \colon r_{u,g}(x)\geq \xi \text{ and } \mu^u_x\left( W^u_{g,\xi}(x)\cap \widehat{Q}_1^c\right) \leq \delta_4\mu^u_x\left( W^u_{g,\xi}(x) \right) \right\}.
		\end{equation*}
	By \eqref{eq:equidistribution_measure_hatQ_1} and our choice of $\xi_0$, we can apply \cite[Lemma 4.4]{ExpMixingBernoulli} to $\mathfrak{B}=\widehat{Q}_1^c$ to get that 
		\begin{equation}\label{eq:equidistribution_measure_Q2}
			\mu(Q_2)>1-4\delta_4>1-\frac{\e_1}{2}.
		\end{equation}
	Fix 
		$
			Q:= Q_2\cap K_{m_n}.
		$
	Since $\e_{12}^{10}<\frac{\e_1}{2}$, we can combine \eqref{eq:equidistribution_g_measure_K,R} with \eqref{eq:equidistribution_measure_Q2} to get that 
		$$
			\mu(Q)>1-\e_1.
		$$
	For the rest of the proof, we fix $x\in Q$. We now choose the set $J=J(x)\subset I_1$. Let $J_g(x)$ be from Theorem \ref{thm:exp_equidistribution_g_expmixing}. We obtain $J(x)$ from $J_g(x)$ by throwing out cubes that we designate as ``bad." For $i\in J_g(x)$, we say that $B_i$ is \textit{bad} if 
		\begin{equation*}
			\mu^u_x\left( W^u_{g,\xi}(x) \cap \widehat{Q}_1^c \cap g^{-n}(B_i) \right) \geq \frac{\e_5}{4} \mu(B_i) \mu^u_x\left( W^u_{g,\xi}(x) \right)
		\end{equation*}
	Set 
		$
			J(x):= J_g(x)\setminus \{ i \colon B_i \text{ is bad} \}. 
		$
	Note that $J(x)$ is independent of choice of $D$. We now show that cubes in $J(x)$ cover all but $\e_4$ of $M$. Combining Chebychev's inequality with the fact that $x\in Q_2$, we get that 
		\begin{equation*}
			\mu \left( \bigcup_{i \ \text{bad}} B_i \right) \leq \frac{4}{\delta_4\e_5}=\frac{\e_1\e_4}{2}\leq \frac{\e_4}{2}.
		\end{equation*}
	This combined with the fact that $\mu\left( \bigcup_{i\in J_g} B_i \right)>1-\e_{12}^{10}$ and that $\e_{12}^{10}<\frac{\e_4}{2}$ gives that 
		\begin{equation*}
			\mu \left( \bigcup_{i \in J(x)} B_i \right)>1-\e_4. 
		\end{equation*}
	Now, we fix $i\in J(x)$ and $y\in N$. We need to show that \eqref{eq:thm_exp_equidistribution_unstables_f} holds. Combining Proposition \ref{prop:unstables_f} and Proposition \ref{prop:decompose_unstable_diagonal_conditionals}, we get that 
		\begin{align*}
			\mathcal{I}_1
			&:= m^u_{(x,y)}\left( f^{-n}(B_i\times D) \cap W^u_{f,\xi}(x,y) \cap (Q_n\times N) \right) \\
			&= \int_{W^u_{g,\xi}(x)\cap Q_n} \chi_{B_i}(g^nx') 
				\int_{L(x')} \chi_{D}(A_n(x')P(x')y') d\nu^u_{x,x',P(x)^{-1}y}(y')
				d\mu^u_x(x').
		\end{align*}
	Applying Proposition \ref{prop:exp_equidistribution_fiber} with $(x_0,y_0)=(x,y)$ along with \eqref{eq:fiberslice_measure_L(x')_bound} gives that
		\begin{equation}\label{eq:equidistribution_bound_I1_1}
			\mathcal{I}_1 \in 
			\left( \left(1-\frac{\e_5}{4}\right)\left(1-\frac{\e_5}{20}\right),\left(1+\frac{\e_5}{4}\right)\left(1+\frac{\e_5}{20}\right) \right) 2\xi \nu(D) \mu^u_x\left( W^u_{g,\xi}(x)\cap Q_n \cap g^{-n}(B_i) \right).
		\end{equation}
	Now, we bound $\mu^u_x\left( W^u_{g,\xi}(x)\cap Q_n \cap g^{-n}(B_i) \right)$. Since $i\in J(x)\subset J_g(x)$, Theorem \ref{thm:exp_equidistribution_g_expmixing} and the fact that $\e_{12}^{10}<\frac{\e_5}{4}$ gives
		\begin{equation*}
			\mu^u_x\left( W^u_{g,\xi}(x)\cap Q_n \cap g^{-n}(B_i) \right) 
			\leq \mu^u_x\left( W^u_{g,\xi}(x)\cap \mathcal{R}_{m_n} \cap g^{-n}(B_i) \right)
			< \left(1+\frac{\e_5}{4}\right) \mu(B_i)\mu^u_x(W^u_{g,\xi}(x)).
		\end{equation*}
	If we additionally use the fact that $B_i$ is not bad, we get that 
		\begin{align*}
			\mu^u_x\left( W^u_{g,\xi}(x)\cap Q_n \cap g^{-n}(B_i) \right)  
			&\geq 
				\mu^u_x\left( W^u_{g,\xi}(x) \cap \mathcal{R}_{m_n} \cap g^{-n}(B_i) \right) 
				-
				\mu^u_x\left( W^u_{g,\xi}(x) \cap \hat Q_1^c \cap g^{-n}(B_i) \right)\\
			&> \left(1-\frac{\e_5}{2}\right) \mu(B_i)\mu^u_x(W^u_{g,\xi}(x)).
		\end{align*}
	Plugging these two bounds into \eqref{eq:equidistribution_bound_I1_1}, we get 
		\begin{equation}\label{eq:equidistribution_bound_I1_2}
			\mathcal{I}_1 \in 
			\left(  \left(1-\frac{\e_5}{2}\right)\left(1-\frac{\e_5}{4}\right)\left(1-\frac{\e_5}{20}\right),\left(1+\frac{\e_5}{4}\right)^2\left(1+\frac{\e_5}{20}\right) \right) 2\xi \nu(D) \mu(B_i)\mu^u_x(W^u_{g,\xi}(x)).
		\end{equation}
	Now, combining Proposition \ref{prop:unstables_f}, Proposition \ref{prop:decompose_unstable_diagonal_conditionals}, and \eqref{eq:fiberslice_measure_L(x')_bound}, we get that 
		\begin{equation*}
			m^u_{(x,y)}\left( W^u_{f,\xi}(x,y) \right) 
			\in \left( 1-\frac{\e_5}{20}, 1+\frac{\e_5}{20} \right) 2\xi \mu^u_x(W^u_{g,\xi}(x)).
		\end{equation*}
	Combining this with \eqref{eq:equidistribution_bound_I1_2} and the fact that $0<\e_5<1$, we get that \eqref{eq:thm_exp_equidistribution_unstables_f} holds.

\section{Proof of Theorem \ref{thm:Main_f_bernoulli}} \label{sec:proof_main_thm}

	To prove that $f$ is Bernoulli, we will follow the general strategy followed in \cite{OrnsteinWeiss}, \cite{KanigowskiBernoulliHomogeneous}, and \cite{ExpMixingBernoulli}. We outline the strategy in Section \ref{sec:VWB} and then prove Theorem \ref{thm:Main_f_bernoulli} in Section \ref{sec:subsec:proof_main}. 
	
	\subsection{Very Weakly Bernoulli Partitions}\label{sec:VWB}
		
		A finite partition $\mathcal{P}=(P_1,...,P_k)$ of a measure space $(X,\mathcal{B},\mu)$ is a finite family of disjoint sets, called \textit{atoms}, $P_1,...,P_k\subset X$ whose union is all of $X$. 
			\begin{itemize}
			\item 
				For $\e>0$, we say that a property holds \textit{for $\e$ almost every atom} of $\mathcal{P}$  if it holds for all atoms of $\mathcal{P}$ except for a set of atoms whose union has measure less than $\e$. 
			\item 
				If the measure space $(X,\mathcal{B},\mu)$ has a distance function $d$, we say that a partition $\mathcal{P}$ is \textit{regular} if $\forall \e>0$, $\exists \delta>0$ so that $\mu\left( V_\delta\left(\partial \mathcal{P}\right) \right)<\e$, where $V_{\delta}\left( \partial \mathcal{P} \right)$ is the $\delta$ neighborhood of $\partial \mathcal{P}$.
			\end{itemize}
	
		Let $\mathcal{P} = (P_1,...,P_k)$ and $\mathcal{Q}=(Q_1,...,Q_\ell)$ be finite partitions of a probability space $(X,\mu)$. 
			\begin{itemize}
			\item 
				The \textit{least common refinement} of $\mathcal{P}$ and $\mathcal{Q}$ is denoted $\mathcal{P}\vee \mathcal{Q}$ and is the partition of $X$ into sets of the form $P_i\cap Q_j$. 
			
			\item 
				If $A\subset X$, then $\mathcal{P}|_A = (P_1\cap A, ..., P_k\cap A)$ is the \textit{induced partition} of $(A, \mu|_A)$.
			
			\item 
				If $k=\ell$, then we can define the distance between the partitions $\mathcal{P}$ and $\mathcal{Q}$ by 
					$$
						\bar{d}(\mathcal{P},\mathcal{Q})=\sum_{i=1}^k \mu(P_i \bigtriangleup Q_i)
					$$
			\end{itemize}
			
		Let $(\mathcal{P}^s)_{s=0}^\infty$, where $\mathcal P ^s = \{P^s_1,...,P^s_{k_s}\}$ is a sequence of finite partitions of $(X,\mathcal{B},\mu)$.
			\begin{itemize}
			\item 
				$\bigvee_{s=0}^\infty \mathcal{P}^s$ is the smallest $\sigma$-algebra with respect to which the partitions $\mathcal{P}^s$ are measurable. 
			\item 
				We say that $(\mathcal{P}^s)_{s=0}^\infty$ converges to a partition into points if $\bigvee_{s=0}^\infty \mathcal{P}^s =\mathcal{B}$.
			\item 
				The $(\mathcal P^s)_{s=0}^\infty$ name of a point $x$ is the sequence $\ell_s=\ell_s(x)$ determined by $x\in P_{\ell_s}^s$. 
			\end{itemize}
	
		Now, we consider two sequences $\mathcal{P}^s=(P_1^s,...,P_k^s)$ and $\mathcal{Q}^s=(Q_1^s,...,Q_k^s)$, $s=1,...,S$ of finite partitions of $(X,\mu)$ and $(Y,\nu)$ respectively. 
			\begin{itemize}
			\item 
				If $(X,\mu)=(Y,\nu)$, then we define the distance between the sequences of partitions as 
					$$
						\bar{d}\left( (\mathcal{P}^s)_{s=1}^S, (\mathcal{Q}^s)_{s=1}^S \right)
						=
						\frac{1}{S}\sum_{s=1}^S \bar{d}(\mathcal{P}^s, \mathcal{Q}^s)
					$$
				If $(X,\mu)\neq (Y,\nu)$, then we define the distance as follows,
					\begin{itemize}
					\item 
						We say that $\mathcal{P}^s\sim \mathcal{Q}^s$ for $s=1,...,S$, if $\mu(P_i^s)=\nu(Q_i^s)$ for all $i=1,...,k$, and all $s=1,...,S$. 
					\item 
						We define the distance the sequences of partitions as follows, 
							$$
								\bar{d}\left( (\mathcal{P}^s)_{s=1}^S, (\mathcal{Q}^s)_{s=1}^S \right)
								=
								\inf_{\substack{\bar{\mathcal{Q}}^s \sim \mathcal{Q}^s \\  \text{ for } s=1,...,S}} \bar{d}\left( (\mathcal{P}^s)_{s=1}^S, (\bar{\mathcal{Q}}^s)_{s=1}^S \right),
							$$
						where $\bar{\mathcal{Q}}^s$ are partitions of $(X,\mu)$. 
					\end{itemize}
			\end{itemize}
		
		Finally, suppose $\mathcal{P}=(P_1,...,P_k)$ is a finite partition of a probability space $(X,\mathcal{B},\mu)$ and $T:(X,\mathcal{B},\mu) \to (X,\mathcal{B},\mu)$ is an automorphism $(X,\mathcal{B},\mu)$.
			\begin{itemize}
			\item 
				The partition $T^n\mathcal{P}$ is the partition of $(X,\mu)$ given by $T^n\mathcal{P}=(T^nP_1,...,T^nP_k)$.
			\item 
				We say that the $S,\mathcal{P}-$name of a point $x\in X$ (where $S\geq 1$ is an integer) is the sequence $(x_i^{\mathcal{P}})_{i=0}^S\in \{1,...,k\}^{S+1}$ determined by $T^i(x)\in P_{x_i^{\mathcal{P}}}$.
			\item 
				A transformation $\theta:(X,\mu)\to (Y,\nu)$ is $\e$-\textit{measure preserving} if $\exists E'\subset X$ with $\mu(E')<\e$ such that for all $A\subset X\setminus E'$, we have that 
					$$
						\left| \frac{\nu(\theta(A))}{\mu(A)}-1\right| <\e.
					$$
			\end{itemize}
			
		We can now define the \textit{Very Weak Bernoulli Property}:
		
		\begin{definition}
			Let $T:(X,\mathcal{B},\mu) \to (X,\mathcal{B},\mu)$ be an automorphism of a probability space $(X,\mathcal{B},\mu)$. A finite partition $\mathcal{P}$ of $X$ is \textit{very weakly Bernoulli} (\textit{VWB}) if $\forall \e>0$, $\exists N_0\in \N$ such that $\forall N'\geq N\geq N_0$, $\forall S\geq N_0$,
				\footnote{In \cite{OrnsteinWeiss}, they require $S\geq 1$, however we use the definition given in \cite{ExpMixingBernoulli} because it is weaker and sufficient to get the Bernoulli Property (See Remark 2.20 in \cite{ExpMixingBernoulli}).}
			and for $\e$ a.e. atom $A\in \bigvee_{i=N}^{N'} T^i(\mathcal{P})$, we have that 
				$$
					\bar{d}\left( (T^{-i}\mathcal{P})_{i=0}^S, (T^{-i}\mathcal{P}|_A)_{i=0}^S\right)<\e
				$$
		\end{definition}
		
		To show that $f$ is Bernoulli, it suffices to find a sequence of VWB partitions that converges to a partition into points. More precisely,
		
		\begin{lemma}[Theorems A+B in \cite{OrnsteinWeiss}, Theorem 2.18 in \cite{ExpMixingBernoulli}]\label{lem:VWB->Bernoulli}
			Let $T:(X,\mathcal{B},\mu) \to (X,\mathcal{B},\mu)$ be an automorphism of a probability space $(X,\mathcal{B},\mu)$. If $(\mathcal{P}^k)_{k=1}^\infty$ is a sequence of VWB partitions for $T$ of $(X,\mathcal{B},\mu)$ that converges to a partition into points (i.e. the smallest $\sigma$-algebra with respect to which all the $\mathcal{P}^k$ are measurable is $\mathcal{B}$), then $T$ is Bernoulli.
		\end{lemma}
		
		To show that a finite partition $\mathcal{P}$ is VWB for $f$, it suffices to  show that $\forall \e>0$, for $\e$ a.e. atom $A\in \bigvee_{i=N}^{N'} f^i(\mathcal{P})$, we can find an $\e$-measure  preserving map $\theta:(A,\mu|_A)\to (X,\mu)$ such that the orbits of $x$ and $\theta(x)$ are in the same atom of $\mathcal{P}$ for most times. More precisely,
		
		\begin{lemma}[Lemma 2.19 in \cite{ExpMixingBernoulli}, Lemma 2.3 in \cite{KanigowskiBernoulliHomogeneous}, Lemma 1.3 in \cite{OrnsteinWeiss}]\label{lem:VWB_coupling_same_atom}
			Let $T:(X,\mathcal{B},\mu) \to (X,\mathcal{B},\mu)$ be an automorphism of a probability space $(X,\mathcal{B},\mu)$, and let $\mathcal{P}$ be a finite partition of $X$. Suppose that for all $\e>0$, $\exists N\in \N$ such that $\forall N'>N$, for $\e$ a.e. atom $A\in \bigvee_{i=N}^{N'}T^i\mathcal{P}$ and for all $S\geq N$, there exists an $\e$-measure preserving map $\theta=\theta(N,S,A):(A,\mu|_A)\to (X,\mu)$ such that 
				$$
					d(x,\theta(x)):= \frac{1}{S}\sum_{i=0}^{S-1}e(x_i^{\mathcal{P}}-(\theta(x))_i^{\mathcal{P}})<\e,
				$$
			where $e:\Z\to \Z$ is defined by $e(n)=\begin{cases} 0, \ & n=0 \\ 1, \ & n\neq 0\end{cases}$. 
			
			Then $\mathcal{P}$ is a VWB partition for $T$.
		\end{lemma}
		
		If the partition $\mathcal{P}$ is regular, we can replace the condition in Lemma \ref{lem:VWB_coupling_same_atom} that the orbits of $x$ and $\theta(x)$ are in the same atom of $\mathcal{P}$ for most times by requiring the orbits of $x$ and $\theta(x)$ to be close for most times. More precisely,
		
		\begin{lemma} [Corollary 2.21 in \cite{ExpMixingBernoulli}, Lemma 2.4\footnote{or more precisely the proof of how Lemma 2.4 follows from the conclusion of Lemma 2.3} in \cite{OrnsteinWeiss}] \label{lem:VWB_coupling_close}
			Let $(X,\mathcal{B},\mu)$ be a probability space with distance function $d$. Let $T:(X,\mathcal B, \mu, d) \to (X,\mathcal B, \mu,d)$ be an ergodic automorphism of the probability space $(X,\mathcal B,\mu)$. Let $\mathcal P$ be a regular partition of $X$. 
			
			Suppose for all $\e>0$, there exists $N\in \N$ such that for all $N'\geq N$, for $\e$-a.e. atom $A\in \bigvee_N^{N'} T^i \mathcal P$ and for all $S\geq N$, there exists an $\e$ measure preserving map $\theta=\theta(N, S, A):(A,\mu|A)\to (X,\mu)$ such that $\forall x\in A$
				$$
					\frac{1}{S}\card\left( \left\{ i\in \{0,...,S-1\}: d(T^ix,T^i(\theta x))<\e\right\} \right)>1-\e
				$$
			Then $\mathcal P$ is very weak Bernoulli.
		\end{lemma}
		
\subsection{Proof of Theorem \ref{thm:Main_f_bernoulli}}\label{sec:subsec:proof_main}
	
	Let $\mathcal{P}$ be a partition of $M\times N$ with piecewise smooth boundary. Since $\mathcal{P}$ has piecewise smooth boundary, it is regular. Since $f$ is ergodic by Proposition \ref{prop:f_ergodic}, if we can show that the hypotheses of Lemma \ref{lem:VWB_coupling_close} hold, we'll have that the partition $\mathcal{P}$ is very weak Bernoulli. Since we chose the partition $\mathcal{P}$ arbitrarily, Lemma \ref{lem:VWB->Bernoulli} will then give that $f$ is Bernoulli. 
	
	So, to prove Theorem \ref{thm:Main_f_bernoulli} we just need to show that the hypotheses of Lemma \ref{lem:VWB_coupling_close} hold. We will do this following the arguments given in \cite{ExpMixingBernoulli}. The only places that exponential mixing is used in \cite{ExpMixingBernoulli} is to show that their map has a nonzero Lyapunov exponent and to establish equidistribution of unstable manifolds on exponentially small cubes (\cite[Lemma 7.3]{ExpMixingBernoulli}). Outside of establishing these two results, the rest of their proof only uses that the map is a $C^{1+\alpha}$ diffeomorphism of a compact manifold that preserves a smooth ergodic measure. In other words, the arguments given in \cite{ExpMixingBernoulli} show that if $f$ is a $C^{1+\alpha}$ diffeomorphism of a compact manifold that preserves a smooth ergodic measure and that if $f$ has some nonzero Lyapunov exponents and the local unstable manifolds of $f$ equidistribute on exponentially small cubes (i.e. the conclusion of \cite[Lemma 7.3]{ExpMixingBernoulli} holds), then $f$ is Bernoulli. 

	Since $g:M\to M$ is $C^{1+\alpha_g}$ and $A:M\to PSL_2\R$ is $C^{1+\alpha}$, we have that the skew product $f:M\times N\to M\times N$ given by \eqref{eq:skew-product-f_intro} is $C^{1+\alpha_f}$, where $\alpha_f=\min\{\alpha_g,\alpha\}$. We know that $f$ preserves the smooth measure $m=\mu\times \nu$ and and is ergodic (Proposition \ref{prop:f_ergodic}). The Lyapunov exponents of $f$ are the union (with multiplicity) of the Lyapunov exponents of $g$ and of $\mathcal{A}$. Thus, $f$ has some nonzero Lyapunov exponents. 
	
	To show that $f$ is Bernoulli, all that remains is to show that the unstables of $f$ equidistribute on exponentially small cubes. We did this in Theorem \ref{thm:exp_equidistribution_unstables_f}. There are several differences between Theorem \ref{thm:exp_equidistribution_unstables_f} and \cite[Proposition 7.3]{ExpMixingBernoulli}. The two that impact the proof are
		\begin{enumerate}
		\item 
			Theorem \ref{thm:exp_equidistribution_unstables_f} proves equidistribution on exponentially small product cubes, whereas \cite[Proposition 7.3]{ExpMixingBernoulli} proves equidistribution on exponentially small ``$f$-adapted parallelograms." A ``$f$-adapted parallelogram" is a thickening of a local unstable manifold $W^u_{f,\xi}(z)$ in the $cs$-direction using a Lyapunov chart for $f$ centered at $z$. For more details, see \cite[(5.1),(5.13)]{ExpMixingBernoulli}.
		\item 
			Theorem \ref{thm:exp_equidistribution_unstables_f} proves equidistribution of unstables restricted to the set $Q_n\times N$,  whereas \cite[Proposition 7.3]{ExpMixingBernoulli} proves equidistribution of unstables restricted to the set $\mathcal{R}_n$. It is important to the proof of the Bernoulli property in \cite{ExpMixingBernoulli} that $\mathcal{R}_n$ is saturated by the fake center-stable foliation and that it has local product structure given by the fake center-stable foliation and the local unstable manifolds. 
		\end{enumerate}
	Thus, the final step is to prove that Theorem \ref{thm:exp_equidistribution_unstables_f} implies an analogue of \cite[Proposition 7.3]{ExpMixingBernoulli}. Before stating this analogue, we need to introduce some notation and definitions. 
	
	Let $\mathcal{P}_{f,\tau_f}$ be the Pesin sets for $f$ defined in \cite[Section 2.1]{ExpMixingBernoulli}. The local unstable manifolds for $f$ defined in \cite[Lemma 2.8]{ExpMixingBernoulli} are the same as the local unstable manifolds we define for $f$ in Section \ref{sec:unstable_f}. For fixed $\e_6>0$, we define
		$$
			\mathcal{L}^f_{n,\tau_f}:= \mathcal{P}_{f,\tau_f} \cap f^{-n}\mathcal{P}_{f,\tau_f} \cap f^{-\lfloor \e_6 n \rfloor}\mathcal{P}_{f,\tau_f}.
		$$
	As in \cite[(5.4)]{ExpMixingBernoulli}, we assume that $\tau_f$ has been chosen small enough that 
		\begin{equation*}
			m(\mathcal{L}^f_{n,\tau_f})\geq 1-3\e_6^{\mathbf{b}},
		\end{equation*}
	where $\mathbf{b}=10^{10}$. As in \cite[Section 5]{ExpMixingBernoulli}, for $\eta_f>0$, $\e_6>0$, and $n\in \N$ let $\{\mathcal{B}_k^f\}_{k\in J_f}=\{\mathcal{B}_k^f(e^{-\eta_f \e_6 n})\}_{k\in J_f}$ be the family of pairwise disjoint $f$-adapted parallelograms defined in \cite[Section 5, (5.1),(5.13)]{ExpMixingBernoulli}. From \cite[(5.10)]{ExpMixingBernoulli}, we have that 
		\begin{equation}\label{eq:measure_cover_f_adapted_parallelograms}
			m\left( \bigcup_{k\in J_f} \mathcal{B}_k^f \right)\geq1-100\e_6^{\mathbf{b}/16}.
		\end{equation}
	Also note that $\mathcal{B}_k^f$ has unstable radius $e^{-(\eta_f-\e_6)\e_6n}$ and center-stable radius $e^{-\eta_f\e_6 n}$ \cite[(5.1), Section 7.2]{ExpMixingBernoulli}. We assume without loss of generality that $\e_6<\eta_f$ so that the unstable radius of each $\mathcal{B}_k^f$ is exponentially small. We let $\mathcal{R}_n^f$ be the domain of the fake center-stable foliation for $f$ \cite[Sections 5 and 6, (7.14)]{ExpMixingBernoulli}. As in \cite[Lemma 6.7]{ExpMixingBernoulli}, we have that 
		\begin{equation*}
			m(\mathcal{R}_n^f)\geq 1-300\e_6^{\mathbf{b}/16}.
		\end{equation*}

	We now state our analogue of \cite[Proposition 7.3]{ExpMixingBernoulli}:
	
	\begin{prop}\label{prop:equidistribution_f_adapted}
		There exists a number $\eta_f>0$ such that for all $\e_6>0$, there exists $\xi_{0,f}>0$ such that for all $\xi\in (0,\xi_{0,f})$, there exists $\bar{n}=\bar{n}(\e_6,\xi)\in \N$ such that for all $n\geq \bar{n}$, there exists a set $K_n\subset \mathcal{L}^f_{n,\tau_f}$ with $m(K_n)\geq 1-\e_6^{10}$ such that if $\{\mathcal{B}_k^f\}_{k\in J_f}=\{\mathcal{B}_k^f(e^{-\eta_f \e_6 n})\}_{k\in J_f}$ is a family of $f$-adapted parallelograms, then for all $x\in K_n$, there is a subset $J'_f(x)\subset J_f$ such that $m\left( \bigcup_{k\in J'_f}\mathcal{B}_k^f \right)>1-\e_6^{10}$ and 
			\begin{equation}\label{eq:prop:equidistribution_f_adapted1}
				\hat{m}_{W^u_{f,\xi}(x)}^u\left( f^{-\lfloor \e_6 n \rfloor} \left( \mathcal{L}_{n,\tau_f}^f \right)\right) \geq 1-\e_6^{10}
			\end{equation}
		and for all $k\in J_f'(x)$,
			\begin{equation}\label{eq:prop:equidistribution_f_adapted2}
				\hat{m}_{W^u_{f,\xi}(x)}^u \left( \mathcal{R}_n^f \cap f^{-\lfloor \e_6 n\rfloor}(\mathcal{B}_k^f)  \right) \in (1-\e_6^{10},1+\e_6^{10})m(\mathcal{B}_{k}^f)
			\end{equation}
		and for $\tilde{\e}_6=\e_6^{4000}$, 
			\begin{equation}\label{eq:prop:equidistribution_f_adapted3}
				\hat{m}_{W^u_{f,\xi}(x)}^u \left( \mathcal{R}_n^f \cap f^{-\lfloor \e_6 n\rfloor}((1-\tilde{\e}_6)\mathcal{B}_k^f)  \right) \in (1-\e_6^{10},1+\e_6^{10})m((1-\tilde{\e}_6)\mathcal{B}_{k}^f)
			\end{equation}			
	\end{prop}
	
	\begin{proof}
		Fix $0<\eta<\frac{\alpha_1 \lambda_2}{44}-\frac{4\e}{11}>0$. Note that $\eta<\eta_0$ where $\eta_0$ is taken as in Theorem \ref{thm:exp_equidistribution_unstables_f}. Fix $0<\eta_f<\min\{\frac{\eta_g}{2},\eta\}$. 
		
		Fix $\e_6>0$. Assume without loss of generality that $\e_6<\min\{\eta_f,\frac{1}{2}\}$. Fix $0<\e_{12}<\min\{\frac{\e_6^{12}}{4^{1/10}}, \frac{\eta_g}{2}\}$. Let $\eta':=\min\{ \eta_g-\e_{12},\eta \}$. Note that $\eta_f<\min\{\frac{\eta_g}{2},\eta\}<\eta'$. 
		Let $\xi_{0,f}:=\min\{\xi_0,\xi_1\}$, where 
			\begin{itemize}
			\item 
				$\xi_0$ is the $\xi_0$ from applying Theorem \ref{thm:exp_equidistribution_unstables_f} with $\e_1=\e_4=\e_5=\e_6^{120}$ and $\e_{12}$ and $\eta$ as above.
			\item 
				$\xi_1$ is the $\xi_0$ we get from applying \cite[Lemma 4.4]{ExpMixingBernoulli} with $\hat{\e}=\e_6^{90}$. 
			\end{itemize}
		Fix $\xi\in (0,\xi_{0,f})$. To prove Proposition \ref{prop:equidistribution_f_adapted}, it suffices to show that \eqref{eq:prop:equidistribution_f_adapted1} and \eqref{eq:prop:equidistribution_f_adapted2} hold. Then, we can repeat the argument with $(1-\tilde{\e}_6)\mathcal{B}_k^f$ instead of $\mathcal{B}_k^f$. We then let $N_0$ be the maximum of the two $N_0$'s we get, $K_n$ be the intersections of the two $K_n$'s we get, and $J'_f$ be the intersection of the two $J'_f$'s we get. 
		
		Now, let
			\begin{align*}
				T_0:=&\max\left\{ 
					N_0, 
					\frac{1}{\eta}\log\left( \frac{1}{\delta_6}\right),
					\frac{2}{\eta_g}\log\left( \frac{2}{\tau_g} \right),
					\frac{1}{\eta}\log\left( \frac{2}{\delta_{inj}} \right),
					\frac{1}{\eta_f-\e_6}\log\left( \frac{2}{\delta_7} \right), \right. \\
					&\hspace{1.5in} \left.
					\frac{1}{\eta'-\eta_f}\log\left( \frac{C_9 e^{\eta_f}}{C_{10}} \right),
					\frac{1}{\eta'-\eta_f}\log\left( \frac{C_5C_9e^{\eta_f}}{\e_6^{20}} \right)
				 \right\}
			\end{align*}
		where 
			\begin{itemize}
			\item 
				$N_0$ is from applying Theorem \ref{thm:exp_equidistribution_unstables_f} with $\e_1=\e_4=\e_5=\e_6^{120}$ and $\e_{12}$, $\eta$, and $\xi$ as above. 
			\item 
				$\delta_6$ is the $\delta_\e$ we get from applying Lemma \ref{lem:covering_N_cubes} with $\e=\e_6^{40}$.
			\item 
				$\delta_7$, $C_5$, and $C_{10}$ are from Lemma \ref{lem:bound_delta_nbhd_f_adapted_cube}
			\item 
				$C_9$ is from Lemma \ref{lem:diameter_product_cubes}
			\end{itemize}
		Take
			$
				\bar{n}=\frac{T_0+1}{\e_6}.
			$
		We now fix $n\geq \bar{n}$. Let $t:=\lfloor \e_6 n\rfloor$. Note that $t\leq \e_6n <t+1$, so the fact that $n\geq \bar{n}\geq \frac{T_0+1}{\e_6}$ implies that $t>T_0$. 
		
		Now, let $\{\mathcal{B}_i^g\}_{i\in I_1}$ be the family of cubes we get from applying Theorem \ref{thm:exp_equidistribution_g_expmixing} with $\e=\e_{12}$ and at time $n'$ such that $\lfloor \e_{12}n'\rfloor =t$. We can apply Theorem \ref{thm:exp_equidistribution_g_expmixing} at time $n'$ since $n'\geq \frac{t}{\e_{12}}\geq \frac{N_0}{\e_{12}}\geq \bar{n}_g$. Note that the unstable radius of $\mathcal{B}_i^{g}$ is $e^{-(\eta_g-\e_{12})\e_{12}n'}\leq e^{-(\eta_g-\e_{12})t}$ and the center-stable radius of $\mathcal{B}_i^g$ is $e^{-\eta_g \e_{12}n'}\leq e^{-\eta_g t}$. 
		
		Let $\{D_j\}_{j\in I_2}$ be a disjoint family of cubes in $N$ with radius $e^{-\eta t}$ such that 
			$%\displaystyle
				\nu\left( \bigcup_{j\in I_2} D_j \right)\geq 1-\e_6^{40}.
			$
		We can find such a family by Lemma \ref{lem:covering_N_cubes} since $t>T_0\geq -\frac{1}{\eta}\log(\delta_6)$.
		For $(i,j)\in I_1\times I_2$, we define the product cube $C_{ij}=\mathcal{B}_i^g\times D_j$. 
		
		We apply Theorem \ref{thm:exp_equidistribution_unstables_f} with $\e_1=\e_4=\e_5=\e_6^{120}$ and $\eta$ and $\xi$ as above at time $t>T_0=N_0$. This gives us a set $Q\subset M$ with $\mu(Q)>1-\e_6^{120}$ and a set $Q_t\subset M$ with $\mu(Q_t)>1-\frac{\e_6^{360}}{64}-\frac{300}{4^{\mathbf{b}/160}}\e_6^{12\mathbf{b}/16}$, such that for all $(x,y)\in Q\times N$, there exists a subset $J=J(x)\subset I_1$ with 
			$
				\mu\left( \bigcup_{i\in J} \mathcal{B}_i^g \right)>1-\e_6^{120}
			$
		such that for all $(i,j)\in J\times I_2$,
			\begin{equation}\label{eq:prop:equidistribution_f_adapted-productcubes}
				\hat{m}^u_{W^u_{f,\xi}(x,y)}\left( (Q_t\times N)\cap f^{-t}(C_{ij}) \right) \in \left( 1-\e_6^{120},1+\e_6^{120} \right)m(C_{ij})
			\end{equation}
		
		We define the set $K_n$ to be the set of points $z$ in $(Q\times N)\cap \mathcal{L}_{n,\tau_f}^f$ for which not too large a proportion of $W^u_{f,\xi}(x)$ is contained in the bad set 
			$$
				E_n:= \left( (Q_t\times N)\Delta \mathcal{R}_n^f \right) \cup f^{-t}\left( (\mathcal{L}_{n,\tau_f}^f)^c \right).
			$$
		More precisely, 
			$$
				K_n:= \bar{K}_n \cap (Q\times N) \cap \mathcal{L}_{n,\tau_f}^f, 
				\qquad \bar{K}_n:=\left\{ z\in M\times N \colon r_{u,f}(z)\geq \xi \text{ and } \hat{m}^u_{W^u_{f,\xi}(z)}(E_n)\leq \e_6^{90} \right\},
			$$
		where $r_{u,f}(z)$ is the radius of the local unstable manifold $W^u_f(z)$. We need to show that $m(K_n)\geq 1-\e_6^{10}$. To do this, we first bound $m(\bar{K}_n)$. Observe that 
			\begin{align*}
				m(E_n)
				&\leq m\left( (Q_t\times N)^c \right)+ m\left( (\mathcal{R}_n^f)^c \right)+m\left( f^{-t}((\mathcal{L}_{n,\tau_f}^f)^c )\right)\\
				&\leq \frac{\e_6^{360}}{64}+\frac{300}{4^{\mathbf{b}/160}}\e_6^{12\mathbf{b}/16} 
					+300\e_6^{\mathbf{b}/16} 
					+3\e_6^{\mathbf{b}}
				\leq \e_6^{360}
			\end{align*}
		since $\e_6\leq \frac{1}{2}$. Since $\xi<\xi_0\leq \xi_1$, we can apply \cite[Lemma 4.4]{ExpMixingBernoulli} with $\hat{\e}=\e_6^{90}$ to get that $m(\bar{K}_n)\geq 1-4\e_6^{90}$. Thus, since $\e_6\leq \frac{1}{2}$, 
			\begin{align*}
				m(K_n)
				&\geq 1- m\left( \bar{K}_n^c \right)-m\left( (Q\times N)^c \right)-m\left( (\mathcal{L}_{n,\tau_f}^f)^c \right)
				\geq 1-4\e_6^{90}-\e_6^{120}-3\e_6^{\mathbf{b}} 
				\geq 1-\e_6^{10}.
			\end{align*}
		
		From now on, we fix $z\in K_n$. We let $\mathcal{W}=W^u_{f,\xi}(z)$. We now can show that \eqref{eq:prop:equidistribution_f_adapted1} holds. Since $z\in K_n\subset \bar{K}_n$, we have that $\hat{m}_{\mathcal{W}}^u(E_n)\leq \e_6^{90}$. Since $f^{-t}((\mathcal{L}^f_{n,\tau_f})^c)\subset E_n$, we get that \eqref{eq:prop:equidistribution_f_adapted1} holds. 
		
		The remainder of the proof is devoted to showing that \eqref{eq:prop:equidistribution_f_adapted2} holds. 
		
	\textit{Step 1:} 
		The first step is to replace the set $Q_t\times N$ in \eqref{eq:prop:equidistribution_f_adapted-productcubes} with $\mathcal{R}_n^f$. We begin by getting rid of the product cubes $C_{ij}$ that spend too much time inside $(Q_t\times N)\Delta \mathcal{R}_n^f$. Let 
			$$
				J''=J''(z):= \left\{ (i,j)\in J\times I_2 \colon \hat{m}^u_{\mathcal{W}}\left( ((Q_t\times N)\Delta \mathcal{R}_n^f) \cap f^{-t}(C_{ij}) \right)< \e_6^{40} m(C_{ij}) \right\}
			$$
		We show that the cubes in $J''$ cover most of $M\times N$. By Chebychev's inequality and the fact that $z\in K_n\subset \bar{K}_n$, we get that
			\begin{align*}
				m\left( \bigcup_{(J\times I_2)\setminus J''} C_{ij} \right)
				&\leq \e_6^{-40} \sum_{I_1\times I_2} \hat{m}^u_{\mathcal{W}}\left( ((Q_t\times N)\Delta \mathcal{R}_n^f) \cap f^{-t}(C_{ij}) \right) \\
				&\leq \e_6^{-40}\hat{m}_{\mathcal{W}}^u\left( (Q_t\times N)\Delta \mathcal{R}_n^f \right)
				\leq \e_6^{-40} \hat{m}^u_{\mathcal{W}}(E_n)
				\leq \e_6^{-40} \e_6^{90}
				= \e_6^{50}
			\end{align*}
		This along with the fact that $\mu\left( \bigcup_{i\in J} \mathcal{B}_i^g \right)>1-\e_6^{120}$ and $\nu\left( \bigcup_{j\in I_2} D_j \right)\geq 1-\e_6^{40}$ gives that 
			\begin{equation}\label{eq:prop:equidistribution_f_adapted_J''_measure}
				m\left( \bigcup_{J''} C_{ij} \right) 
				=m\left( \bigcup_{J\times I_2} C_{ij} \right)-m\left( \bigcup_{(J\times I_2)\setminus J''} C_{ij} \right)
				> \left( 1-\e_6^{120} \right)\left( 1-\e_6^{40} \right)-\e_6^{50}
				\geq 1-3\e_6^{40}
			\end{equation}
			
		Now, we're ready to replace $Q_t\times N$ by $\mathcal{R}_n^f$ in \eqref{eq:prop:equidistribution_f_adapted-productcubes}. Take $(i,j)\in J''$.  We begin with the lower bound. Since $z\in K_n\subset Q\times N$, $(i,j)\in J''\subset J\times I_2$, and $t\geq N_0$, we have that \eqref{eq:prop:equidistribution_f_adapted-productcubes} holds. Additionally, since $(i,j)\in J''$, we have that 
			\begin{equation}\label{eq:prop:equidistribution_f_adapted_J''_condition}
				\hat{m}^u_{\mathcal{W}}\left( ((Q_t\times N)\Delta \mathcal{R}_n^f) \cap f^{-t}(C_{ij}) \right)< \e_6^{40} m(C_{ij})
			\end{equation}
		Combining \eqref{eq:prop:equidistribution_f_adapted-productcubes} and \eqref{eq:prop:equidistribution_f_adapted_J''_condition}, we get that 
			\begin{align*}
				\hat{m}^u_{\mathcal{W}}\left(\mathcal{R}_n^f \cap f^{-t}(C_{ij})\right)
				&\geq 	\hat{m}^u_{\mathcal{W}}\left( (Q_t\times N)\cap f^{-t}(C_{ij}) \right) - \hat{m}^u_{\mathcal{W}}\left( ((Q_t\times N)\Delta \mathcal{R}_n^f) \cap f^{-t}(C_{ij}) \right) \\
				&> (1-\e_6^{120}-\e_{6}^{40})m(C_{ij}) 
				> (1-2\e_6^{40})m(C_{ij})
			\end{align*}
		The upper bound is similar. Combining \eqref{eq:prop:equidistribution_f_adapted-productcubes} and \eqref{eq:prop:equidistribution_f_adapted_J''_condition}, we get that 
			\begin{align*}
				\hat{m}^u_{\mathcal{W}}\left(\mathcal{R}_n^f \cap f^{-t}(C_{ij})\right)
				&\leq 	\hat{m}^u_{\mathcal{W}}\left( (Q_t\times N)\cap f^{-t}(C_{ij}) \right) + \hat{m}^u_{\mathcal{W}}\left( ((Q_t\times N)\Delta \mathcal{R}_n^f) \cap f^{-t}(C_{ij}) \right) \\
				&< (1+\e_6^{120}+\e_{6}^{40})m(C_{ij})
				< (1+2\e_6^{40})m(C_{ij})
			\end{align*}			
		We have shown that 
			\begin{equation}\label{eq:prop:equidistribution_f_adapted_replace_QxN}
				\hat{m}^u_{\mathcal{W}}\left(\mathcal{R}_n^f \cap f^{-t}(C_{ij})\right)
				\in \left( 1-2\e_{6}^{40}, 1+2\e_{6}^{40} \right)m(C_{ij})
			\end{equation}
		It follows immediately from \eqref{eq:prop:equidistribution_f_adapted_replace_QxN} and \eqref{eq:prop:equidistribution_f_adapted_J''_measure} that $\mathcal{W}$ spends very little time outside $\bigcup_{J''}C_{ij}$. More precisely, 
			\begin{equation}\label{eq:prop:equidistribution_f_adapted_time_outside_J''}
				\hat{m}^u_{\mathcal{W}}\left(\mathcal{R}_n^f \cap f^{-t}\left( \bigcup_{J''}C_{ij} \right)\right)
				> (1-2\e_{6}^{40})m\left( \bigcup_{J''}C_{ij} \right)
				\geq (1-2\e_{6}^{40})(1-3\e_6^{40})
				>1-5\e_6^{40}
			\end{equation}
	
	\textit{Step 2:} 
		In this step, we go from having equidistribution for product cubes $C_{ij}$ to equidistribution for most $f$-adapted parallelograms $\mathcal{B}_{k}^f$. Let 
			$$
				E':= (M\times N)\setminus \left( \bigcup_{J''}C_{ij} \right).
			$$
		The strategy is to show that for most $k\in J_f$, we can fill/cover all but a small portion of $\mathcal{B}_k^f$ by (good) product cubes $C_{ij}$ and then use equidistribution of $\mathcal{W}$ on each $C_{ij}$ to get equidistribution on $\mathcal{B}_k^f$.
		
		We begin by defining the set of ``good'' $f$-adapted parallelograms. Let 
			$
				J_f':= J_1\cap J_2,
			$
		where
			\begin{equation*}
				J_1:= \left\{ k\in J_f \colon m(\mathcal{B}_k^f\cap E') <\e_6^{20}m(\mathcal{B}_k^f)\right\}, \qquad 
				J_2:= \left\{ k\in J_f \colon \hat{m}_{\mathcal{W}}^u\left( \mathcal{R}_n^f \cap f^{-t}(\mathcal{B}_k^f\cap E') \right)< \e_6^{20}m(\mathcal{B}_k^f) \right\}.
			\end{equation*}
		We need to show that parallelograms in $J_f'$ cover most of $M\times N$. Applying Chebychev's inequality and using \eqref{eq:prop:equidistribution_f_adapted_J''_measure} and \eqref{eq:prop:equidistribution_f_adapted_time_outside_J''}, we get
			$$
				m\left( \bigcup_{J_f\setminus J_1} \mathcal{B}_k^f \right)
				\leq \e_6^{-20} m(E')<3\e_6^{20},
				\qquad
				m\left( \bigcup_{J_f\setminus J_2} \mathcal{B}_k^f \right)
				\leq \e_6^{-20} \hat{m}_{\mathcal{W}}^u\left( \mathcal{R}_n^f \cap f^{-t}(E') \right)<5\e_6^{20}.			
			$$
		Combining this with \eqref{eq:measure_cover_f_adapted_parallelograms} and the fact that $\e_6<\frac{1}{2}$ gives that
			\begin{equation*}
				m\left( \bigcup_{J_f'} \mathcal{B}_k^f \right)
				> 1-100\e_6^{\mathbf{b}/16}-8\e_6^{20}
				>1-\e_6^{10}
			\end{equation*}
		
		Fix $k\in J_f'$. We now fill/cover all but a small proportion of $\mathcal{B}_k^f$ by good product cubes (i.e. $C_{ij}$ with $(i,j)\in J''$). Define
			\begin{align*}
				J_k^- := \left\{ (i,j)\in J'' \colon C_{ij}\subset \mathcal{B}_k^f \right\}, \qquad %F_k^-:= \bigcup_{J_k^-} C_{ij}
				J_k^+ := \left\{ (i,j)\in J'' \colon C_{ij}\cap  \mathcal{B}_k^f\neq \emptyset \right\}.
			\end{align*}
		So we get that 
			\begin{equation*}
				F_k^- \subset \mathcal{B}_k^f \subset F_k^+ \cup \left( \mathcal{B}_k^f \cap E' \right), \qquad \text{where} \quad 
					F_k^-:= \bigcup_{J_k^-} C_{ij}, \quad 
					F_k^+:= \bigcup_{J_k^+} C_{ij}
			\end{equation*}
		We can therefore apply \eqref{eq:prop:equidistribution_f_adapted_replace_QxN} to each $(i,j)\in J_k^-\subset J''$ to get 
			\begin{equation*}
				\hat{m}^u_{\mathcal{W}}\left(\mathcal{R}_n^f \cap f^{-t}(\mathcal{B}_k^f)\right)
				\geq \hat{m}^u_{\mathcal{W}}\left(\mathcal{R}_n^f \cap f^{-t}(F_k^-)\right)
				> (1-2\e_6^{40}) m(F_k^-)
			\end{equation*}
		Similarly, applying \eqref{eq:prop:equidistribution_f_adapted_replace_QxN} for each $(i,j)\in J_k^+\subset J''$ and using the fact that $k\in J_f'\subset J_2$ gives that 
			\begin{equation*}
				\hat{m}^u_{\mathcal{W}}\left(\mathcal{R}_n^f \cap f^{-t}(\mathcal{B}_k^f)\right)
				\leq \hat{m}^u_{\mathcal{W}}\left(\mathcal{R}_n^f \cap f^{-t}(F_k^+)\right) 
					+\hat{m}^u_{\mathcal{W}}\left(\mathcal{R}_n^f \cap f^{-t}(\mathcal{B}_k^f\cap E')\right)
				< (1+2\e_6^{40})m(F_k^+)+\e_6^{20}m(\mathcal{B}_k^f)
			\end{equation*}
		Combining these bounds with the fact that $F_k^-\subset \mathcal{B}_k^f$ and that $F_k^{+}\subset (F^+_k\setminus \mathcal{B}_k^f) \cup \mathcal{B}_k^f$, we get that 
			\begin{equation}\label{eq:prop:equidistribution_f_adapted_bounds_filling_cubes}
				(1-2\e_6^{40})m(\mathcal{B}_k^f)-m(\mathcal{B}_k^f\setminus F_k^-) 
				<
				\hat{m}^u_{\mathcal{W}}\left(\mathcal{R}_n^f \cap f^{-t}(\mathcal{B}_k^f)\right)
				<
				(1+3\e_6^{20})m(\mathcal{B}_k^f)+(1+2\e_6^{40})m(F_k^+\setminus \mathcal{B}_k^f)
			\end{equation}
			
		We now show that $m(F_k^-\setminus \mathcal{B}_k^f)$ and $m(F_k^+\setminus \mathcal{B}_k^f)$ are small relative to $m(\mathcal{B}_k^f)$. We begin by bounding the diameters of the product cubes $C_{ij}$ using the following lemma.
		
		\begin{lemma}\label{lem:diameter_product_cubes}
			There exists $C_9>0$ such that if $\mathcal{B}^g(r_1,r_2)$ is a $g$-adapted parallelogram as defined in \cite[(5.1)]{ExpMixingBernoulli} with $0<r_2<r_1\leq \frac{\tau_g}{2}$ and $C(r_3)y\subset N$ is a cube in $N$ of radius $0<r_3\leq \frac{\delta_{inj}}{2}$, then 
				$
					\diam\left( \mathcal{B}^g(r_1,r_2)\times C(r_3)y \right)\leq C_9 \max\{r_1,r_3\}.
				$
		\end{lemma}
		\begin{proof}
			Combining Lemma 2.4, Lemma 2.6 and (5.1) in \cite{ExpMixingBernoulli}, we get a constant $C_9'>0$ such that for all $0<r_2<r_1\leq \frac{\tau_g}{2}$ $\diam(\mathcal{B}^g(r_1,r_2))\leq C_9'r_1$. Similarly, since the map $\theta$ that we used to define cubes in $N$ in Section \ref{sec:FiberCubes} is a diffeomorphism on $C_{SL_2\R}(\delta_{inj})$, we get a constant $C_9''>0$ such that for any $r_3\in \left(0,\frac{\delta_{inj}}{2}\right]$ and $y\in N$, $\diam(C(r_3)y)\leq C_9''r_3$. Combining these two bounds, we get Lemma \ref{lem:diameter_product_cubes}.
		\end{proof}
		
		Recall that for $(i,j)\in I_1\times I_2$, $C_{ij}=\mathcal{B}_i^g\times D_j$ where $\mathcal{B}_i^g=\mathcal{B}_i^g(r_1,r_2)$ with $r_1\leq e^{-(\eta_g-\e_{12})t}$ and $r_2\leq e^{-\eta_gt}$ and $D_j=C(r_3)y_j$ with $r_3=e^{-\eta t}$. Since 
			$
				t>T_0\geq \max\{ \frac{2}{\eta_g}\log(\frac{2}{\tau_g}), \frac{1}{\eta}\log(\frac{2}{\delta_{inj}})\}
			$
		and $\e_{12}<\frac{\eta_g}{2}$, we get that $r_2<r_1\leq \frac{\tau_g}{2}$ and $r_3\leq \frac{\delta_{inj}}{2}$. We can therefore apply Lemma \ref{lem:diameter_product_cubes} to get that 
			\begin{equation}\label{eq:prop:equidistribution_f_adapted_diameter_product_cube}
				\diam(C_{ij}) \leq C_9\max\{ e^{-(\eta_g-\e_{12})t}, e^{-\eta t} \}\leq C_9 e^{-\eta' t}, \qquad
				\eta':=\min\{ \eta_g-\e_{12},\eta \}.
			\end{equation}
		Let $V_{\delta}(\partial \mathcal{B}_k^f)$ be the $\delta$-neighborhood of $\partial \mathcal{B}_k^f$ in $M\times N$. By \eqref{eq:prop:equidistribution_f_adapted_diameter_product_cube}, if for $(i,j)\in I_1\times I_2$, $C_{ij}\cap V_{C_9e^{-\eta' t}}(\partial \mathcal{B}_k^f)=\emptyset$, then either $C_{ij}\subset \mathcal{B}_k^f$ or $C_{ij}\cap \mathcal{B}_k^f=\emptyset$. Thus 
			\begin{equation*}
				m(F_k^+\setminus \mathcal{B}_k^f) \leq m\left( V_{C_9e^{-\eta' t}}(\partial \mathcal{B}_k^f)  \right)
			\end{equation*}
		and since $k\in J_f'\subset J_1$, 
			\begin{equation*}
				m(\mathcal{B}_k^f\setminus F_k^-)
				\leq m(B_k^f\cap E')+m\left( V_{C_9e^{-\eta' t}}(\partial \mathcal{B}_k^f)  \right)
				<\e_6^{20}m(\mathcal{B}_k^f)+m\left( V_{C_9e^{-\eta' t}}(\partial \mathcal{B}_k^f)  \right)
			\end{equation*}
		Combining these two bounds with \eqref{eq:prop:equidistribution_f_adapted_bounds_filling_cubes} gives that
			\begin{equation}\label{eq:prop:equidistribution_f_adapted_bounds_neighborhood}
				(1-3\e_6^{20})m(\mathcal{B}_k^f)-m\left( V_{C_9e^{-\eta' t}}(\partial \mathcal{B}_k^f)  \right)
				< \hat{m}^u_{\mathcal{W}}\left(\mathcal{R}_n^f \cap f^{-t}(\mathcal{B}_k^f)\right)
				< (1+3\e_6^{20})m(\mathcal{B}_k^f)+(1+2\e_6^{40})m\left( V_{C_9e^{-\eta' t}}(\partial \mathcal{B}_k^f)  \right)
			\end{equation}
		
		We now bound $m\left( V_{C_9e^{-\eta' t}}(\partial \mathcal{B}_k^f)  \right)$ using the following lemma:
		
		\begin{lemma}\label{lem:bound_delta_nbhd_f_adapted_cube}
			There exist constants $C_5>0$, $\delta_7\in (0,\tau_f)$ and $C_{10}>0$ such that if $0<r_2<r_1\leq \frac{\delta_7}{2}$, $\mathcal{B}^f(r_1,r_2)$ is an $f$-adapted parallelogram as defined in \cite[(5.1)]{ExpMixingBernoulli}, and $0<\delta<C_{10}r_2$, then
				$$
					m\left( V_\delta(\partial \mathcal{B}^f(r_1,r_2)) \right)\leq C_5\frac{\delta}{r_2}m(\mathcal{B}^f(r_1,r_2)),
				$$
			where $V_\delta(\partial \mathcal{B}^f(r_1,r_2))$ is the $\delta$-neighborhood of $\partial \mathcal{B}^f(r_1,r_2)$ in $M\times N$. 
		\end{lemma}
		\begin{proof}
			Let $R(r_1,r_2)=[-r_1,r_1]^{d_u}\times [-r_2,r_2]^{d_{cs}}\subset \R^{u}\times \R^{cs}$, where $d_u$ is the number of positive Lyapunov exponents of $f$ (counted with multiplicity) and $d_{cs}=\dim(M\times N)-d_u$. Combining Lemma 2.4, Lemma 2.6 and (5.1) in \cite{ExpMixingBernoulli}, we get that there exists $\delta_7=\delta_7(\tau_f)\in (0,\tau_f)$ such that for any $r_1,r_2\in(0,\delta_7)$ and for any point $x\in \mathcal{P}_{\tau_f}^f$, there is a bi-Lipschitz function $H$ defined on $(-\delta_7,\delta_7)^{\dim(M\times N)}$ such that $H(R(r_1,r_2))=\mathcal{B}^f(r_1,r_2)$, where $\mathcal{B}^f(r_1,r_2)$ is the $f$-adapted parallelogram centered at $x$ defined in \cite[(5.1)]{ExpMixingBernoulli}. Note that the Lipschitz constant for $H$ and $H^{-1}$ can be taken to be uniform in $x\in \mathcal{P}_{\tau_f}^f$ and in $r_1,r_2\in (0,\delta_7)$. Let $L:=\max\{\Lip(H),\Lip(H^{-1})\}\geq 1$.
			
			Let $C_{10}=\frac{1}{L}$. Take $0<r_2\leq r_1\leq \frac{\delta_7}{2}$ and $0<\delta<C_{10}r_2=\frac{r_2}{L}$. Since $V_{L\delta}(\partial R(r_1,r_2))\subset (-\delta_7,\delta_7)^{\dim(M\times N)}$ and $H$ is bi-Lipschitz on $(-\delta_7,\delta_7)^{\dim(M\times N)}$ with Lipschitz constant $L$, we get that 
				\begin{equation}\label{eq:bound_delta_nbhd_f_adapted_cube1}
					m\left( V_\delta(\partial \mathcal{B}^f(r_1,r_2)) \right) 
					\leq L^{\dim(M\times N)}\vol\left( H^{-1}( V_\delta(\partial \mathcal{B}^f(r_1,r_2)) ) \right)
					\leq L^{\dim(M\times N)}\vol\left( V_{L\delta}(\partial R(r_1,r_2)) \right)
				\end{equation}
			and 
				\begin{equation}\label{eq:bound_delta_nbhd_f_adapted_cube2}
					\vol(R(r_1,r_2))\leq L^{\dim(M\times N)}m(\mathcal{B}^f(r_1,r_2))
				\end{equation}
			Now, we bound $\vol\left( V_{L\delta}(\partial R(r_1,r_2)) \right)$ in terms of $\vol(R(r_1,r_2))$. Since $L\delta<r_2=\min\{r_1,r_2\}$, 
				$$
					\vol(V_{L\delta}(\partial R(r_1,r_2))) %=2^{m+n}(r_1+{L\delta})^{d_u}(r_2+{L\delta})^{d_{cs}}-2^{m+n}(r_1-{L\delta})^{d_u}(r_2-{L\delta})^{d_{cs}}
					=2^{d_u+d_{cs}}r_1^{d_u}r_2^{d_{cs}}\left[ \left( 1+\frac{L\delta}{r_1} \right)^{d_u}\left( 1+\frac{L\delta}{r_2} \right)^{d_{cs}} -\left( 1-\frac{L\delta}{r_1} \right)^{d_u}\left( 1-\frac{L\delta}{r_2} \right)^{d_{cs}} \right]
				$$
			Since $0<\frac{L\delta}{r_1}, \frac{L\delta}{r_2}\leq1$, we can apply the Mean Value Theorem to the function $(a,b)\mapsto (1+a)^{d_u}(1+b)^{d_{cs}}$ along the segment from $(-\frac{L\delta}{r_1},-\frac{L\delta}{r_2})$ to $(\frac{L\delta}{r_1},\frac{L\delta}{r_2})$ to get that 
				$$
					\vol(V_{L\delta}(\partial R(r_1,r_2))) 
					\leq 2^{2d_u+2d_{cs}} r_1^{d_u}r_2^{d_{cs}} L\delta \left( \frac{d_u}{r_1}+\frac{d_{cs}}{r_2} \right)
					\leq 2^{2d_u+2d_{cs}+1}\max(d_u,d_{cs})\vol(R(r_1,r_2)) \frac{L\delta}{r_2} 
				$$
			where the last inequality uses the fact that $r_2\leq r_1$. Combining this with \eqref{eq:bound_delta_nbhd_f_adapted_cube1} and \eqref{eq:bound_delta_nbhd_f_adapted_cube2} gives that 
				$$
					m\left( V_\delta(\partial \mathcal{B}^f(r_1,r_2)) \right) 
					\leq (2L)^{2\dim(M\times N)+1}\max(d_u,d_{cs})\frac{\delta}{r_2}m(\mathcal{B}^f(r_1,r_2))
				$$
			Taking $C_5=(2L)^{2\dim(M\times N)+1}\max(d_u,d_{cs})$ finishes the proof. 
		\end{proof}
		
		We want to apply Lemma \ref{lem:bound_delta_nbhd_f_adapted_cube} to $m\left( V_{C_9e^{-\eta' t}}(\partial \mathcal{B}_k^f)  \right)$. Recall that $\mathcal{B}_k^f=\mathcal{B}_k^f(e^{-(\eta_f-\e_6)\e_6n},e^{-\eta_f\e_6 n})$. Since $t>T_0\geq \frac{1}{\eta_f-\e_6}\log(\frac{2}{\delta_7})$ and $\e_6n\geq t$, we get that $e^{-(\eta_f-\e_6)\e_6n}\leq \frac{\delta_7}{2}$. Since $t>T_0\geq \frac{1}{\eta'-\eta_f}\log(\frac{C_9e^{\eta_f}}{C_10})$ and $\e_6n<t+1$, we get that $C_9e^{-\eta't}<C_{10}e^{-\eta_f\e_6n}$. We can now apply Lemma \ref{lem:bound_delta_nbhd_f_adapted_cube} to get 
			\begin{equation*}
				m\left( V_{C_9e^{-\eta' t}}(\partial \mathcal{B}_k^f)  \right)
				\leq C_5C_9e^{-\eta' t+\eta_f\e_6n}m(\mathcal{B}_k^f)
				\leq C_5C_9e^{\eta_f}e^{-(\eta'-\eta_f)t}m(\mathcal{B}_k^f)
			\end{equation*}
		Since $t>T_0\geq \frac{1}{\eta'-\eta_f}\log(\frac{C_5C_9e^{\eta_f}}{\e_6^{20}})$, we get that 
			\begin{equation*}
				m\left( V_{C_9e^{-\eta' t}}(\partial \mathcal{B}_k^f)  \right)
				\leq \e_6^{20}m(\mathcal{B}_k^f)
			\end{equation*}
		Plugging this into \eqref{eq:prop:equidistribution_f_adapted_bounds_neighborhood}, gives 
			\begin{equation*}
				\hat{m}^u_{\mathcal{W}}\left(\mathcal{R}_n^f \cap f^{-t}(\mathcal{B}_k^f)\right) 
				\in \left( 1-4\e_6^{20},1+6\e_6^{20} \right) m(\mathcal{B}_k^f)
			\end{equation*}
		Since $\e_6\leq\frac{1}{2}$, $6\e_6^{20}\leq \e_6^{10}$, we have shown that \eqref{eq:prop:equidistribution_f_adapted2} holds.
	\end{proof}

\appendix
\section{Pesin Theory}\label{appendix:pesin_theory}
\subsection{Regularity of Osceledets subspaces}\label{appendix:pesin_regularity_osceledets}

	The goal of this section is to prove Proposition \ref{prop:E^sE^uHolderPesin}. The fact that the stable and unstable distributions for a H\"older cocycle are H\"older on Pesin sets is a standard fact from Pesin Theory \cite[Section 5.3]{BarreiraPesinNonUniform}. However, we can not directly apply those arguments to $\tilde{A}$ because $\tilde{A}$ need not even be continuous. We prove Proposition \ref{prop:E^sE^uHolderPesin} by proving an analogous result for the cocycle $\mathcal{A}$, which is H\"older, and then applying Corollary \ref{cor:stable_unstable_osceledets_subspaces_Ad}. Before doing this, we give the precise definition of the Pesin sets $\Lambda_{\e,\mathcal{A}}^\ell$ for $\mathcal{A}$:
	
	\begin{definition}\label{def:pesin_sets_Ad(A)}
		Fix $\e>0$ and $\ell\geq 1$. the \emph{Pesin set} $\Lambda_{\e,\mathcal{A}}^\ell$ is the set of regular points in $M$ such that  $\forall k\in \Z$ the following hold:
		\begin{enumerate}
		\item 
			If $v\in E^s_3(g^kx)$ and  $m\geq 0$,
				\begin{align*}
					\left\| \mathcal{A}_{m}(g^kx)v\right\|_F \leq \ell e^{-2\lambda_2 m+\e m+\e |k|} \|v\|_F
					\quad \text{and} \quad 
					\left\| \mathcal{A}_{-m}(g^kx)v\right\|_F \geq \ell^{-1} e^{2\lambda_2 m-\e m-\e |k-m|} \|v\|_F
				\end{align*}
		\item 
			If $v\in E^c_3(g^kx)$ and $m\in \Z$
				\begin{align*}
					\ell^{-1} e^{-\e |m|-\e |k+m|} \|v\|_F
					\leq \left\| \mathcal{A}_{m}(g^kx)v\right\|_F 
					\leq \ell e^{\e |m|+\e |k|} \|v\|_F
				\end{align*}
		\item 
			If $v\in E^u_3(g^kx)$ and $m\geq 0$
				\begin{align*}
					\left\| \mathcal{A}_{-m}(g^kx)v\right\|_F \leq \ell e^{-2\lambda_2 m+\e m+\e |k|} \|v\|_F
					\quad \text{and} \quad 
					\left\| \mathcal{A}_{m}(g^kx)v\right\|_F \geq \ell^{-1} e^{2\lambda_2 m-\e m-\e |k+m|} \|v\|_F
				\end{align*}
		\item If $i_1,i_2,i_3\in \{u,c,s\}$ with $i_1\neq i_2, i_2\neq i_3, i_1\neq i_3$, then 
			$$
				\angle\left( E^{i_1}_3(g^kx),E^{i_2}_3(g^kx)\oplus E^{i_3}_3(g^kx)\right) \geq \ell^{-1}e^{-\e |k|}
			$$
		\end{enumerate}
	\end{definition}

	\begin{proof}[Proof of Proposition \ref{prop:E^sE^uHolderPesin}]
		We first prove an analogue of Proposition \ref{prop:E^sE^uHolderPesin} for the distributions $E_3^u(x)$ and $E_3^s(x)$ for $\mathcal{A}$. The arguments used to prove \cite[Theorem A]{AraujoBufetovFilip} along with growth and angle estimates from Definition \ref{def:pesin_sets_Ad(A)} give that for $\sigma\in \{s,u\}$ and $x,y\in \Lambda_\e^\ell$ with $d(x,y)\leq 1$, 
			\begin{equation*}
				\dist\left( E^\sigma_3(x),E^\sigma_3(y) \right) \leq 64\ell^6 e^{2\lambda_2-4\e}d(x,y)^{\frac{\alpha(2\lambda_2-4\e)}{\log(a)+2\lambda_2-\e}}
			\end{equation*}
		where $a\geq 1$ is a constant that depends on $A$ and $g$, but not on $\e,\ell, x, $ or $y$. Since $\e\in (0,\frac{\lambda_2}{4})$, 
			$$
				\frac{\alpha(2\lambda_2-4\e)}{\log(\alpha)+2\lambda_2-\e}\geq \beta:= \frac{\alpha\lambda_2}{\log(a)+2\lambda_2}. 
			$$
		Combining this with Proposition \ref{prop:PesinSetsFacts}(2) gives that for $k\in \Z$, 
			\begin{equation}\label{eq:AdE^sE^uHolderPesin}
				\dist\left( E^\sigma_3(g^kx),E^\sigma_3(g^ky) \right)
				\leq C_{\e,\ell}' e^{6\e|k|}d(g^kx,g^ky)^\beta,
				\qquad C_{\e,\ell}' = 64 \ell^6 e^{2\lambda_2-4\e}
			\end{equation}
		Now, we prove that these regularity estimates for $E_3^s$ and $E_3^u$ imply Proposition \ref{prop:E^sE^uHolderPesin}. Take $X_x\in E^\sigma_3(g^kx)$ and $X_y\in E^\sigma_3(g^ky)$ with $\|X_x\|_F=\|X_y\|_F=1$ and $\angle(X_x,X_y)\in [0,\frac{\pi}{2}]$. By Corollary \ref{cor:stable_unstable_osceledets_subspaces_Ad}
			$$
				\ker(X_z)=\im(X_z)=E^\sigma_2(g^kz), \qquad z\in \{x,y\}.
			$$
		We can therefore write orthogonal projection $P_z:\R^2\to \R^2$ onto $E_2^\sigma(g^kz)$ as $P_z=\id-X_z^TX_z$. From \eqref{eq:angle_distance_subspaces}, we have that 
			$$
				\dist\left( E^\sigma_2(g^kx), E^\sigma_2(g^ky) \right) 
				= \left\| P_x-P_y\right\|_{op}
				\leq 2\left\|X_y-X_x\right\|_{op}
				\leq 2\sqrt{2}\sin\left( \angle(X_y,X_x) \right)
				=2\sqrt2 \dist(E^\sigma_3(g^kx),E^\sigma_3(g^ky))
			$$
		Combining this with \eqref{eq:AdE^sE^uHolderPesin} gives Proposition \ref{prop:E^sE^uHolderPesin}.
	\end{proof}
	
\subsection{Proof of the Osceledets-Pesin Reduction}\label{appendix:proof_osceledets_pesin_reduction}

	\begin{proof}[Proof of Theorem \ref{thm:Osceledets_Pesin_Reduction}]
		Let $v^u(x)\in E_2^u(x)$ and $v^s(x)\in E_2^s(x)$ be unit vectors such that $\det(v^u(x),v^s(x))=\sin\theta_x$. Define the matrix $\tilde P(x)\in SL_2\R$ to be the change of basis matrix taking the standard orthonormal basis  of $\R^2$ to the basis $\left\{ \dfrac{v^u(x)}{(\sin\theta_x)^{1/2}}, \dfrac{v^s(x)}{(\sin\theta_x)^{1/2}}  \right\}$. The columns of $\tilde P(x)$ are the images of the standard basis vectors. So if we write $v^u(x)$ and $v^s(x)$ in terms of the standard basis, then 
			$$
				\tilde{P}(x)=(\sin\theta_x)^{-1/2}\begin{pmatrix} v^u(x) & v^s(x) \end{pmatrix}
			$$
		Since $\det(\tilde P(x))=1$, $\tilde{P}(x)\in SL_2\R$. Since $E_2^u$ and $E_2^s$ are $\tilde{A}$-invariant, 
			$$
				\tilde{B}(x)=\tilde P(gx)^{-1}\tilde A(x)\tilde P(x)
			$$
		is diagonal. Since $\tilde{B}(x)\in SL_2\R$, $\tilde{B}(x)=\diag(b(x),b(x)^{-1})$. 
	
		\begin{claim}\label{claim:osceledets-pesin_norm_P(x)}
			$\|\tilde P(x)\|_{2,op}=\|\tilde P(x)^{-1}\|_{2,op}=\left( \tan(\frac{\theta_x}{2}) \right)^{-1/2}$
		\end{claim}
		\begin{proof}
			We compute the singular values of $\tilde P(x)$. Since $v^u(x)$ and $v^s(x)$ are unit vectors with $\det(v^u(x),v^s(x))=\sin\theta_x$ (and $\theta_x\in \left(0,\frac{\pi}{2}\right]$), we get that $|v^u(x)\cdot v^s(x)|=\cos \theta_x$. This along with the fact that 
				$
					\tilde{P}(x)=(\sin\theta_x)^{-1/2}\begin{pmatrix} v^u(x) & v^s(x) \end{pmatrix},
				$
			gives that
				$$
					\tilde{P}(x)^T\tilde{P}(x)
					=\frac{1}{\sin \theta_x}\begin{pmatrix} 1 & c\\ c& 1 \end{pmatrix}, 
					\qquad c:=v^u(x)\cdot v^s(x).
				$$
			The eigenvalues of $\tilde{P}(x)^T\tilde{P}(x)$ are $\frac{1+c}{\sin\theta_x}$ and $\frac{1-c}{\sin\theta_x}$. Since $|c|=\cos\theta_x$, we get that the largest singular value of $\tilde{P}(x)$ is 
				$
					\left( \frac{1+\cos\theta_x}{\sin\theta_x} \right)^{1/2}=\left(\tan(\frac{\theta_x}{2})\right)^{-1/2}.
				$
		\end{proof}
		
		Claim \ref{claim:osceledets-pesin_norm_P(x)} along with the half angle formula for tangent gives that 
			$$
				\log \|\tilde{P}(g^kx)\|=\frac{1}{2}\log(1+\cos\theta_{g^kx})-\frac{1}{2}\log(\sin\theta_{g^kx})
			$$
		Since the first term is bounded and the second term is sub-linear by Osceledets Theorem (see section \ref{sec:osceledets_A}), we get that $\tilde{P}$ is tempered. Since $\tilde{P}$ is tempered, we get that $\tilde{B}$ has the same Lyapunov exponents as $\tilde A$. 
		
		Now we see that $\log|b|\in L^1(\mu)$. We can write 
			\begin{equation}\label{eq:b(x)_formula}
				b(x)=a(x)\left( \frac{\sin\theta_{gx}}{\sin\theta_x} \right)^{1/2},
			\end{equation}
		where $a(x)\in \R$ is defined by $\tilde{A}(x)v^u(x)=a(x)v^u(gx)$. This along with the fact that 
			\begin{equation}\label{eq:osceledets_bound_ratio_angles_norm}
				\|\tilde A(x)\|^{-2}
				\leq \frac{\sin|\angle(\tilde A(x)v^u(x),\tilde A(x)v^s(x))|}{\sin|\angle(v^u(x),v^s(x))|}
				= \frac{\sin \theta_{gx}}{\sin\theta_x}
				\leq \|\tilde A(x)\|^2
			\end{equation}
		allows us to conclude that $\log|b|\in L^1(\mu)$ from the fact that $\log\|A(x)^{\pm 1}\|\in L^1(\mu)$.
	\end{proof}
	
\subsection{Proofs of Regularity Estimates for the Osceledets-Pesin Reduction} \label{appendix:proofs_regularity_osceledets_pesin_reduction}

\subsubsection{Proof of Proposition \ref{prop:tau_holder_on_pesin}} \label{appendix:proof_tau_holder}
	
	Proposition \ref{prop:tau_holder_on_pesin} is a corollary of the following proposition:
	
	\begin{prop}\label{prop:b_holder_on_pesin}
		There exists $\alpha_b\in (0,1)$ such that for all $\e\in (0,\frac{\lambda_2}{4})$ and $\ell\geq 1$, there exists a constant $C_b=C_b(\e,\ell)>0$ such that for all $k\in \Z$, if $x,y\in \Lambda_\e^\ell$ with $d(g^kx,g^ky)\leq 1$, then 
			\begin{equation}\label{eq:b_holder_on_pesin}
				\left||b(g^kx)|-|b(g^ky)|\right|\leq C_b e^{6\e|k|} d(g^kx,g^ky)^{\alpha_b}
			\end{equation}
	\end{prop}
	\begin{proof}
		Since $\|v^u(gz)\|=1$, \eqref{eq:b(x)_formula} gives that if $z$ is a regular point for $A$, then 
			\begin{equation}\label{eq:|b(x)|}
				\left|b(z)\right|=b_1(z)b_2(z), 
				\qquad b_1(z):=\left( \frac{\sin \theta_{gz}}{\sin \theta_{z}} \right)^{1/2},
				\qquad b_2(z):= \|\tilde A(z) v^u(z)\|
			\end{equation}
		We show H\"older regularity of $|b|$ by that both $b_1$ and $b_2$ are H\"older and uniformly bounded. 
		
		First we show that $b_1$ is H\"older on Pesin sets. Combining Definition \ref{def:PesinSets}(3) with the fact that $\frac{t}{2}\leq \sin(t)\leq 1$ for $t\in [0,\frac{\pi}{2}]$ gives that for $x,y\in \Lambda_\e^\ell$, 
			\begin{equation*}
				\left| \frac{\sin\theta_{g^{k+1}x}}{\sin\theta_{g^kx}} - \frac{\sin\theta_{g^{k+1}y}}{\sin\theta_{g^ky}} \right|
				\leq 
				4\ell^{2}e^{2\e|k|}\left( \left|\sin \theta_{g^ky} - \sin \theta_{g^k x} \right| + \left| \sin \theta_{g^{k+1}x} - \sin \theta_{g^{k+1}y} \right|   \right)
			\end{equation*}
		Since $\sin(t)$ is Lipschitz with Lipschitz constant 1 and $g$ is bi-Lipschitz with Lipschitz constant $L_g$, applying Corollary \ref{cor:theta_x_holder} twice at times $k$ and $k+1$ gives that 
			\begin{equation*}
				\left| \frac{\sin\theta_{g^{k+1}x}}{\sin\theta_{g^kx}} - \frac{\sin\theta_{g^{k+1}y}}{\sin\theta_{g^ky}} \right|
				\leq 16C_{\e,\ell}\ell^2 e^{6\e} e^{8\e|k|}(1+L_g^\beta)d(g^kx,g^ky)^{\beta}
			\end{equation*}
		Since $|\sqrt{t_1}-\sqrt{t_2}|\leq |t_1-t_2|^{1/2}$ for all $t_1,t_2\geq 0$, we get that 
			\begin{equation}\label{eq:ratio_angles_Holder}
				|b_1(g^kx)-b_1(g^ky)|\leq 4C_{\e,\ell}^{1/2}\ell e^{3\e}(1+L_g^\beta)^{1/2}e^{4\e|k|}d(g^kx,g^ky)^{\beta/2}
			\end{equation}
		Since $\|\tilde A(z)\|$ is uniformly bounded on $M$, \eqref{eq:osceledets_bound_ratio_angles_norm} gives that $b_1$ is uniformly bounded. The fact that $\|\tilde A(z)\|$ is uniformly bounded on $M$ also gives that $b_2$ is uniformly bounded. 
		
		We now just need to show that $b_2$ is H\"older on Pesin sets. Let $\phi:PSL_2\R\times \R\mathbb{P}^1\to \R$ be the smooth function
			$
				\phi(A,E)= \|\tilde A v\|,
			$
		where $\tilde{A}\in SL_2\R$ is a lift of $A$ and $v\in E$ is a unit vector spanning $E$. Note that $\phi$ is well defined since $\|\tilde A v\|$ is independent of the choice of lift $\tilde{A}$ and of choice of $v\in E$. Then, $b_2(z)=\phi(A(z),E^u_2(z))$. Since $A(M)\times \mathbb{RP}^1$ is compact, $\phi$ is Lipschitz there, so 
			$$
				|b_2(g^kx)-b_2(g^ky)|\leq C\left( d(A(g^kx),A(g^ky))+\dist(E^u_2(g^kx),E^u_2(g^ky)) \right)
			$$
		Then H\"older continuity of $A$ along with Proposition \ref{prop:E^sE^uHolderPesin} gives that for some $C'=C'(\e,\ell)$
			\begin{equation}\label{eq:b_holder_norm_Av^u}
				|b_2(g^kx)-b_2(g^ky)|\leq C'e^{6\e|k|}d(g^kx,g^ky)^{\min(\alpha,\beta)}
			\end{equation}
		Combining \eqref{eq:|b(x)|}, \eqref{eq:ratio_angles_Holder}, and \eqref{eq:b_holder_norm_Av^u} with the uniform boundedness of $b_1$ and $b_2$ gives \eqref{eq:b_holder_on_pesin} with $\alpha_b=\min\{\alpha,\frac{\beta}{2}\}$.
	\end{proof}
	
	\begin{proof}[Proof of Proposition \ref{prop:tau_holder_on_pesin}]
		By definition of the Pesin set, for $z\in \bar{\Lambda}_\e^\ell$, $|b(g^kz)|\geq \ell^{-1}e^{\lambda_2-\e|k|}$. So applying the Mean Value Theorem for $\log$ and Proposition \ref{prop:b_holder_on_pesin} gives Proposition \ref{prop:tau_holder_on_pesin}. 
	\end{proof}

\subsubsection{Proofs of H\"older Estimates for $P(x)$}\label{appendix:proofs_holder_P(x)}
	
	\begin{proof}[Proof of Proposition \ref{prop:P_holder_on_pesin-norm}]
		Recall from the proof of Theorem \ref{thm:Osceledets_Pesin_Reduction} that we choose unit vectors $v^u(z)\in E^u_2(z)$ and $v^s(z)\in E^s_2(z)$ such that $\det(v^u(z),v^s(z))=\sin\theta_z$. Choose $\tilde{v}^u(g^ky)\in E^u_2(g^ky)$ to be a unit vector such that $|\angle(v^u(g^kx),\tilde{v}^u(g^ky))|\leq \frac{\pi}{2}$. Choose $\tilde{v}^s(g^ky)\in E^s_2(g^ky)$ to be the unique unit vector such that $\det(\tilde{v}^u(g^ky),\tilde{v}^s(g^ky))=\sin\theta_{g^ky}>0$. Then
			$$
				\bar{P}(g^ky):=(\sin\theta_{g^ky})^{-1/2}\begin{pmatrix} \tilde{v}^u(g^ky) & \tilde{v}^s(g^ky)\end{pmatrix}.
			$$
		equals either $\tilde{P}(g^ky)$ or $-\tilde{P}(g^ky)$. 
		
		For $\sigma\in\{s,u\}$, 
			\begin{equation}\label{eq:prop_P_holder_norm1}
				\left\|\frac{v^\sigma(g^kx)}{\sqrt{\sin \theta_{g^kx}}} - \frac{\tilde v^\sigma(g^ky)}{\sqrt{\sin \theta_{g^ky}}} \right\|
				\leq \left|\frac{1}{\sqrt{\sin \theta_{g^kx}}}-\frac{1}{\sqrt{\sin \theta_{g^ky}}}\right| 
					+\frac{1}{\sqrt{\sin \theta_{g^ky}}} \left\|v^\sigma(g^kx)-\tilde{v}^\sigma(g^ky)\right\|
			\end{equation}
		Proposition \ref{prop:E^sE^uHolderPesin} implies that for some constant $C_1=C_1(\e,\ell)>0$, 
			\begin{equation}\label{eq:prop_P_holder_norm2}
				\left\|v^\sigma(g^kx)-\tilde{v}^\sigma(g^ky) \right\| \leq C_1e^{6\e|k|}d(g^kx,g^ky)^\beta
			\end{equation}
		Since $y\in \Lambda_\e^\ell$ and $\sin(t)\geq \frac{t}{2}$ for $t\in[0,\frac{\pi}{2}]$, we get that 
			\begin{equation}\label{eq:prop_P_holder_norm3}
				(\sin \theta_{g^ky})^{-1/2}\leq \sqrt{2}\ell^{1/2}e^{\e|k|/2}.
			\end{equation}
		For $z\in \Lambda_\e^\ell$, $\theta_{g^kz}\geq \ell^{-1}e^{-\e|k|}$. So, combining Corollary \ref{cor:theta_x_holder} with the Mean Value Theorem gives that, for some constant $C_2=C_2(\e,\ell)>0$,
			\begin{equation}\label{eq:prop_P_holder_norm4}
				\left|(\sin\theta_{g^kx})^{-1/2}-(\sin\theta_{g^ky})^{-1/2}\right| 
				\leq C_2 e^{\frac{15}{2}\e|k|}d(g^kx,g^ky)^{\beta}
			\end{equation}
		Plugging \eqref{eq:prop_P_holder_norm2}, \eqref{eq:prop_P_holder_norm3}, and \eqref{eq:prop_P_holder_norm4} into \eqref{eq:prop_P_holder_norm1} gives that for some constant $C_3=C_3(\e,\ell)>0$, 
			$$
				\left\|\frac{v^\sigma(g^kx)}{\sqrt{\sin \theta_{g^kx}}} - \frac{\tilde v^\sigma(g^ky)}{\sqrt{\sin \theta_{g^ky}}} \right\|
				\leq C_3e^{\frac{15}{2}\e|k|}d(g^kx,g^ky)^\beta
			$$
		Applying this to the two columns of $\tilde{P}(g^kx)-\bar{P}(g^ky)$ gives that for some constant $C_{P,1}=C_{P,1}(\e,\ell)>0$, 
			$$
				\|\tilde{P}(g^kx)-\bar{P}(g^ky)\| 
				\leq C_{P,1}e^{\frac{15}{2}\e|k|}d(g^kx,g^ky)^\beta
				\leq C_{P,1}e^{8\e|k|}d(g^kx,g^ky)^\beta
			$$
		Since $\bar{P}(g^ky)=\pm\tilde{P}(g^ky)$, we get Proposition \ref{prop:P_holder_on_pesin-norm}.
	\end{proof}
	
	\begin{proof}[Proof of Proposition \ref{prop:P_holder_on_pesin_unstables}]
		Taking $\bar{P}(g^ky)$ as in the proof of Proposition \ref{prop:P_holder_on_pesin-norm}, we have that
			\begin{equation}\label{eq:prop:P_holder_on_pesin_unstables1}
				d_{PSL_2\R}\left( P(g^{-k}x),P(g^{-k}y) \right) 
				\leq 
				d_{SL_2\R}\left(e,\bar P(g^{-k}y)  \tilde P(g^{-k}x)^{-1} \right)
			\end{equation}
	
		Since $x\in \Lambda_\e^\ell$, Claim \ref{claim:osceledets-pesin_norm_P(x)} along with the fact that $\sin(t)\geq \frac{t}{2}$ for $t\in[0,\frac{\pi}{2}]$ gives that 
			\begin{equation}\label{eq:bound_norm_P(g^-kx)}
				\|\tilde{P}(g^{-k}x)^{-1}\|
				\leq 2\ell^{1/2}e^{\e k/2}\leq 2\ell^{1/2}e^{\e k}
			\end{equation}
		Combining this with Proposition \ref{prop:P_holder_on_pesin-norm} and \eqref{eq:g_unstable_exp_contracting2} gives that 
			\begin{equation}\label{eq:prop:P_holder_on_pesin_unstables2}
				\|e-\bar P(g^{-k}y)  \tilde P(g^{-k}x)^{-1}\|
				\leq 2\ell^{1/2}C_{P,1}\tau_g^{-\beta}e^{(9\e-\frac{\lambda_g\beta}{2})k}d(x,y)^{\beta}
			\end{equation}
		Take $U$ as in Lemma \ref{lem:SL2R_difference_equivalent_to_Riem_dist}. Since $\e<\frac{\lambda_g\beta}{18}$, by \eqref{eq:prop:P_holder_on_pesin_unstables2}, we can find $r_p>0$ (which is independent of both $x$ and $k$) such that if $d(x,y)\leq r_p$, then $\bar P(g^{-k}y)  \tilde P(g^{-k}x)^{-1}\in U$ for all $k\geq 0$. Thus, combining \eqref{eq:prop:P_holder_on_pesin_unstables1} and \eqref{eq:prop:P_holder_on_pesin_unstables2} with Lemma \ref{lem:SL2R_difference_equivalent_to_Riem_dist} gives that 
			$$
				d_{PSL_2\R}\left( P(g^{-k}x),P(g^{-k}y) \right) 
				\leq 2K\ell^{1/2}C_{P,1}\tau_g^{-\beta}e^{(9\e-\frac{\lambda_g\beta}{2})k}d(x,y)^{\beta}
			$$
	\end{proof}

\section{Smoothing Lemma}

	The purpose of this appendix is to prove the smoothing lemma that we use in the proof of Proposition \ref{prop:exp_equidistribution_fiber}. Recall that $\chi_{C(\e)}$ denotes the indicator function of the cube of radius $\e$ in $SL_2\R$. 
	
	\begin{lemma}\label{lem:smooth_indicator_fcn_cube}
		Take $\e,\delta\in(0,1)$ such that $\e(1+\delta)<\delta_0$. Then there exist functions $\phi^+, \phi^-\in C^\infty_c(SL_2\R)$ such that 
			\begin{enumerate}[(1)]
			\item for $\sigma \in \{+,-\}$, $\phi^\sigma \equiv 0$ outside $C(\e(1+\delta))$.
			\item $\phi^-\leq \chi_{C(\e)}\leq \phi^+$
			\item for $\sigma\in \{+,-\}$, 
				\begin{equation*}
					\int_{SL_2\R} \left| \phi^\sigma-\chi_{C(\e)} \right|d\nu \leq \frac{K_4}{C_1}\delta \nu(C(\e))
				\end{equation*}
			\item There exists a constant $C_7'\geq 1$ such that for $\sigma\in \{+,-\}$, 
				\begin{equation*}
					\left\| \phi^\sigma \right\|_{C^4}\leq C_7'(\delta\e)^{-4}
				\end{equation*}
			\end{enumerate}
	\end{lemma}
	
	Here $\delta_0$ is chosen as in Section \ref{sec:FiberCubes} so that $C(r)$ is defined for every $r\in (0,\delta_0)$.
	
	\begin{proof}
		For $r>0$, let $Q_r$ denote the Euclidean cube $C_{\sl_2\R}(r)$. Choose $\rho\in C^\infty_c(\sl_2\R)$ with $\rho\geq 0$,  $\supp(\rho)\subset Q_1$, and $\int \rho=1$. For $a>0$, we define 
			$$
				\rho_a(x):=a^{-3}\rho\left( \frac{x}{a} \right).
			$$
		Observe that $\supp(\rho_a)\subset Q_a$, $\int \rho_a=1$, and for every multi-index $\alpha$ with $|\alpha|\leq 4$, 
			\begin{equation}\label{eq:smoothing_lemma_pf1}
				\left\|D^\alpha \rho_a \right\|_{L^1}=a^{-|\alpha|}\left\|D^{\alpha}\rho\right\|_{L^1}.
			\end{equation}
		Fix $a=\frac{\e\delta}{2}$, and let 
			\begin{equation*}
				\overline{\phi}^+:=\chi_{Q_{\e+a}} * \rho_a
				\qquad \text{ and } \qquad 
				\overline{\phi}^-:=\chi_{Q_{\e-a}} * \rho_a
			\end{equation*}
		Since $\rho_a\geq 0$ and $\int \rho_a=1$, both functions take values in $[0,1]$. Moreover, since $\supp(\rho_a)\subset Q_a$ and $\int \rho_a=1$, we get that 
			\begin{equation*}
				\bar{\phi}^+\equiv 1 \text{ on } Q_\e, 
				\quad 
				\supp(\bar{\phi}^+)\subset Q_{\e+2a}=Q_{\e(1+\delta)},
				\quad 
				\bar{\phi}^-\equiv 1 \text{ on } Q_{\e-2a}=Q_{\e(1-\delta)}, 
				\quad 
				\supp(\bar{\phi}^-)\subset Q_{\e}.
			\end{equation*}
		Define $\phi^{\pm}:SL_2\R\to [0,1]$ by 
			\begin{equation*}
				\phi^{\pm} := \begin{cases}
					\bar{\phi}^{\pm}\circ \theta^{-1}(x), & \qquad x\in C(\delta_0), \\
					0, & \qquad x\notin C(\delta_0)
				\end{cases}
			\end{equation*}
		where $\theta: Q_{\delta_0}\to C(\delta_0)$ is from Section \ref{sec:FiberCubes}. We immediately get that $\phi^{\pm}\in C^\infty_c(SL_2\R)$ and properties (1) and (2). We now show property (3). Observe that 
			$$
				0\leq \phi^+-\chi_{C(\e)}\leq \chi_{C(\e(1+\delta))}-\chi_{C(\e)}.
			$$
		Thus, Proposition \ref{prop:upper_bound_difference_measure_cubes_SL2R} along with \eqref{eq:lower_bound_measure_cubes_SL2R} give that
			$$
				\int_{SL_2\R}|\phi^+ - \chi_{C(\e)}| 
				\leq K_4\e^3((1+\delta)^3-1) 
				\leq 7K_4\e^3\delta 
				\leq \frac{K_4}{C_1}\delta \nu(C(\e)).
			$$
		Similarly, 
			$$
				0\leq \chi_{C(\e)}-\phi^-\leq \chi_{C(\e)}-\chi_{C(\e(1-\delta))}.
			$$
		As before we get 
			$$
				\int_{SL_2\R}|\phi^- - \chi_{C(\e)}| 
				\leq K_4\e^3(1-(1-\delta)^3) 
				\leq 4K_4\e^3\delta 
				\leq \frac{K_4}{C_1}\delta \nu(C(\e)).
			$$			
		We've now shown that property (3) holds. Now to show property (4) holds, observe that we can find a constant $C'>0$ such that for all multi-indices $\alpha$ with $|\alpha|\leq 4$, 
			$$
				\|D^\alpha \rho\|_{L^1}\leq C'.
			$$
		Combining this with the fact that
			$$
				D^\alpha \bar{\phi}^{\pm} =\chi_{Q_{\e\pm a}} * D^\alpha \rho_a,
			$$
		and with \eqref{eq:smoothing_lemma_pf1}, we get that
			$$
				\|D^\alpha \bar{\phi}^{\pm} \|_\infty
				\leq \|D^\alpha \rho_a\|_{L^1}
				\leq C' \left( \frac{\e\delta}{2} \right)^{-|\alpha|}.
			$$
		Since $\theta^{-1}$ is smooth on $\overline{C(\delta_0)}$, all its derivatives through order 4 are uniformly bounded on $\overline{C(\delta_0)}$. The chain rule then gives that for some constant $C''>0$,
			$$
				\|D^j \phi^{\pm}\|_{\infty} \leq C'' (\e\delta)^{-j}, \qquad 0\leq j\leq 4.
			$$
		Then, since $\e\delta<1$, we get that property (4) holds. 
	\end{proof}
	
	\begin{cor}\label{cor:smoothing_lemma_N}
		Take $\e,\delta\in(0,1)$ such that $\e(1+\delta)<\delta_{inj}$. Then there exist functions $\phi^+, \phi^-\in C^\infty_c(N)$ such that for $y\in N$,
			\begin{enumerate}[(1)]
			\item for $\sigma \in \{+,-\}$, $\phi^\sigma \equiv 0$ outside $C(\e(1+\delta))y$.
			\item $\phi^-\leq \chi_{C(\e)y}\leq \phi^+$
			\item for $\sigma\in \{+,-\}$, 
				\begin{equation*}
					\int_{N} \left| \phi^\sigma-\chi_{C(\e)y} \right|d\nu \leq \frac{K_4}{C_1}\delta \nu(C(\e)y)
				\end{equation*}
			\item There exists a constant $C_7'\geq 1$ such that for $\sigma\in \{+,-\}$, 
				\begin{equation*}
					\left\| \phi^\sigma \right\|_{C^4}\leq C_7'(\delta\e)^{-4}
				\end{equation*}
			\end{enumerate}
	\end{cor}

	\bibliographystyle{amsplain}
	\bibliography{BernoulliSkewExpMixingBase}

\end{document}